\documentclass[12pt]{amsart}
\usepackage{amsmath,amsthm,amssymb}
\usepackage{wasysym}
\usepackage{chngcntr}
\usepackage{graphicx}
\usepackage{cite,float}
\usepackage{enumerate,amsmath,amsthm,amssymb,cite,mathrsfs,eufrak,graphicx,subfigure}

\usepackage{tabularx}
\usepackage{tabulary}
\usepackage{ctable}
\newcolumntype{Q}{>{$\displaystyle}l<{$}}
\newcolumntype{A}{>{$}c<{$}}

\newcommand{\lemref}[1]{Lemma~\ref{#1}}

\newcommand{\disc}{\mathbb{D}}
\newcommand{\Snl}{\mathcal{S}^*_{n\mathcal{L}}}
\DeclareMathOperator{\RE}{Re}

\numberwithin{equation}{section}
\newtheorem{theorem}{Theorem}[section]
\newtheorem{lemma}[theorem]{Lemma}

\theoremstyle{remark}
\newtheorem{remark}[theorem]{Remark}

\makeatletter
\@namedef{subjclassname@2020}{%
	\textup{2020} Mathematics Subject Classification}
\makeatother

\begin{document}
 	\title{Geometric properties of a domain with cusps}
 \author[S. Gandhi]{Shweta Gandhi}

 \address{Department of Mathematics, Miranda House, University of Delhi,
 	Delhi--110 007, India}
 \email{gandhishwetagandhi@gmail.com}
 	
 \author[P. Gupta]{Prachi Gupta}
 \address{Department of Mathematics, University of Delhi, Delhi--110 007, India}
 \email{prachigupta161@gmail.com}

 \author[S. Nagpal]{Sumit Nagpal}
 \address{Department of Mathematics, Ramanujan College, University of Delhi,
 	Delhi--110 019, India}
 \email{sumitnagpal.du@gmail.com }

 \author[V. Ravichandran]{V. Ravichandran}

 \address{Department of Mathematics,
 	National Institute of Technology,
 	Tiruchirappalli-620015 }
 \email{vravi68@gmail.com}

\begin{abstract}
For $n\geq 4$ (even), the function $\varphi_{n\mathcal{L}}(z)=1+nz/(n+1)+z^n/(n+1)$ maps the unit disk $\disc$ onto a domain bounded by an epicycloid with $n-1$ cusps. In this paper, the class $\Snl=\mathcal{S}^*(\varphi_{n\mathcal{L}})$ is studied and various inclusion relations are established with other subclasses of starlike functions. The bounds on initial coefficients is also computed. Various radii problems are also solved for the class $\Snl.$
\end{abstract}
 	\maketitle
 	\section{Introduction}
 	 An Epicycloid\cite{LAW} is a plane curve produced by tracing the path of a chosen point on the circumference of a circle of radius $b$ which rolls without slipping around a fixed circle of radius $a$. The parametric equation of an epicycloid is
 	\begin{align*}
 		 x(t)&=m\cos t- b\cos\left(\frac{mt}{b}\right),\\	y(t)&=m\sin t- b\sin\left(\frac{mt}{b}\right),\qquad-\pi\leq t\leq \pi,
 	\end{align*}
where $m=a+b.$ If $m/b$ is an integer, then the curve has $m/b-1$ cusps. Some of the epicycloid have  special names. For $a=b,$ the curve obtained is called a cardiod and has one cusp; for $a=2b$ it is a nephroid with two cusps and for $a=5b$, the curve formed is called ranunculoid, a five-cusped epicycloid.
 	 A parametric curve $\left(f(t),g(t)\right)$ has a cusp \cite{HAGEN} at the point $\left(f(t_0),g(t_0)\right)$ if $f'(t_0)$ and $g'(t_0)$ is zero but either $f''(t_0)$ or $g''(t_0)$ is not equal to zero. Many curves have been widely studied having no cusp, one cusp, two cusps and three cusps. For instance, the boundary of image domains of the functions  $e^z$, $1+\sin z$ and $2/(1+e^{-z})$\cite{CHO1,GOEL,MEND1}, under unit disk, have no cusp. The Lemniscate of Bernoulli $\sqrt{1+z}$, the reverse Lemniscate $\varphi_{RL}(z)$ and cardiod type domain (see\cite{GUPTA,SUSHIL,MEND,KANSH,SIVA,SOKOL}) contains one cusp on the real axis. Nephroid \cite{WANI2} has two cusps on real axis whereas lune\cite{RAINA} and petal-like domain \cite{SIVA1} contains two cusps at the angle $\pi/2$ and $3\pi/2.$ Gandhi \cite{GANDHI} studied the class of functions for which boundary of the image domain contains three cusps, one on real axis and two at the angles $\pi/3$ and $5\pi/3.$ Motivated by this work, we have considered a more general domain whose boundary has the following parametric form:
 	\begin{align}\label{eqn0}
 		\begin{split}
 					x(t)&=1+\frac{n}{n+1}\cos t+\frac{1}{n+1}\cos(nt),\\
 			y(t)&=\frac{n}{n+1}\sin t+\frac{1}{n+1}\sin(nt),
 		\end{split}
 	\end{align}
 for $n\geq 4$ (even). For $a=(n-1)/(n+1)$ and $b=1/(n-1),$ the curve (\ref{eqn0}) represents a rotated and translated epicycloid\cite{MADACHY} with $(n-1)$ cusps. It is an algebraic curve of order $2n$. It can be easily seen that $x'(t_k)=0$ and $y'(t_k)=0$ for $t_k=(2k-1)\pi/(n-1),$ where $k=1,2,\ldots (n-2)/2.$ Also, $x''(t_k)$ and $y''(t_k)$ are not zero together. By the definition of cusps, the curve (\ref{eqn0}) has cusps at the points $t_k.$  The function $\varphi_{n\mathcal{L}}:\disc\rightarrow\mathbb{C}$ given by
 	\begin{align}\label{eqn2}
 	\varphi_{n\mathcal{L}}(z)=1+\frac{nz}{n+1}+\frac{z^n}{n+1},\quad (z\in\disc)
 	\end{align}
 maps unit circle to this curve and the unit disk onto the region bounded by the curve (\ref{eqn0}).

\par Ma and Minda\cite{MaMinda} introduced the unified class of starlike functions $\mathcal{S}^*(\varphi)$ consisting of functions $f\in\mathcal{S}$ such that $zf'(z)/f(z)\prec \varphi(z),$ for all $z\in\disc,$ where $\varphi$ is univalent function having positive real part, $\varphi(\disc)$ is symmetric about real axis and starlike with respect to $\varphi(0)=1$ and $\varphi'(0)>0.$ The image domain $\varphi_{n\mathcal{L}}(\disc)$ is symmetric about real axis, has positive real part  and starlike with respect to $\varphi_{n\mathcal{L}}(0)=1$. Also, $\varphi_{n\mathcal{L}}'(0)>0.$ Thus, the function satisfies all the conditions of Ma-Minda class and hence we can define the following class.

Let $\Snl=\mathcal{S}^*(\varphi_{n\mathcal{L}})$ be the class of function $f:\disc\rightarrow\mathbb{C}$ such that
\[\frac{zf'(z)}{f(z)}\prec \varphi_{n\mathcal{L}}(z)=1+\frac{nz}{n+1}+\frac{z^n}{n+1},\quad (z\in\disc),\]
for $n\geq 4,$ even. A function $f:\disc\rightarrow\mathbb{C}$ belongs to the class $\Snl$ if and only if there exists an analytic function $\phi$ satisfying $\phi\prec\varphi_{n\mathcal{L}}$ such that
\[f(z)=z\exp\left(\int_{0}^{z}\frac{\phi(t)-1}{t}dt\right)\]
 	The function $f_{n\mathcal{L}}:\disc\rightarrow\mathbb{C}$ given by
 	\begin{align}\label{eqn4}
 	f_{n\mathcal{L}}(z)=z\exp\left(\frac{n}{n+1}z+\frac{1}{n(n+1)}z^n\right)=z+\frac{n}{n+1}z^2+\frac{n^2}{2(n+1)}z^3+\ldots,
 	\end{align}
where $\varphi_{n\mathcal{L}}$ is given by (\ref{eqn2}). This function acts as extremal function for most of the results for the class $\Snl.$ Also, the concept of cusps is important to study the geometry for this domain as the cusp at the angle $\pi/(n-1)$ plays a vital role in computing various radii constants concerning the class $\Snl.$ Also, the class $\Snl$ becomes the class $\mathcal{S}^*(1+z)$ as the limit $n\rightarrow\infty.$ In the limiting case, the $n$-cusp domain transforms to the disk with center and radius $1$ (see Figure \ref{limit}).
\begin{figure}[h]
	\begin{center}
		\subfigure[n=10]{\includegraphics[width=1.2in]{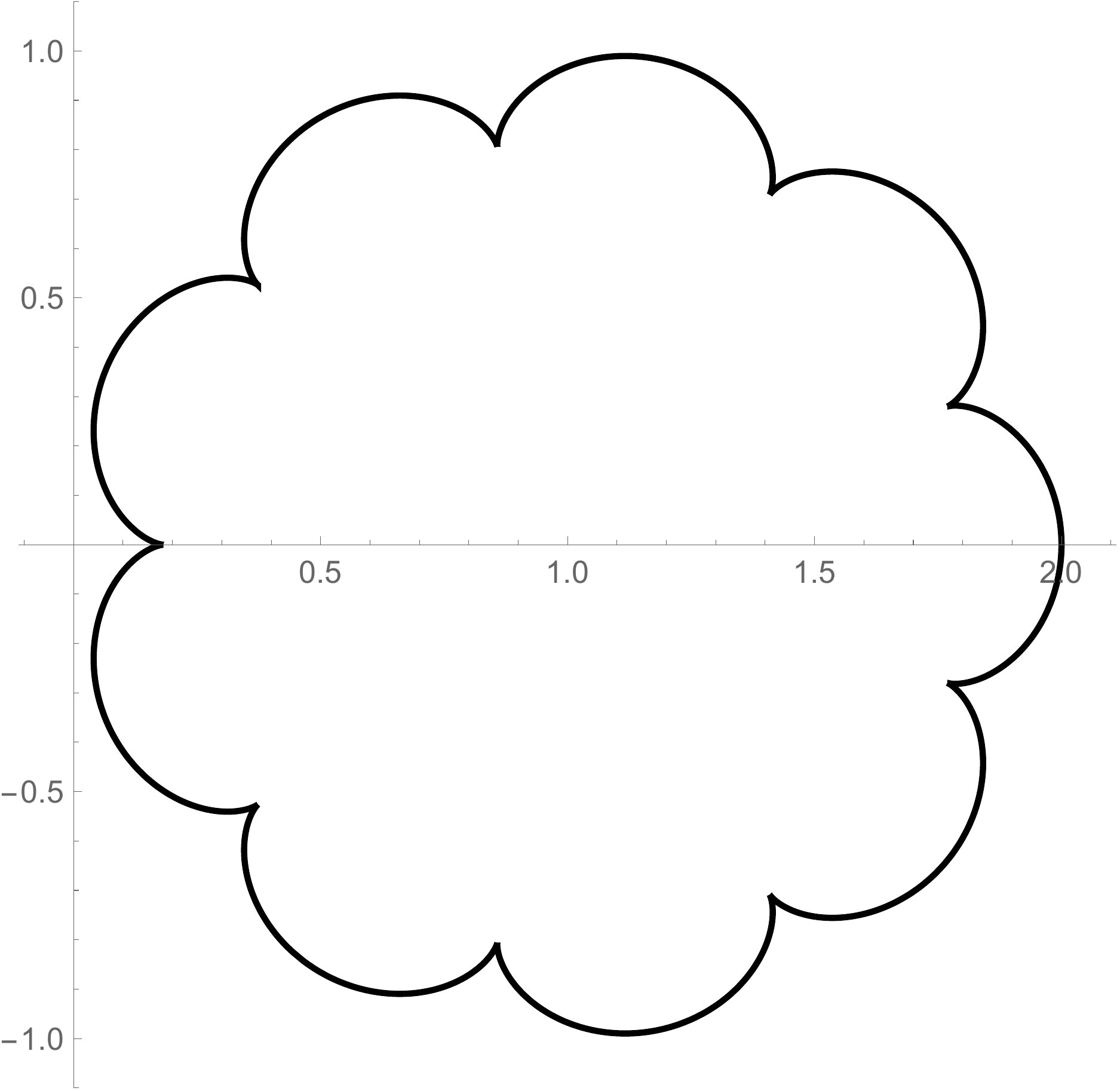}}\hspace{10pt}
		\subfigure[n=50]{\includegraphics[width=1.2in]{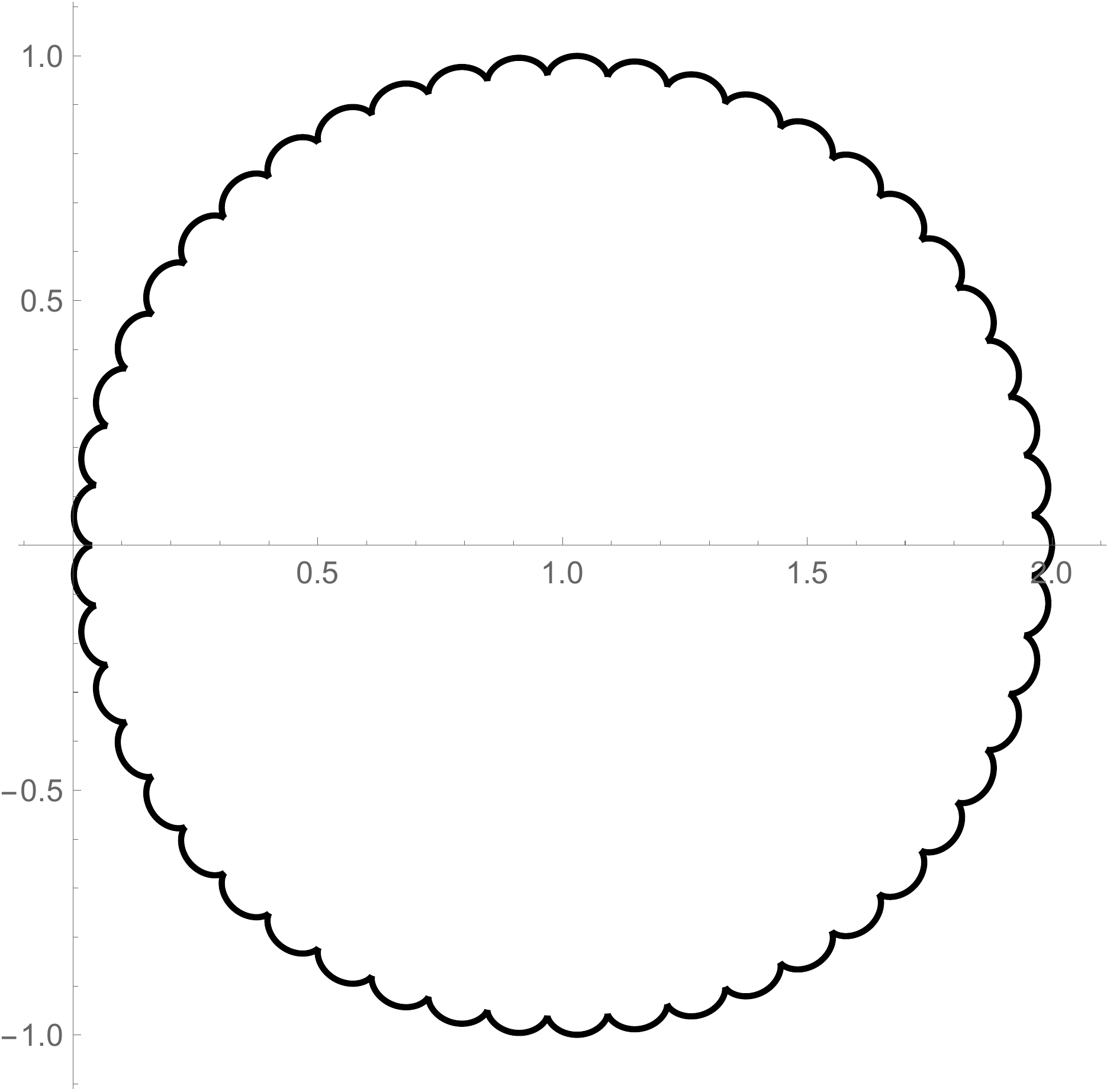}}\hspace{10pt}
		\subfigure[n=100]{\includegraphics[width=1.2in]{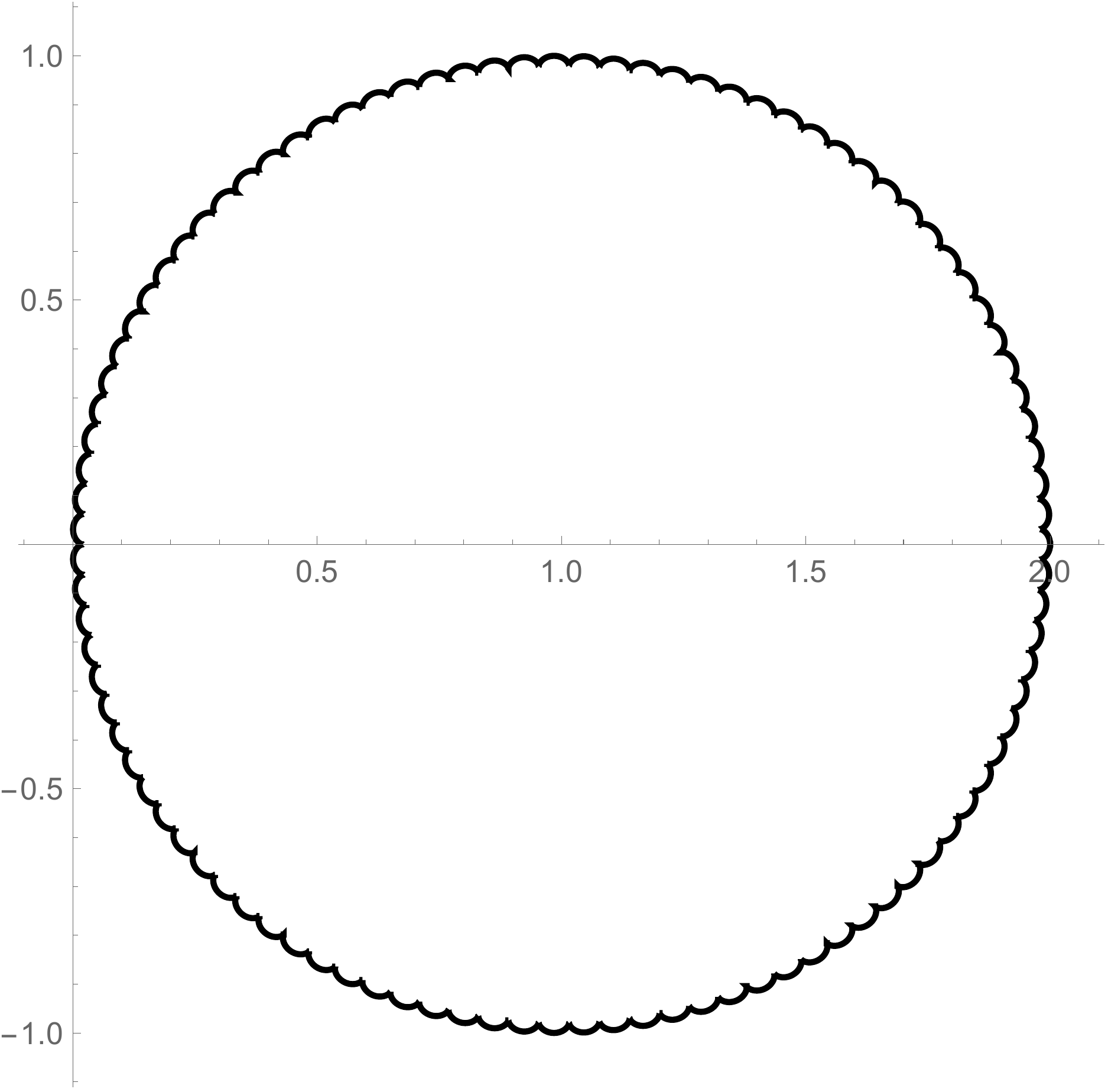}}\hspace{10pt}
		\subfigure[n=1000]{\includegraphics[width=1.2in]{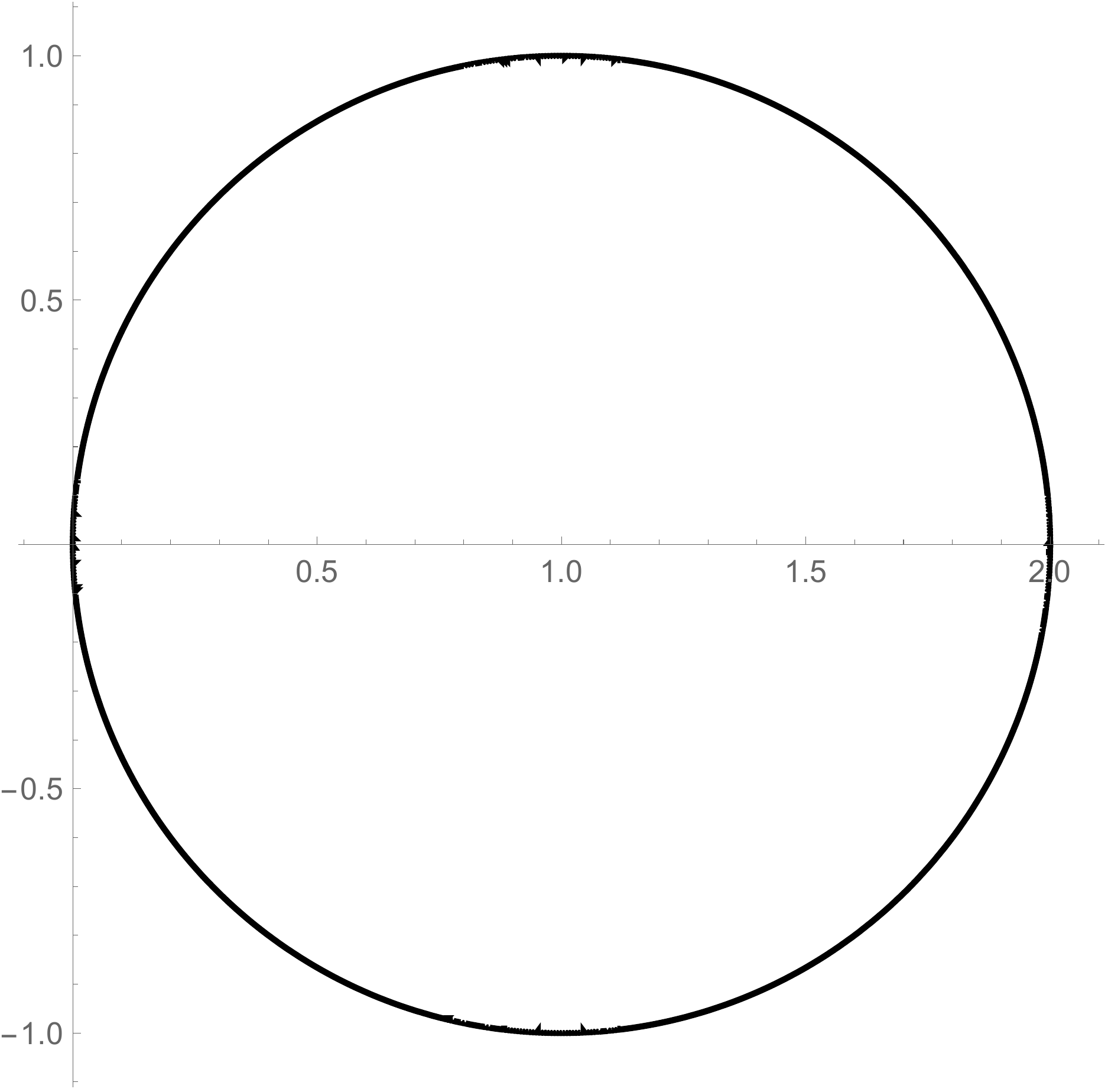}}\hspace{10pt}
		\caption{Limiting case}\label{limit}
	\end{center}
\end{figure}
\par In the present work, various inclusion relations and radii problems for the class $\Snl$ are investigated. The sharp bounds for the first fifth coefficients of a function $f\in\Snl$ are computed. Further, various inclusion relations have been established between the class $\Snl$ and various subclasses of starlike functions such as $\mathcal{S}^*(\alpha),\,\mathcal{S}\mathcal{S}^*(\beta)$ and many others. Also, the sharp $\Snl-$radius is computed for various known classes os starlike functions and radius estimates for the class $\mathcal{S}^*(1+z)$ are obtained by taking the limit as $n\rightarrow\infty.$ In the last section, the radii constants for the class $\Snl$ are computed.
 	\begin{lemma}\label{lemma}
 		For $(n+1)<a<2,$ let $r_a$ be given by
 		\begin{align*}
 			r_a=\begin{cases}
 				a-\displaystyle\frac{2}{n+1},& \displaystyle\frac{2}{n+1}<a\leq 1,\\ 		
 				\sigma\left(\displaystyle\frac{\pi}{n-1}\right),& 1\leq a<a_3, \\
 				2-a,& a_3<a<2,
 			\end{cases}
 		\end{align*}
 		where $a_3$ is the solution of the equation $\displaystyle \sigma\left(\pi/(n-1)\right)=\sigma(0)$ and the function $\sigma$ is the square of the distance from the point $(a,0)$ to the points on the curve $\partial\varphi_{n\mathcal{L}}(\disc).$ Then $\{w:|w-a|<r_a\}\subseteq \varphi_{n\mathcal{L}}(\mathbb{D})$.
 	\end{lemma}
 	\begin{proof}
 		Let $\varphi_{n\mathcal{L}}(z)$ be given by (\ref{eqn2}). Then any point on the boundary of $\varphi_{n\mathcal{L}}(\mathbb{D})$ is of the form $\varphi_{n\mathcal{L}}(e^{it})$. Since the curve $w=\varphi_{n\mathcal{L}}(e^{it})$ is symmetric with respect to real axis, so it is sufficient to consider the interval $0\leq t \leq \pi$. The parametric equation of $\varphi_{n\mathcal{L}}(e^{it})$ is given as follows:
 		\[\varphi_{n\mathcal{L}}(e^{it})=1+\frac{n}{n+1}\cos t+\frac{1}{n+1}\cos (nt) +i\left(\frac{n}{n+1}\sin t+\frac{1}{n+1} \sin (nt)\right)\\
 		\]
 		The square of the distance from the point $(a,0)$ to the points on the curve $\varphi_{n\mathcal{L}}(e^{it})$ is given by:
 		\begin{equation}\label{eqn1}
 			\sigma(t)=\left(1+\frac{n}{n+1}\cos t+\frac{1}{n+1}\cos (nt)-a\right)^2+\left(\frac{n}{n+1}\sin t+\frac{1}{n+1} \sin (nt)\right)^2.
 		\end{equation}
 		It can be easily seen that
 		\begin{equation*}
 			\sigma'(t)=4n\cos \left(\frac{(n-1)t}{2}\right) \left[(n-1) \sin\left(\frac{(1-n)t}{2}\right)+(n+1)(a-1)\sin \left(\frac{(1+n)t}{2}\right) \right].
 		\end{equation*}
 		A calculation shows that $\sigma'(t)= 0$ for $t=0,\pi,\frac{\pi}{n-1},\frac{3\pi}{n-1},\cdots,\frac{(n-3)}{n-1}\pi$ and
 		\begin{equation*}
 			\sigma''(t)=\frac{-2n\left((1-a)(1+n)\cos t+(1-a)n(n+1)\cos (nt)+(n-1)^2\cos(t-nt)\right)}{(n+1)^2}.
 		\end{equation*}
 		Clearly, it can be seen that \[\sigma''(0)=\displaystyle\frac{2n(a(1+n)^2-2(1+n^2))}{(1+n)^2} >0,\text{ for } a>\frac{2(1+n^2)}{(1+n)^2}>1.\]
 		Also, $\sigma''(\pi)>0$ for $a>2/(1+n)$ and \[\sigma''\left(\frac{\pi}{n-1}\right)>0 \text{ for } a< a_1=1-\frac{(n-1)^2}{(1+n)\cos\left(\frac{\pi}{n-1}\right)+n(n+1)\cos\left(\frac{n\pi}{n-1}\right)},\]
 		and $a_1>2(1+n^2)/(1+n)^2>1.$
 	Let us assume $a<1.$	
 		Now, $\sigma(\frac{\pi}{n-1})-\sigma(\pi)>0$ yields $(-1 + a)(1 + n)(-1 + n + n \cos(\pi/(n-1))+\cos(n \pi/(n-1)))<0$. Also, $(-1 + n + n \cos(\pi/(n-1))+\cos(n \pi/(n-1))>0$ and therefore $\sigma(\pi/(n-1))-\sigma(\pi)>0$. Hence, minimum value cannot be $\sigma(\pi/n-1)$.
 		Consider $\sigma(k\pi/(n-1))-\sigma(\pi)=(-1 + a)(1 + n)(-1 + n + n\ cos(k\pi/(n-1))+\cos(kn \pi/(n-1)))>0,$
 		for $k=3,5,\dots,n-3.$  Since $(-1 + n + n \cos(k\pi/(n-1))+\cos(kn \pi/(n-1)))>0,$ $\sigma(k\pi/(n-1))$ cannot be minimum for this case. By checking the sign of second derivative, minimum can be $\sigma(\pi/(n-1))$, $\sigma(k\pi/(n-1))$ or $\sigma(\pi)$ where $k=3,5,\dots,n-3$. A simple computation gives  $\sigma(k\pi/(n-1))-\sigma(\pi)>0$ and $\sigma(\pi/(n-1))-\sigma(\pi)>0$ and therefore minimum is $\sigma(\pi)$.
 		
 		Let us assume $a>1.$ For this case, $\sigma(\pi)>\sigma(\pi/(n-1))$ and thus $\sigma(\pi)$ cannot be minimum and $\sigma(0)$ can be minima for $a>2(1+n^2)/(1+n)^2$. In the interval $(1,2(1+n^2)/(1+n)^2)$ minimum can be $\sigma(\pi/(n-1))$ or $\sigma(k\pi/(n-1))$. By considering $\sigma(k\pi/(n-1))-\sigma(\pi/(n-1))=(a-1)(n(\cos(\pi/(n-1))-\cos(k\pi/(n-1))+(\cos(n\pi/(n-1)-\cos(kn\pi/(n-1))$ which can be proved to be greater than $0$ for $a>1$ and therefore $\sigma(k\pi/(n-1))$ cannot be the minimum and hence in the interval $(1,2(1+n^2)/(1+n)^2)$ minimum is $\sigma(\pi/(n-1))$. Now, we discuss the minimum in the interval $(2(1+n^2)/(1+n)^2,a_1)$. A calculation shows that $\sigma\left(\pi/(n-1)\right)-\sigma(0)>0$ for \[ a>a_3=\frac{-(1+4n+n^2)+n(1+n)\cos\left(\frac{\pi}{n-1}\right)+(n+1)\cos\left(\frac{n\pi}{n-1}\right)}{n(1+n)\cos\left(\frac{\pi}{n-1}\right)+(n+1)\cos\left(\frac{n\pi}{n-1}\right)-(n+1)^2},\] which is also the solution of the equation $\displaystyle \sigma\left(\pi/(n-1)\right)=\sigma(0)$. Also, $a_3$ belongs to the interval $(2(1+n^2)/(1+n)^2,a_1)$. Hence, $\sigma(\pi/(n-1))$ is minimum for $(2(1+n^2)/(1+n)^2,a_3)$ and $\sigma(0)$ is minimum for $(a_3,2)$.
 	\end{proof}
 	\section{Coefficient Estimates}
In this section, we will compute bounds on the coefficients for function in class $\Snl.$ The proof will use the following estimates (see \cite{KEOGH}, \cite{SZYNAL}, \cite{RAVI2}, respectively) for the class of analytic functions $p(z)=1+c_1z+c_2z^2+\cdots$ such that $\RE p(z)>0$ for all $z\in\disc.$
 	\begin{lemma}\label{lem1}
 		For $p(z)=1+c_1z+c_2z^2+\cdots\in\mathcal{P},$ then the following estimates holds.
 		\begin{itemize}
 			\item [(i)] $|c_2-vc_1^2|\leq 2\max\{1,|2v-1|\},$
 			\item[(ii)] $|c_3-2\beta c_1c_3+\delta c_1^3|\leq 2$ if $0\leq\beta\leq 1$ and $\beta(2\beta-1)\leq\delta\leq\beta,$
 			\item[(iii)] $|\gamma c_1^4+ac_2^2+2\alpha c_1c_3-(3/2)\beta c_1^2c_2-c_4|\leq 2, $ when $0<\alpha<1,\,0<a<1$ and $8a(1-a)((\alpha\beta-2\gamma)^2+(\alpha(a+\alpha)-\beta)^2)+\alpha(1-\alpha)(\beta-a\alpha)^2\leq 4\alpha^2(1-\alpha)^2a(1-a).$		
 		\end{itemize}
 	\end{lemma}
 	\begin{theorem}
 		If $f(z)=z+a_2z^2+a_3z^3+\cdots\in\mathcal{S}^*_{n\mathcal{L}},$ then $|a_2|\leq n/(n+1),\, |a_3|\leq n/(2(n+1)),\,|a_4|\leq n/(12(n+1))$ and $|a_5|\leq n/(4(n+1)).$ All the estimates are best possible.
 	\end{theorem}
 \begin{proof}
 	Let $p(z)=zf'(z)/f(z)=1+b_1z+b_2z^2+\cdots\in\mathcal{P}.$ A simple computation gives
 	\begin{align}\label{eqn2.1}
 	(n-1)a_n=\sum_{k=1}^{n-1}b_ka_{n-k},\, \text{for }n>1.
 	\end{align}
 	Since $\varphi_{n\mathcal{L}}$ is univalent and $p\prec\varphi_{n\mathcal{L}}$, we get
 		\[p_1(z)=\frac{1+\varphi^{-1}_{n\mathcal{L}}\left(p(z)\right)}{1-\varphi^{-1}_{n\mathcal{L}}\left(p(z)\right)}=1+c_1z+c_2z^2+c_3z^3\cdots\in\mathcal{P}.\]
 		Thus,
 		\[p(z)=\varphi_{n\mathcal{L}}\left(\frac{p_1(z)-1}{p_1(z)+1}\right).\]
 A calculation using (\ref{eqn2.1}) gives
 \begin{align*}
 a_2&=b_1=\frac{n}{2(n+1)}c_1\\
 a_3&=\frac{n}{8(n+1)^2}\left(16\left(n+1\right)c_2-c_1^2\right)\\
 a_4&=\frac{n}{48(n+1)^3}\left(\left(n+2\right)c_1^3-2\left(n^2+5n+4\right)c_1c_2+8\left(n^2+2n+1\right)c_3\right)\\
 a_5&=\frac{n}{384}\left(\frac{48}{n+1}c_4-\frac{(2+n)(3+2n)}{(n+1)^4}c_1^4+\frac{4(n^2+7n+9)}{(n+1)^3}c_1^2c_2-\frac{12(n+2)}{(n+1)^2}c_2^2-\frac{16(n+3)}{(n+1)^2}c_1c_3\right).
 \end{align*}
Since $|c_i|\leq 2,$ for all $i$, we get $|a_2|\leq n/(n+1).$ Using Lemma \ref{lem1} (i) for $v=1/(2(n+1)),$ we obtain
\[|a_3|\leq \frac{n}{4(n+1)}\left|c_2-\left(\frac{1}{2(n+1)}\right)c_1^2\right|\leq\frac{n}{2(n+1)}.\]
Now,
\begin{align*}
|a_4|&=\frac{n}{48(n+1)^3}\left|(n+2)c_1^3-2(n+1)(n+4)c_1c_2+8(n+1)^2c_3\right|\\
&=\frac{n}{48(n+1)}\left|\frac{(n+2)}{8(n+1)^2}c_1^3-\frac{n+4}{4(n+1)}c_1c_2+c_3\right|.
\end{align*}
Let us take $\beta=(n+4)/(8(n+1))$ and $\delta=(n+2)/(8(n+1)^2)$. For $n\geq 4,$ it can be easily seen that $0\leq\beta\leq 1$ and $\delta\leq \beta.$ Also, $\beta(2\beta-1)=-3n(n+4)/(32(n+1)^2)<0<\delta\leq\beta.$ Thus, by Lemma \ref{lem1}(ii), $|a_4|\leq n/(12(n+1)).$ Lastly,
\[|a_5|=\frac{n}{8(n+1)}\left|\frac{(n+2)(3+2n)}{48(n+1)^3}c_1^4-\frac{n^2+7n+9}{12(n+1)^2}c_1^2c_2+\frac{n+2}{4(n+1)}c_2^2+\frac{n+3}{3(n+1)}c_1c_3-c_4\right|.\]
We shall show that $\beta=(n^2+7n+9)/(18(n+1)^2),\,a=(n+2)/(4(n+1)),\,\alpha=(n+3)/(6(n+1))$ and $\gamma=(n+2)(2n+3)/(48(n+1)^3)$ satisfies the conditions of Lemma \ref{lem1} (iii). For $n\geq 4,$ it is clear that $0<a,\alpha<1.$ Now, the condition $8a(1-a)((\alpha\beta-2\gamma)^2+(\alpha(a+\alpha)-\beta)^2)+\alpha(1-\alpha)(\beta-a\alpha)^2- 4\alpha^2(1-\alpha)^2a(1-a)$ reduces to $-(5832 + 46656 n + 156564 n^2 + 286536 n^3 + 310942 n^4 + 203428 n^5 + 77806 n^6 + 15816 n^7 + 1301 n^8)/(93312 (1 + n)^8)\leq 0.$ This holds for all $n\in\mathbb{N}.$ Since $\alpha,\beta,\gamma$ and $a$ satisfies all the conditions of Lemma \ref{lem1}(iii), $|a_5|\leq n/(4(n+1)).$ For sharpness, the following functions are extremal for the initial coefficients $a_i(i=2,3,4,5)$ and are given by
\[f_i(z)=z\exp\left(\int_{0}^{z}\frac{\varphi_{n\mathcal{L}}(t^{i-1})-1}{t}dt\right),\quad i=2,3,4,5.\qedhere\]
 \end{proof}
 	\section{Inclusion Relations}
 	This section deals with inclusion relation between the class $\Snl$ and various classes which depends on a parameter. For instance, $\mathcal{S}\mathcal{S}^*(\beta)\,(0<\beta<1)$ is the class characterized by $|\arg(zf'(z)/f(z))|<\beta\pi/2,$ $\mathcal{S}^*[A,B]\,(-1\leq B<A\leq 1)=\mathcal{S}^*(1+Az)/(1+Bz)$ is the class of Janowski starlike functions, $\mathcal{S}^*(\alpha)=\mathcal{S}^*[1-2\alpha,-1]$ is the class of starlike function or order $\alpha\,(0\leq \alpha<1).$ Sokol\cite{SOKOL1} introduced the class $\mathcal{S}^*(\sqrt{1+cz})$ which is associated with right loop of the Cassinian ovals given by $(u^2+v^2)^2-2(u^2-v^2)=c^2-1,$ for $0<c\leq 1.$ For $c=1,$ this class reduced to the class $\mathcal{S}^*_L$. Also, for $0\leq \alpha<1,$ the generalized class $\mathcal{S}\mathcal{L}^*(\alpha)=\mathcal{S}^*(\alpha+(1-\alpha)\sqrt{1+z})$ was introduced by Khatter \emph{et.al} \cite{KHATTER} and this class also reduces to $\mathcal{S}^*_L$ for $\alpha=0.$ Another interesting class $\mathcal{M}(\beta)$ of analytic functions such that $\RE(zf'(z)/f(z)) < \beta,$ for $\beta>1$, was studied by Uralegaddi \cite{URAL}. The next theorem gives various inclusion relation of the class $\Snl$ with these mentioned classes.
 	
 	\begin{theorem}\label{inclusion}
 		For $\mathcal{S}^*_{n\mathcal{L}},$ the following inclusion relations holds:
 		\begin{itemize}
 			\item [(a)] $\mathcal{S}^*_{n\mathcal{L}}\subset\mathcal{S}^*(\alpha),$ where  $0\leq \alpha\leq \alpha_0,$ for $(n+1)\alpha_0= 1+\cos(nt_0)+n(1+\cos t_0)$ and $t_0=n\pi/(n+1).$
 			\item[(b)] $\mathcal{S}^*_{n\mathcal{L}}\subset\mathcal{S}\mathcal{S}^*(\beta),$ for $\beta\geq 2\beta_0/\pi,$ where $\tan\beta_0=\sin\left(\pi/n\right)/\left(1-\cos\left(\pi/n\right)\right).$
 			\item[(c)] $\mathcal{S}\mathcal{L}^*(\alpha)\subset \mathcal{S}^*_{n\mathcal{L}}$ for $\alpha\geq 2/(n+1).$
 			\item[(d)] $\mathcal{S}^*\left(\sqrt{1+cz}\right)\subset\mathcal{S}^*_{n\mathcal{L}},$ for $0<c\leq 1-4/(n+1)^2.$
 			\item[(e)] $\mathcal{S}^*[1-\alpha,0]\subset\Snl,$ for $2/(n+1)\leq\alpha\leq1$.
 			\item[(f)] $\mathcal{S}^*[\alpha,-\alpha]\subset\Snl,$ for $0\leq\alpha\leq|(t^n+tn)/(2+t^n+2n+tn)|,$ where $t=e^{i\pi/(n-1)}.$
 			\item[(g)] $\Snl\subset\mathcal{S}^*[1,-(M-1)/M],$ for $M\geq 1.$
 			\item[(h)] $\Snl\subset\mathcal{M}(\beta),$ for $\beta>2.$
 		\end{itemize}
 	\end{theorem}
 \begin{proof}
 	(a) Let $f\in\Snl.$ Then \[\RE\left(\frac{zf'(z)}{f(z)}\right)>\underset{|z|=1}{\min}\RE\left(\varphi_{n\mathcal{L}}(z)\right).\]
 	
 	For $z=e^{it}$,
 	\[ \RE\left(\varphi_{n\mathcal{L}}(e^{it})\right)=1+\frac{n\cos t}{n+1}+\frac{\cos (nt)}{n+1}:=h(t),\]
 	where $t\in(-\pi,\pi).$ To compute the minimum value of $h(t)$, we shall obtain all the possible values of $t$ such that $h'(t)=0$ and $h''(t)>0.$ For $t_0=\pm n\pi/(n+1),$
 	\begin{align*}
 	h'(t_0)=\mp\frac{n\left(\sin\left(\displaystyle\frac{n\pi}{n+1}\right)+\sin\left(\displaystyle\frac{n^2\pi}{n+1}\right)\right)}{n+1}.
 	\end{align*}
 	Since $n$ is even, $h'(t_0)=0.$ Also,
 	\[h''(t_0)=\frac{-n\left(\cos\left(\displaystyle\frac{n\pi}{n+1}\right)+n\cos\left(\displaystyle\frac{n^2\pi}{n+1}\right)\right)}{n+1}>0,\]
 	for $n$ even.
 	 Hence,
 	\[\underset{|z|=1}{\min}\RE\left(\varphi_{n\mathcal{L}}(z)\right)=\RE\left(\varphi_{n\mathcal{L}}(e^{it_0})\right)=1+\frac{n\cos t_0}{n+1}+\frac{\cos(nt_0)}{n+1}=\alpha_0.\]
 	Thus $f\in\Snl\subset\mathcal{S}^*(\alpha),$ for $0<\alpha\leq \alpha_0.$ For instance, the curve $\gamma_1:\,\RE w=\alpha_0$ in Figure \ref{fig1} shows the result is best possible for $n=8$.\\
 	\begin{tabular}{cl}

 	\begin{tabular}{c}\label{fig1}
 	\includegraphics[width=2.5in]{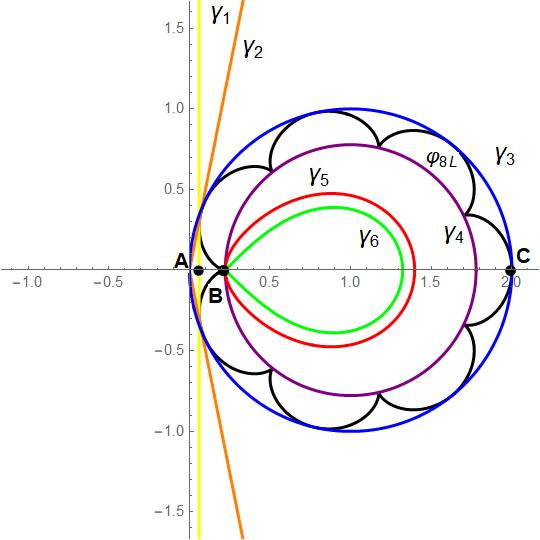}
 	\end{tabular}
 & \begin{tabular}{l}
 	\parbox{0.4\linewidth}{
 	$\gamma_1:\,\RE w =1-\cos\left(\displaystyle\frac{\pi}{9}\right)$\\
 	$\gamma_2:\,|\arg w|=\tan^{-1}(5.0273)$\\
 	$\gamma_3: \,|w-1|=1$\\
 	$\gamma_4: |w-1|=\displaystyle\frac{7}{9}$\\
 	$\gamma_5: |w^2-1|=\displaystyle\frac{77}{81}$\\
 	$\gamma_6: \left|\left(\displaystyle\frac{9w-2}{7}\right)^2-1\right|=1$\\
 	$\text{A}:\, 1-\cos\left(\displaystyle\frac{\pi}{9}\right)$\\
 	$\text{B}: \,\displaystyle\frac{2}{9}$\\

 	$\text{C}:\,2$
 }
 \end{tabular}\\
\end{tabular}
 		\begin{center}
 			{Figure 1. Inclusion Relation for class $\Snl$}
 		\end{center}
 \par (b) For $f\in\Snl,$
 \begin{align*}
 \left|\arg\left(\frac{zf'(z)}{f(z)}\right)\right|&<\underset{|z|=1}{\max}\arg\left(\varphi_{n\mathcal{L}}(z)\right)\\&=\underset{t\in(-\pi,\pi]}{\max}\arg\left(\varphi_{n\mathcal{L}}(e^{it})\right)\\ &=\underset{t\in(\pi,\pi]}{\max}\tan^{-1}\left(\frac{n\sin t+\sin(nt)}{n+1+n\cos t+\cos(nt)}\right) \\ &= \tan^{-1}\left(\underset{t\in(-\pi,\pi]}{\max}g(t)\right),
 \end{align*}
 where $g(t)=(n\sin t+\sin(nt))/(n+1+n\cos t+\cos(nt)).$ It is sufficient to compute the maximum value of $g(t),$ for $-\pi<t\leq\pi.$ For $t_1= (n-1)\pi/n,$
 \begin{align*}
 g'(t_1)=\frac{4n(n+1)\cos\left(\displaystyle\frac{(n-1)^2\pi}{2n}\right)\sin\left(\displaystyle\frac{\pi}{2n}\right)\sin\left(\displaystyle\frac{n\pi}{2}\right)}{\left(-1-n+n\cos\left(\displaystyle\frac{\pi}{n}\right)\cos\left(n\pi\right)\right)}=0,
 \end{align*}
as $n$ is even. A simple computation shows that $g''(t_1)<0$ for $n$ even. Hence,
\[\left|\arg\left(\frac{zf'(z)}{f(z)}\right)\right|<\tan^{-1}\left(g(t_1)\right)=\tan^{-1}\left(\frac{\sin\left(\pi/n\right)}{1-\cos(\pi/n)}\right)=\beta_0.\]
So, $\Snl\subset\mathcal{S}\mathcal{S}^*(\beta),$ where $\beta\geq 2\beta_0/\pi.$ Sharpness for the case $n=8$ is depicted by the curve $\gamma_2:\,\arg w=\tan^{-1}\left(\sin(\pi/8)/(1-\cos(\pi/8))\right)$ in the Figure \ref{fig1}.
 	\par(c) To show the function $f\in\mathcal{S}^*_L(\alpha)$ lies in the class $\Snl,$ we will use the \cite[Lemma 2.1,pp 236]{KHATTER} that gives
 	\[\alpha<\RE\left(\frac{zf'(z)}{f(z)}\right)<\alpha+(1-\alpha)\sqrt{2}.\]
 	The function $f\in\Snl$ if either $\alpha\geq 2/(n+1)$ or $\alpha+(1-\alpha)\sqrt{2}\leq 2.$ Thus, $f\in\Snl(n\geq 4)$ for $\alpha\geq 2/(n+1).$ The case $n=8$ is illustrated in Figure \ref{fig1} by curve $\gamma_5.$
 	\par(d) Let $f\in\mathcal{S}^*(\sqrt{1+cz})\,(0<c<1).$ Then the quantity $zf'(z)/f(z)\prec \sqrt{1+cz}$ and
 	\[\sqrt{1-c}<\RE\left(\frac{zf'(z)}{f(z)}\right)<\sqrt{1+c}.\]
 	Note that $\sqrt{1+c}<\sqrt{2}<2.$ Thus the function $f\in\Snl$ if $\sqrt{1-c}\geq 2/(n+1).$ This gives $c\leq 1-4/(n+1)^2.$ To see sharpness for $n=8,$ see the curve $\gamma_6$ in Figure \ref{fig1}.
 	\par (e) Proceeding as in part (d), we get the function $f\in\mathcal{S}^*[1-\alpha,0]$ lies in the class $\Snl$ if
 	\[\frac{2}{n+1}\leq \alpha<\RE\left(\frac{zf'(z)}{f(z)}\right)<2-\alpha\leq 2,\]
 	which holds for  $\alpha\geq 2/(n+1).$ (See $\gamma_4$ in Figure \ref{fig1} )
 	\par (f) Let $f\in\mathcal{S}^*[\alpha,-\alpha].$ In order to obtain condition on $\alpha$ such that $f\in\Snl,$ we compute the solution of the equation $(1+\alpha r)/(1-\alpha r)=\varphi_{n\mathcal{L}}(e^{i\pi/(n-1)}),$ which simplifies to $\alpha\leq |\left(e^{in\pi/(n-1)}+ne^{i\pi/(n-1)}\right)/\left(2+2n+e^{in\pi/(n-1)}+ne^{i\pi/(n-1)}\right)|.$
 	\par (g) Let $f\in\Snl.$ Then for $z\in\disc,$
 	\begin{align*}
 	\left|\frac{zf'(z)}{f(z)}-M\right|&\leq\left|1+\frac{nz}{n+1}+\frac{z^n}{n+1}-M\right|\\&\leq \left|1-M\right|+\frac{n|z|}{n+1}+\frac{|z|^n}{n+1}\\
 	&< |1-M|+\frac{n}{n+1}+\frac{1}{n+1}\\
 	&=|1-M|+1
 \end{align*}
 Thus, for $M\geq 1,$ $|zf'(z)/f(z)-M|< M.$ For $n=8,$ sharpness for this class can be seen by curve $\gamma_3$ in Figure \ref{fig1}.
\end{proof}
 \begin{theorem}
 	The class $\mathcal{S}^*[A,B]\subset\Snl,\,-1\leq B<A\leq1 ,$ if one of the following conditions holds.
 	\begin{itemize}
 		\item [(a)] $2(1-B^2)\leq (n+1)(1-AB^2)\leq (n+1)(1-B^2)$ and $(n+1)A\leq 2B+n-1,$
 		\item[(b)] $(1-B^2)\leq 1-AB^2\leq a_3(1-B^2)$ and $A\leq \sigma(\pi/(n-1))(1-B^2)+B,$
 		\item[(c)] $a_3(1-B^2)\leq1-AB^2\leq 2(1-B^2)$ and $a\leq 2B+1,$
 	\end{itemize}
 where
 \[\sigma(t)=\left(\frac{1-AB^2}{1-B^2}-\left(1+\frac{n \cos t}{n+1}+\frac{\cos (nt)}{n+1}\right)\right)^2+\left(\frac{n\sin t}{n+1}+\frac{\sin (nt)}{n+1}\right)^2\]
 and $a_3$ is the point lying in interval $(1,2)$ such that $\sigma(0)=\sigma(\pi/(n-1)).$
 \end{theorem}
\begin{proof}
	Let $f\in\mathcal{S}^*[A,B].$ Then the image of $zf'(z)/f(z)$ lies inside the disk
	\[\left|\frac{zf'(z)}{f(z)}-\frac{1-AB}{1-B^2}\right|\leq\frac{A-B}{1-B^2},\] with center $a:=(1-AB)/(1-B^2)$ and radius $r_a:=(A-B)/(1-B^2).$ To show that this disk lies in the domain $\varphi_{n\mathcal{L}}(\disc),$ we shall use the \lemref{lemma}. If $2/(n+1)<a\leq 1,$ then $r_a<a-2/(n+1)$ which is equivalent to part (a). For $1<a<a_3,$ the condition in (b) is obtained by solving $r_a\leq \sigma(\pi/(n-1)).$ Lastly, part (c) is equivalent to $r_a\leq 2-a,$ for $a_3<a< 2.$
\end{proof}
\section{$\Snl-$radius}
This section deals with the $\Snl-$radius for various known subclasses of starlike functions. MacGregor \cite{MAC,MAC1,MAC2} studied the class of $\mathcal{W}$ of functions $f\in\mathcal{A}$ such that $f(z)/z\in\mathcal{P},$ the class $\mathcal{F}_1$ of functions $f\in\mathcal{A}$ such that $\RE(f(z)/g(z))>0$ for some $g\in\mathcal{A}$ with $\RE(g(z)/z)>0$ and the class $\mathcal{F}_2$ of functions $f\in\mathcal{A}$ such that $|f(z)/g(z)-1|<1$ for some $g\in\mathcal{A}$ satisfying $\RE(g(z)/z)>0.$ An analytic function $p(z)=1+c_1z+c_2z^2+\ldots\in\mathcal{P}(\alpha),$ for $0\leq\alpha<1$ and $z\in\disc,$ satisfies
\begin{align}\label{eqn4.1}
	\left|\frac{zp'(z)}{p(z)}\right|\leq \frac{2r(1-\alpha)}{(1-r)(1+(1-2\alpha)r)},
\end{align}
for $|z|=r<1.$ Many classes are introduced by various authors for an appropriate choice of the function $\varphi$ in the class $\mathcal{S}^*(\varphi)$ defined by Ma and Minda \cite{MaMinda}. Some of the known classes inspired by Ma-Minda classes are $\mathcal{S}^*_L=\mathcal{S}^*(\sqrt{1+z}),\,\mathcal{S}^*_{RL}=\mathcal{S}^*(\sqrt{2}-(\sqrt{2}-1)\sqrt{(1-z)/(1+2(\sqrt{2}-1)z)}),\,\mathcal{S}^*_e=\mathcal{S}^*(e^z),\,\mathcal{S}^*_C=\mathcal{S}^*(1+4z/3+2z^2/3),\,\mathcal{S}^*_{\leftmoon}=\mathcal{S}^*(z+\sqrt{1+z^2}),\mathcal{S}^*_R=\mathcal{S}^*((k^2+z^2)/(k^2-kz))\,(k=\sqrt{2}+1),\,\mathcal{S}^*_{sin}=\mathcal{S}^*(1+\sin z),\,\mathcal{S}^*_{lim}=\mathcal{S}^*(1+\sqrt{2}z+z^2/2),\,\mathcal{S}^*_{SG}=\mathcal{S}^*(2/(1+e^{-z})),\mathcal{S}^*_{3\mathcal{L}}=\mathcal{S}^*(1+4z/5+z^4/5),\,\mathcal{S}^*_{EL}=\mathcal{S}^*(ke^z+(1-k)(1+z)),\,\mathcal{S}^*_{ne}=\mathcal{S}^*(1+z-z^3/3),\,\mathcal{S}^*(1+ze^z),\,\mathcal{S}^*(\cos z),\,\mathcal{S}^*(\cosh z),\,\mathcal{S}^*(1+\sinh^{-1}(z)),\,\mathcal{S}^*_{car}=\mathcal{S}^*(1+z+z^2/2).$ These classes are studied in \cite{GOEL,GUPTA,KANSH,MEND,MEND1,YUNUS,WANI2,SUSHIL,SIVA,SIVA1,CHO1,GANDHI,PSHARMA,SOKOL}. The class $\mathcal{B}\mathcal{S}(\alpha)=\mathcal{S}^*(1+(1-\alpha z^2))$\cite{KARGAR} is the class of functions $f\in\mathcal{A}$ such that $zf'(z)/f(z)\prec 1/(1-\alpha z^2),$ for $0<\alpha\leq 1$.
\begin{theorem}
	The $\Snl-$radius for various classes $\mathcal{M}(\beta)$ and $\mathcal{B}\mathcal{S}(\alpha)$ is as follows
	\begin{enumerate}
		\item[(a)]$\mathcal{R}_{\Snl}(\mathcal{M}(\beta))=\displaystyle\frac{n-1}{(2\beta-1)n+(2\beta-3)}.$
		\item[(b)] $\mathcal{R}_{\Snl}(\mathcal{B}\mathcal{S}(\alpha))=\displaystyle\frac{1+n+\sqrt{1+4\alpha+2n-8\alpha n+n^2+4\alpha n^2}}{2\alpha(1-n)}$
	\end{enumerate}
\end{theorem}
\begin{proof}
	(a) Let $f\in\mathcal{M}(\beta).$ Then for $|z|=r,$
	\[\left|\frac{zf'(z)}{f(z)}-\frac{1+(1-2\beta)r^2}{1-r^2}\right|\leq \frac{2r(\beta-1)}{1-r^2}.\]
	We observe that the center of the above disk $(1+(1-2\beta)r^2)/(1-r^2)<1,$ for $\beta>1.$ By using \lemref{lemma} , we get
	\[\frac{2r(\beta-1)}{1-r^2}\leq \frac{1+(1-2\beta)r^2}{1-r^2}-\frac{2}{n+1}.\]
On simplification, this gives $r\leq (n-1)/((2\beta-1)n+2\beta-3)=\mathcal{R}_{\Snl}(\mathcal{M}(\beta)).$ The bound is sharp for the function $f_1(z)=z(1-z)^{2(\beta-1)}\in\mathcal{M}(\beta).$ For $z=\mathcal{R}_{\Snl}(\mathcal{M}(\beta)),$ the term $zf_1/f_1$ takes value $2/(n+1).$
\par (b) For $f\in\mathcal{B}\mathcal{S}(\alpha),$ we have $zf'(z)/f(z)\prec 1+z/(1-\alpha z^2)$,
 which gives
 \[\left|\frac{zf'(z)}{f(z)}-1\right|\leq \frac{r}{1-\alpha r^2},\]
 for $|z|<r.$ By using \lemref{lemma}, we get $r/(1-\alpha r^2)\leq 1-2/(n+1)$ and it simplifies to $r\leq \left(1+n+\sqrt{1+4\alpha+2n-8\alpha n+n^2+4\alpha n^2}\right)/(2\alpha(1-n))=\mathcal{R}_{\Snl}(\mathcal{B}\mathcal{S}(\alpha)),$ for $0<\alpha<1.$ For sharpness, consider the function $f_2$ given by
 \[f_2(z)=z\left(\frac{1+\sqrt{\alpha }z}{1-\sqrt{a}z}\right)^{1/(2\sqrt{\alpha})}.\]
 At $z=-\mathcal{R}_{\Snl}(\mathcal{B}\mathcal{S}(\alpha)),$ the quantity $zf_2'(z)/f_2(z)=2/(n+1).$
\end{proof}

\begin{theorem}
	The $\Snl-$radius for the various ratio classes such as $\mathcal{W},$ $\mathcal{F}_1$ and $\mathcal{F}_2$ is given by
	\begin{itemize}
		\item[(a)] $\mathcal{R}_{\Snl}(\mathcal{W})=\displaystyle\frac{\sqrt{2(1+n^2)}-n-1}{n-1}$
		\item[(b)] $\mathcal{R}_{\Snl}(\mathcal{F}_1)=\displaystyle\frac{2(1+n)-\sqrt{5n^2+6n+5}}{n-1}$
		\item[(c)] $\mathcal{R}_{\Snl}(\mathcal{F}_2)=\displaystyle\frac{3(n+1)-\sqrt{17n^2+10n+9}}{4n}.$
	\end{itemize}
\end{theorem}
\begin{proof}
	(a) Let $f\in \mathcal{W}.$ Then $f(z)/z\in\mathcal{P},$ for all $z\in\disc.$ Let us define function $p\in\mathcal{P}$ such that $p(z)=f(z)/z.$ Then
	\[\frac{zf'(z)}{f(z)}=1+\frac{zp'(z)}{p(z)}.\]
	 Thus, we have
	 \[\left|\frac{zf'(z)}{f(z)}-1\right|\leq \frac{2r}{1-r^2}.\]
	 By using \lemref{lemma}, the function $f\in\Snl$ for $|z|<r$ if $2r/(1-r^2)<1-2/(n+1).$ This simplifies to $r\leq (\sqrt{2(1+n^2)}-n-1)/(n-1).$ The result is sharp for function $f_1(z)=z(1+z)/(1-z)$ (See Figure \ref{fig2}(a)). For this function, we have
	 \[\left.\frac{zf_1'(z)}{f_1(z)}\right|_{z=-\frac{\sqrt{2(1+n^2)}-n-1}{n-1}}=\frac{2}{n+1}.\]
	 \par (b) For $f\in\mathcal{F}_1,$ let us define functions $k_1,k_2:\disc\rightarrow\mathbb{C}$ such that $k_1(z)=f(z)/g(z)$ and $k_2(z)=g(z)/z.$ Then $k_1,k_2\in\mathcal{P}$ and $f(z)=zk_1(z)k_2(z).$ A direct calculation shows that
	 \[\frac{zf'(z)}{f(z)}=1+\frac{zk_1'(z)}{k_1(z)}+\frac{zk_2'(z)}{k_2(z)}\]
	 and using (\ref{eqn4.1}), we get 
	 \[\left|\frac{zf'(z)}{f(z)}-1\right|\leq \frac{4r}{(1-r^2)}.\]
By using \lemref{lemma} to get the desired result, we have
  $4r/(1-r^2)\leq 1-2/(n+1)$ which yields $r\leq \left(2(1+n)-\sqrt{5n^2+6n+5}\right)/(n-1).$ For sharpness, consider the function $f_2(z)=z((1+z)/(1-z))^2$ and $g_2(z)=z(1+z)/(1-z).$ Further,
  \[\left.\frac{zf_2'(z)}{f_2(z)}\right|_{z=-\frac{2(1+n)-\sqrt{5n^2+6n+5}}{n-1}}=\frac{2}{n+1}.\]
  \par (c) Let $f\in\mathcal{F}_2.$ Then there is a function $g\in\mathcal{A}$ such that $|f(z)/g(z)-1|<1$ and $g(z)/z\in\mathcal{P}.$ We define functions $k_1,k_2:\disc\rightarrow\mathbb{C}$ as $k_1(z)=g(z)/f(z)$ and $k_2(z)=g(z)/z.$ By definition of class $\mathcal{F}_2$, $k_1\in\mathcal{P}(1/2),\,k_2\in\mathcal{P}$ and $f(z)=zk_2(z)/k_1(z).$ A simple computation shows that
  \[\frac{zf'(z)}{f(z)}=1+\frac{zk_2'(z)}{k_2(z)}-\frac{zk_1'(z)}{k_1(z)}.\]
  By using (\ref{eqn4.1}), we get
  \[\left|\frac{zf'(z)}{f(z)}-1\right|\leq \frac{3r+r^2}{1-r^2}.\]
  Thus, image domain of the function $zf'(z)/f(z)$ lies in $\varphi_{n\mathcal{L}}(\disc)$ if $(3r+r^2)/(1-r^2)\leq 1-2/(n+1),$ by \lemref{lemma}. This holds for $r\leq\left(
  (n+1)-\sqrt{17n^2+10n+9}\right)/(4n).$ The bound is sharp for the function $f_3(z)=z(1+z)^2/(1-z)$ and function $g_3(z)=z(1+z)/(1-z)$. For $z=-\left(
  (n+1)-\sqrt{17n^2+10n+9}\right)/(4n),$ the quantity $zf_3'(z)/f_3(z)=2/(n+1).$
  \par The Sharpness for all the parts are illustrated in Figure \ref{fig2}.
\end{proof}
\begin{figure}[h]
	\begin{center}
		\subfigure[ $\mathcal{W}$]{\includegraphics[width=1.5in]{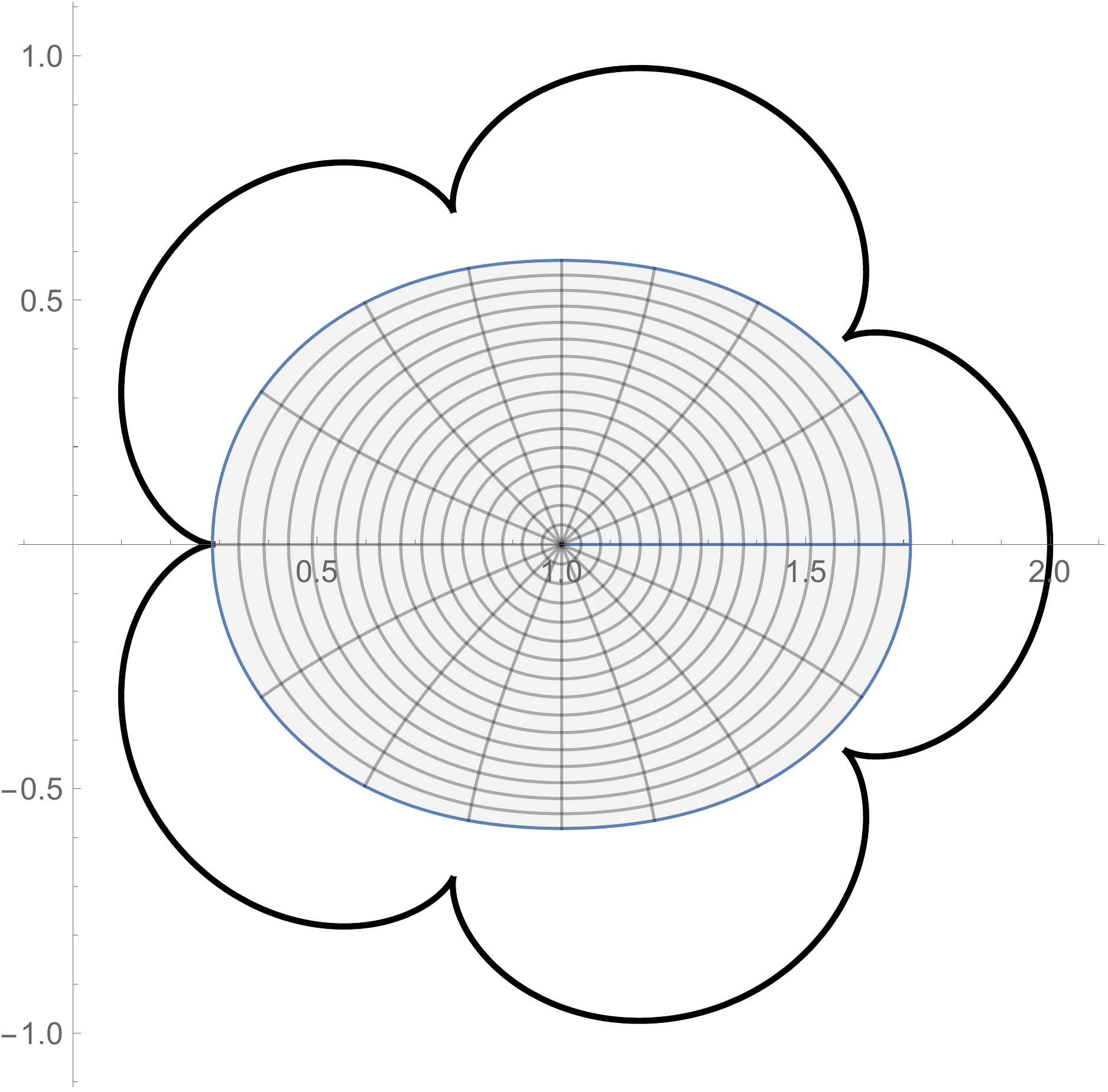}}\hspace{10pt}
		\subfigure[$\mathcal{F}_1$]{\includegraphics[width=1.5in]{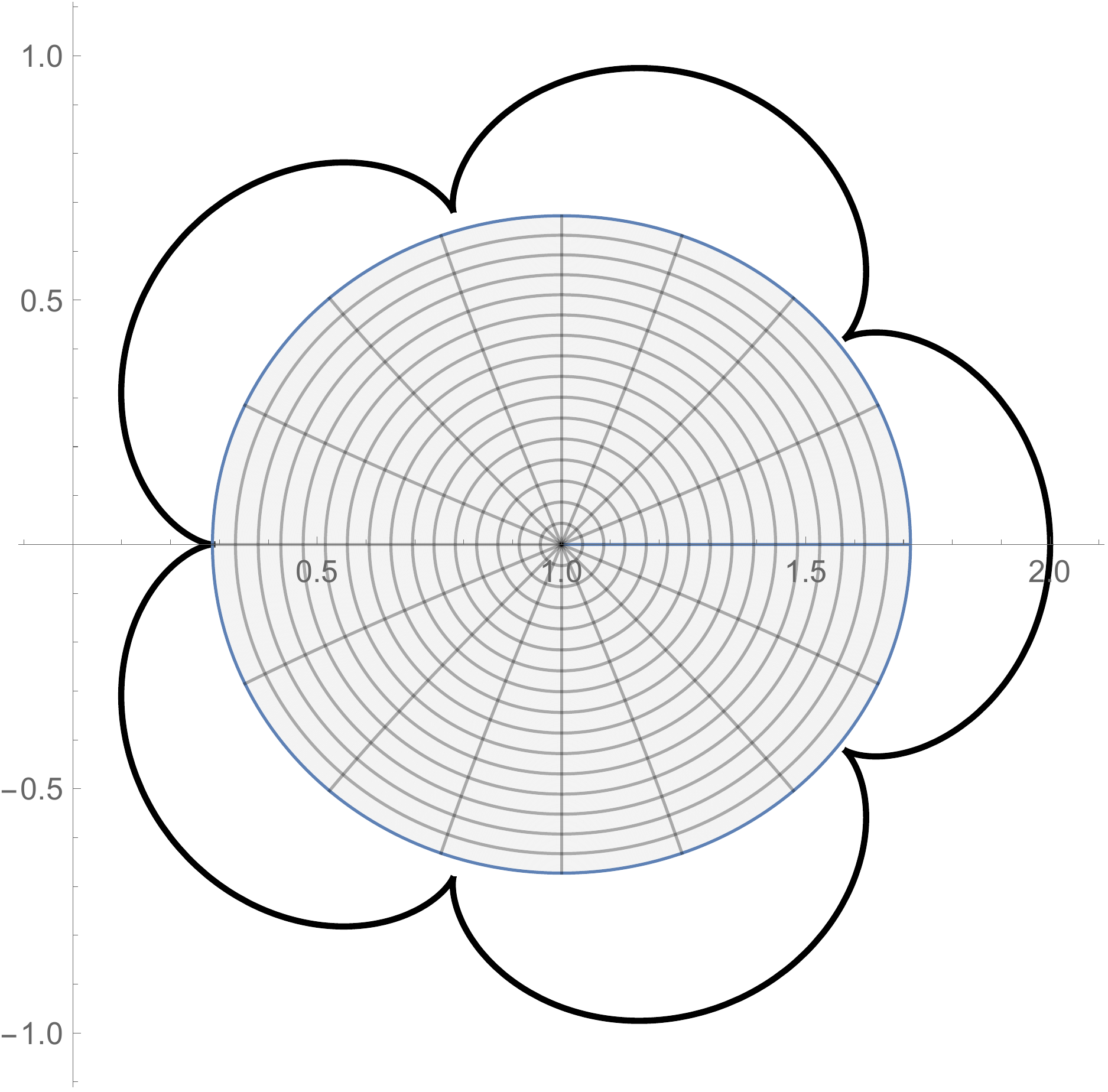}}\hspace{10pt}
		\subfigure[$\mathcal{F}_2$]{\includegraphics[width=1.5in]{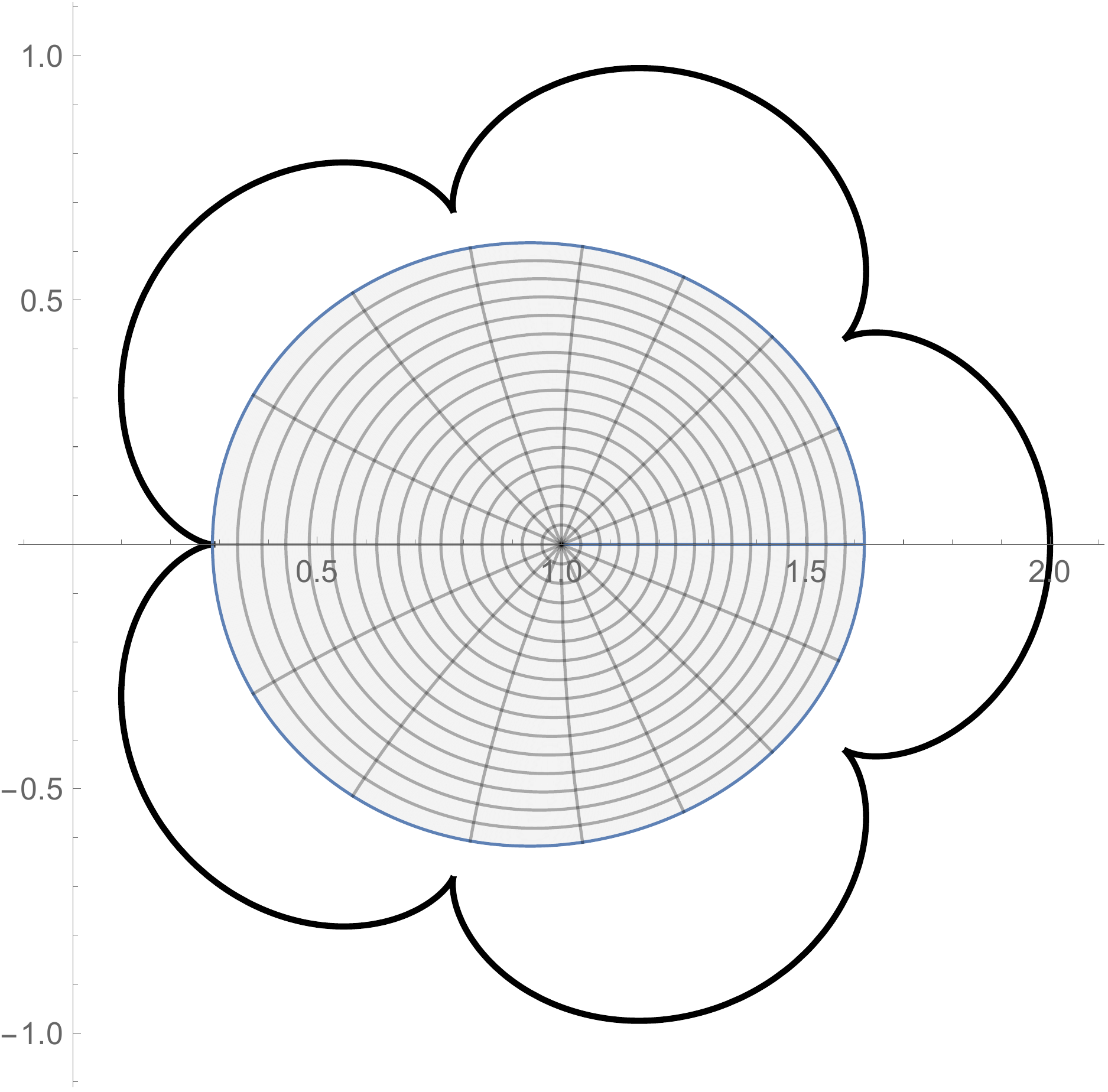}}\hspace{10pt}
			\caption{$\Snl-$radius for class $\mathcal{W}$, $\mathcal{F}_1$ and $\mathcal{F}_2$ }\label{fig2}
	\end{center}
\end{figure}
\begin{theorem}
	For function in class $\mathcal{S}\mathcal{L}^*(\alpha),\,\mathcal{S}^*_{\alpha,e}$ and $\mathcal{S}^*_{EL}$, the following holds:
	\begin{itemize}
		\item[(a)]$\mathcal{R}_{\Snl}\left(\mathcal{S}\mathcal{L}^*(\alpha)\right)=\displaystyle\frac{(n-1)\left(2\alpha(n+1)-n-3\right)}{(n+1)^2(\alpha-1)^2}.$
		\\
		\item[(b)]$\mathcal{R}_{\Snl}\left(\mathcal{S}^*(\sqrt{1+cz})\right)=\displaystyle\frac{n^2+2n-3}{c(n+1)^2},$ for $1-\displaystyle\frac{4}{(n+1)^2}<c\leq1.$
		\\
		\item[(c)]$\mathcal{R}_{\Snl}\left(\mathcal{S}^*_{\alpha,e}\right)=\left|\log\left(\displaystyle\frac{1+\gamma^n+n+\gamma n+\alpha+n\alpha}{(n+1)(1-\alpha)}\right)\right|.$
		\\
		\item[(d)] $\mathcal{R}_{\Snl}(\mathcal{S}^*_{EL})= \gamma^n +n\gamma+(n+1)\alpha-\phi(\delta)(1+n)(1-\alpha),$ where $\phi(z)$ gives the principal solution for $w$ in $z=we^w$ and $\delta=\alpha \exp((\gamma^n +n\gamma+\alpha(1+n))/(n+1)(1-\alpha))/(1-\alpha),$
	\end{itemize}
for $\gamma =e^{i\pi/(n-1)}$ and $0\leq \alpha<1.$ All bounds are sharp.
\end{theorem}
\begin{proof}
		(a) Let $f\in\mathcal{S}\mathcal{L}^*(\alpha).$ Then $zf'(z)/f(z)\prec \alpha+(1-\alpha)\sqrt{1+z}.$ The image of disk $|z|<r$ under the function $zf'(z)/f(z)$ lies inside the domain $\varphi_{n\mathcal{L}}(\disc)$ if
	\begin{align*}
		\left|\frac{zf'(z)}{f(z)}-1\right|&\leq \left|\alpha+(1-\alpha)\sqrt{1+z}-1\right|\\
		&\leq (1-\alpha)\left(1-\sqrt{1-r}\right)\\
		&\leq 1-\frac{2}{n+1}.
	\end{align*}
	This holds for $r\leq (n-1)\left(2\alpha(n+1)-n-3\right)/((n+1)^2.(\alpha-1)^2)$. The result is sharp for the function $f_{n\mathcal{L}}(z)$ given by (\ref{eqn4}). Further,
	\[\frac{zf_{n\mathcal{L}}'(z)}{f_{n\mathcal{L}}(z)}=\frac{2}{n+1}=\varphi_{n\mathcal{L}}(-1),\]
	for $z=(n-1)\left(2\alpha(n+1)-n-3\right)/((n+1)^2.(\alpha-1)^2)$. For $\alpha=0,$ the sharpness is shown in Figure \ref{fig3}(a).
	\par (b) For $0<c\leq 1-4/(n+1)^2,$ the $\Snl-$radius for the class $\mathcal{S}^*(\sqrt{1+cz})$ is 1 by Theorem \ref{inclusion}(d). Let us now assume that $1-4/(n+1)^2<c\leq1.$ Since $f\in\mathcal{S}^*(\sqrt{1+cz}),$ we have
	\[\left|\frac{zf'(z)}{f(z)}-1\right|\leq 1-\sqrt{1-cr}.\]
	By using \lemref{lemma}, we get $1-\sqrt{1-cr}\leq 2/(n+1)$ and this simplifies to $r\leq(n^2+2n-3)/(c(n+1)^2).$
	\par (c) For $f\in\mathcal{S}^*_{\alpha,e}$, we compute the radius by considering the geometries of the domains. The image of disk $|z|<r$ under the function $zf'(z)/f(z)$ lies inside the domain $\varphi_{n\mathcal{L}}(\disc)$ if $r\leq r_1,$ where $r_1$ is the absolute value of the solution of the equation $\alpha+(1-\alpha)e^r=\varphi_{n\mathcal{L}}(e^{i\pi/(n-1)}).$ A direct computation gives
	\[r_1=\left|\log\left(\displaystyle\frac{1+\gamma^n+n+\gamma n+\alpha+n\alpha}{(n+1)(1-\alpha)}\right)\right|.\]
	Clearly, the result is sharp and can be seen from Figure \ref{fig3}(b) for the particular case $\alpha=0$.
	\par (d) Similarly, for this class, the $\Snl-$radius is obtained by solving the equation $\alpha e^r+(1-\alpha)(1+r)=\varphi_{n\mathcal{L}}(e^{i\pi/(n-1)}) $ for $r.$ This gives that the desired result holds for $r\leq\mathcal{R}_{\Snl}(\mathcal{S}^*_{EL})$. For $a\alpha=1/2,$ the sharpness is shown in Figure \ref{fig3}(c).
\end{proof}
\begin{figure}[h]
	\begin{center}
		\subfigure[$\mathcal{S}\mathcal{L}^*(0)$]{\includegraphics[width=1.5in]{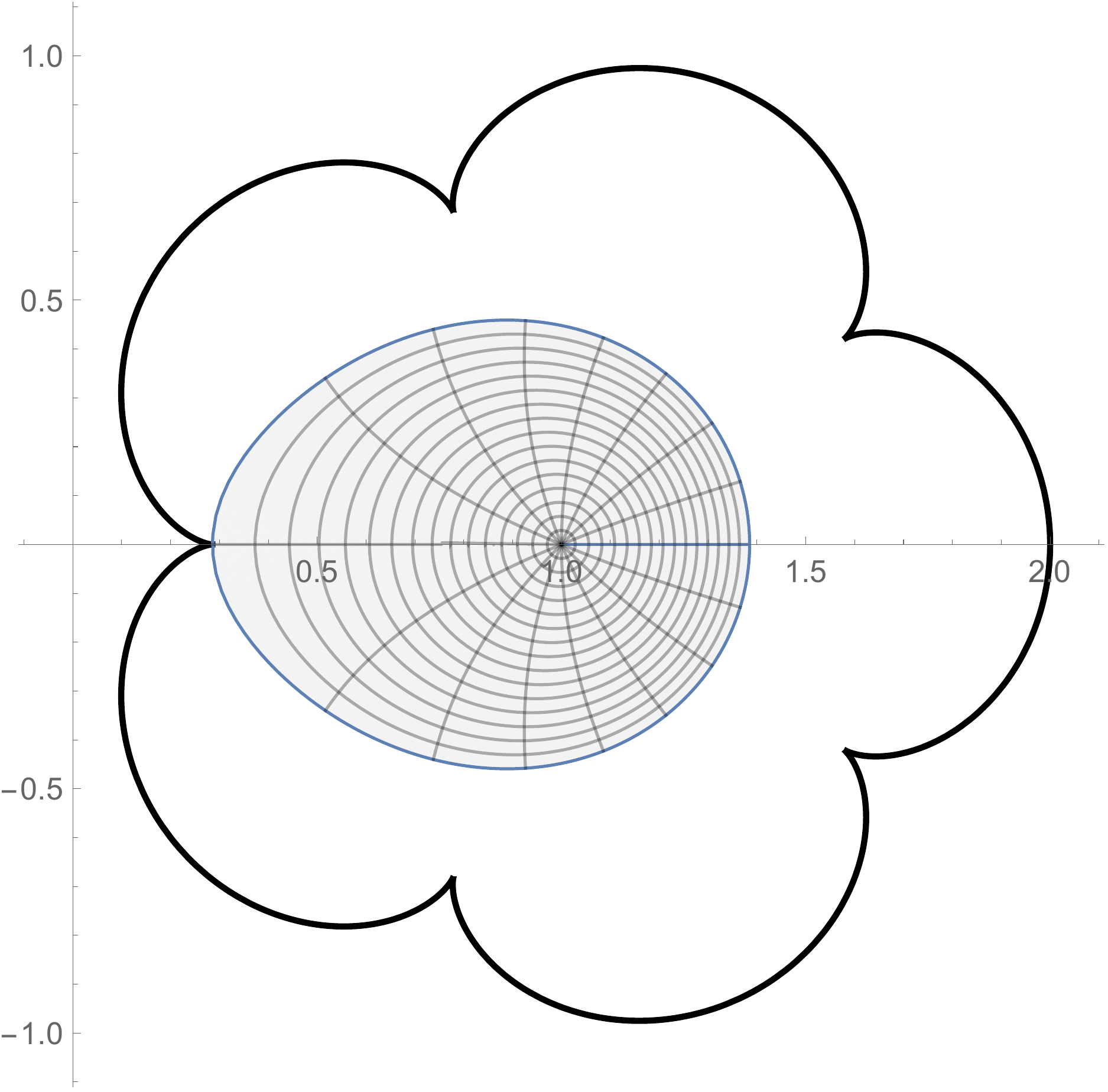}}\hspace{10pt}
		\subfigure[$\mathcal{S}^*_{0,e}$]{\includegraphics[width=1.5in]{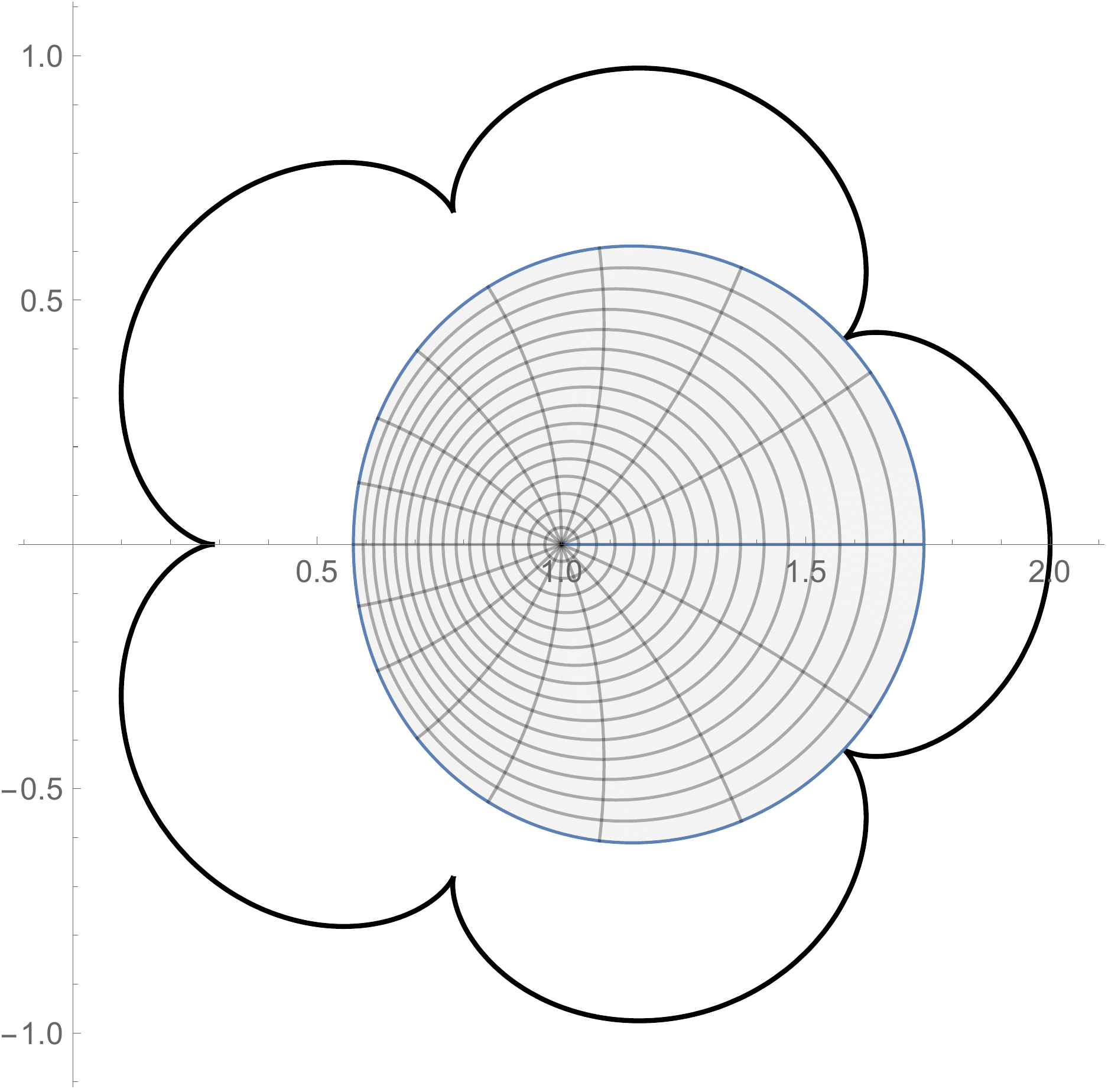}}\hspace{10pt}
		\subfigure[$\mathcal{S}^*_{EL}(\alpha=1/2)$]{\includegraphics[width=1.5in]{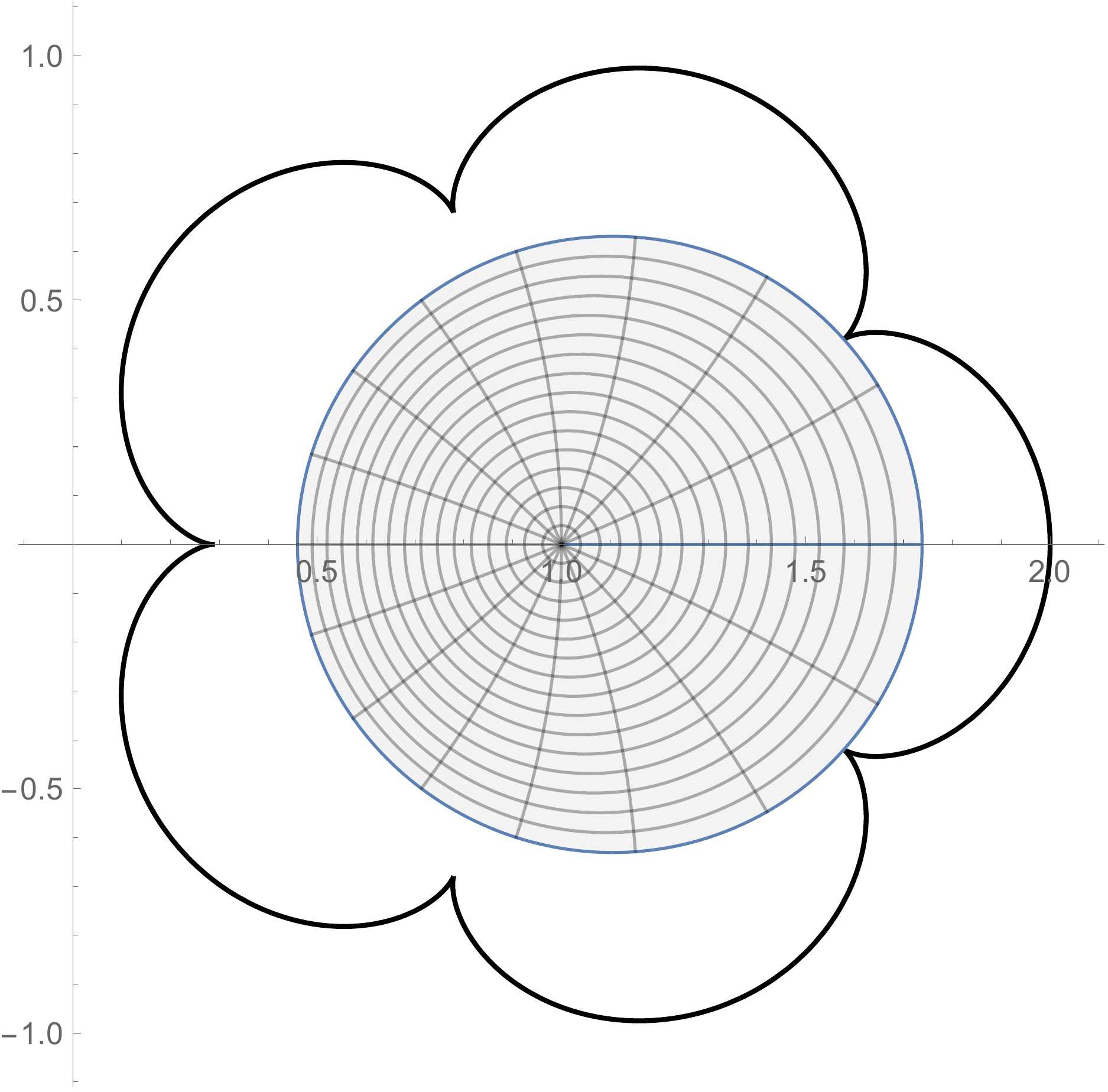}}\hspace{10pt}
		\caption{$\Snl-$radius of classes $\mathcal{S}\mathcal{L}^*(0)$, $\mathcal{S}^*_{0,e}$ and $\mathcal{S}^*_{EL}(\alpha=1/2)$}\label{fig3}
	\end{center}
\end{figure}

\begin{theorem}
	The sharp $\Snl-$radius for various Ma-Minda type subclasses of starlike functions is given by
	\begin{itemize}
		\item[(a)]$\mathcal{R}_{\Snl}\left(\mathcal{S}^*_C\right)=\left|\displaystyle\frac{\sqrt{16(n+1)^2+8(n+1)(3\gamma^n+3n\gamma)}}{4(n+1)}-1\right|$
		\item[(b)]$\mathcal{R}_{\Snl}(\mathcal{S}^*_{\leftmoon})=\displaystyle\frac{1}{2}\left|\frac{2\gamma^{n-1}(n+1)+\gamma^{2n-1}+2\gamma^n+}{(1+n)(1+\gamma^n+n+n\gamma)}\right|$
		\item[(c)] $\mathcal{R}_{\Snl}(\mathcal{S}^*_R)=|R_1|/(2(n+1))$, where
		\begin{multline}\label{eqn5}
			R_1=(1+\sqrt{2})(\gamma^n+1+n+\gamma)+\\\sqrt{(n+1)\left(3+2\sqrt{2}\right)\left(\gamma^n+n\gamma\right)+\left(1+\sqrt{2}\right)^2(1+n+n\gamma+\gamma^n)^2}.
			\end{multline}
		\item[(d)] $\mathcal{R}_{\Snl}(\mathcal{S}^*_{RL})=\displaystyle\frac{(n-1)(1-\sqrt{2}+3n+\sqrt{2}n)}{11-7\sqrt{2}+6n-6\sqrt{2}n+3n^2+\sqrt{2}n^2}$
	
		\item[(e)] $\mathcal{R}_{\Snl}(\mathcal{S}^*_{lim})=\left|\sqrt{\displaystyle\frac{2(1+\gamma^n+n(1+\gamma))}{n+1}}-\sqrt{2}\right|$
		\item[(f)] $\mathcal{R}_{\Snl}(\mathcal{S}^*(1+ze^z))=\left|\phi\left(\displaystyle\frac{\gamma^n +n\gamma}{1+n}\right)\right|,$ where $\phi(z)$ is given as in Theorem 5.5(c),\\
		\item[(g)] $\mathcal{R}_{\Snl}(\mathcal{S}^*_{car})=\left|-1+\sqrt{\displaystyle\frac{1+2\gamma^n+n+2\gamma n}{n+1}}\right|.$
	\end{itemize}
	where $\gamma=e^{i\pi/(n-1)}.$
\end{theorem}
\begin{proof}
(a) Let $f\in\mathcal{S}^*_C$. Then $zf'(z)/f(z)\prec 1+4z/3+2z^2/3.$ By geometric interpretation, the cardiod $(9u^2+9v^2-18u+5)^2-16(9u^2+9v^2-6u+1)=0$ lies in the domain $\varphi_{n\mathcal{L}}(\disc)$ for $r\leq r_3,$  where $r_3$ is the absolute solution of the equation
\[\frac{4r}{3}+\frac{2r^2}{3}=\frac{n e^{i\pi/(n-1)}}{n+1}+\frac{e^{in\pi/(n-1)}}{n+1},\]
given by
\[r_3=\left|-1+\frac{\sqrt{4(n+1)^2+6(n+1)(\gamma^n+n\gamma)}}{n+1}\right|,\]
for $\gamma=e^{i\pi/(n-1)}.$ Sharpness can be seen from Figure \ref{fig4}(a).
\par (b) Proceeding in a similar way, the necessary condition for the lune $|w^2-1|<2|w|,\,w\in\mathbb{C},$ to lie inside the domain $\varphi_{n\mathcal{L}}(\disc)$ is obtained by solving $r+\sqrt{1+r^2}=\varphi_{n\mathcal{L}}(e^{i\pi/(n-1)})$ for $r$. A direct simplification yields the $\Snl-$radius for this class is $r_4$ which is exactly $\mathcal{R}_{\Snl}(\mathcal{S}^*_{\leftmoon})$ (See Figure \ref{fig4}(b)).
\par (c) Similarly, for this class, the $\Snl-$ radius is computed by solving equation
\[\frac{(\sqrt{2}+1)^2+r^2}{(\sqrt{2}+1)(\sqrt{2}+1-r)}= 1+\frac{ne^{i\pi/(n-1)}}{n+1}+\frac{e^{in\pi/(n-1)}}{n+1},\]
for $r$. This gives $r\leq |R_1|/(2(n+1)),$ where $R_1$ is given by (\ref{eqn5}). The sharpness for this class is depicted in Figure \ref{fig4}(c).
\begin{figure}[h]
	\begin{center}
		\subfigure[$\mathcal{S}^*_{car}$]{\includegraphics[width=1.5in]{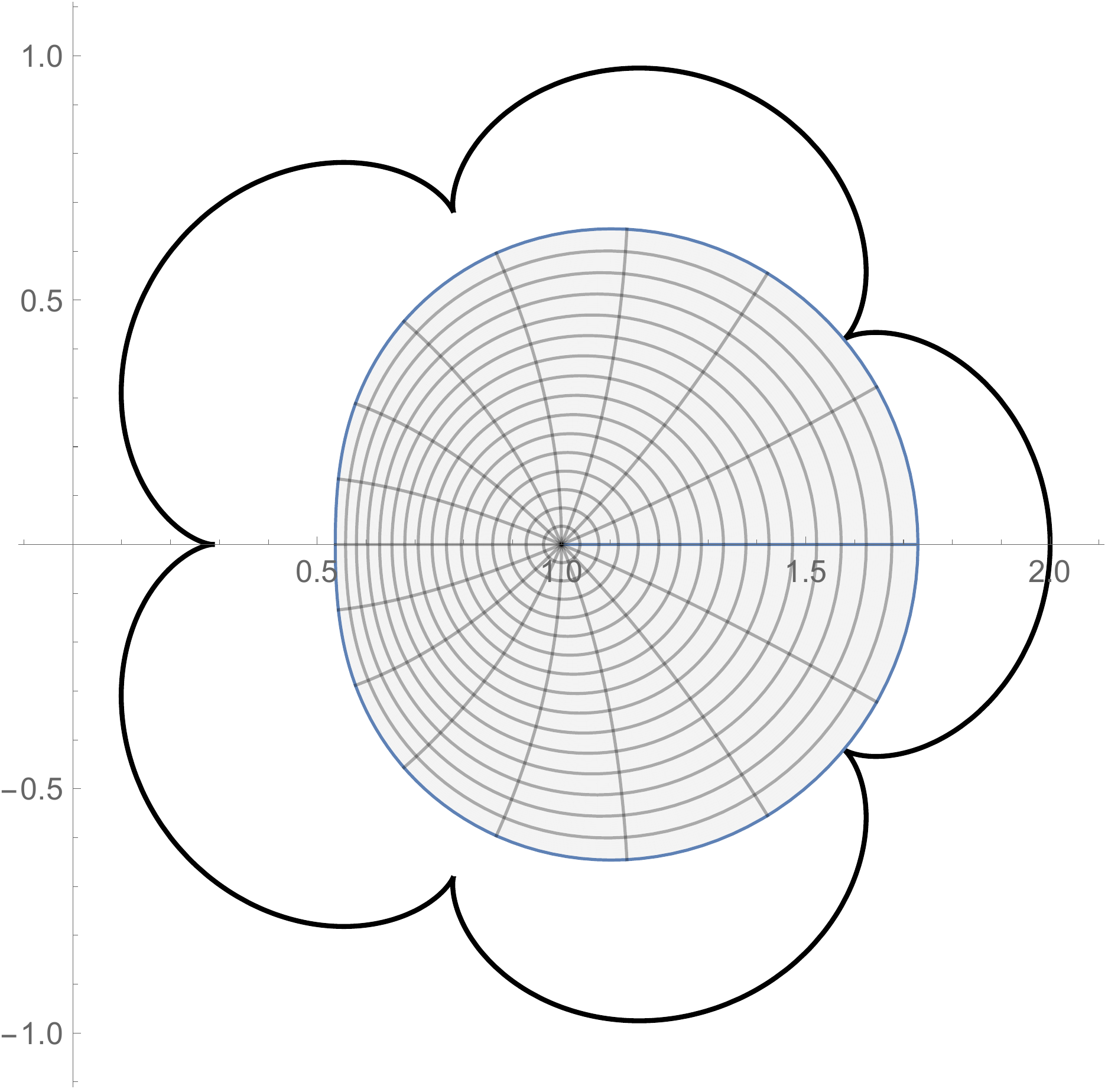}}\hspace{10pt}
		\subfigure[$\mathcal{S}^*_{\leftmoon}$]{\includegraphics[width=1.5in]{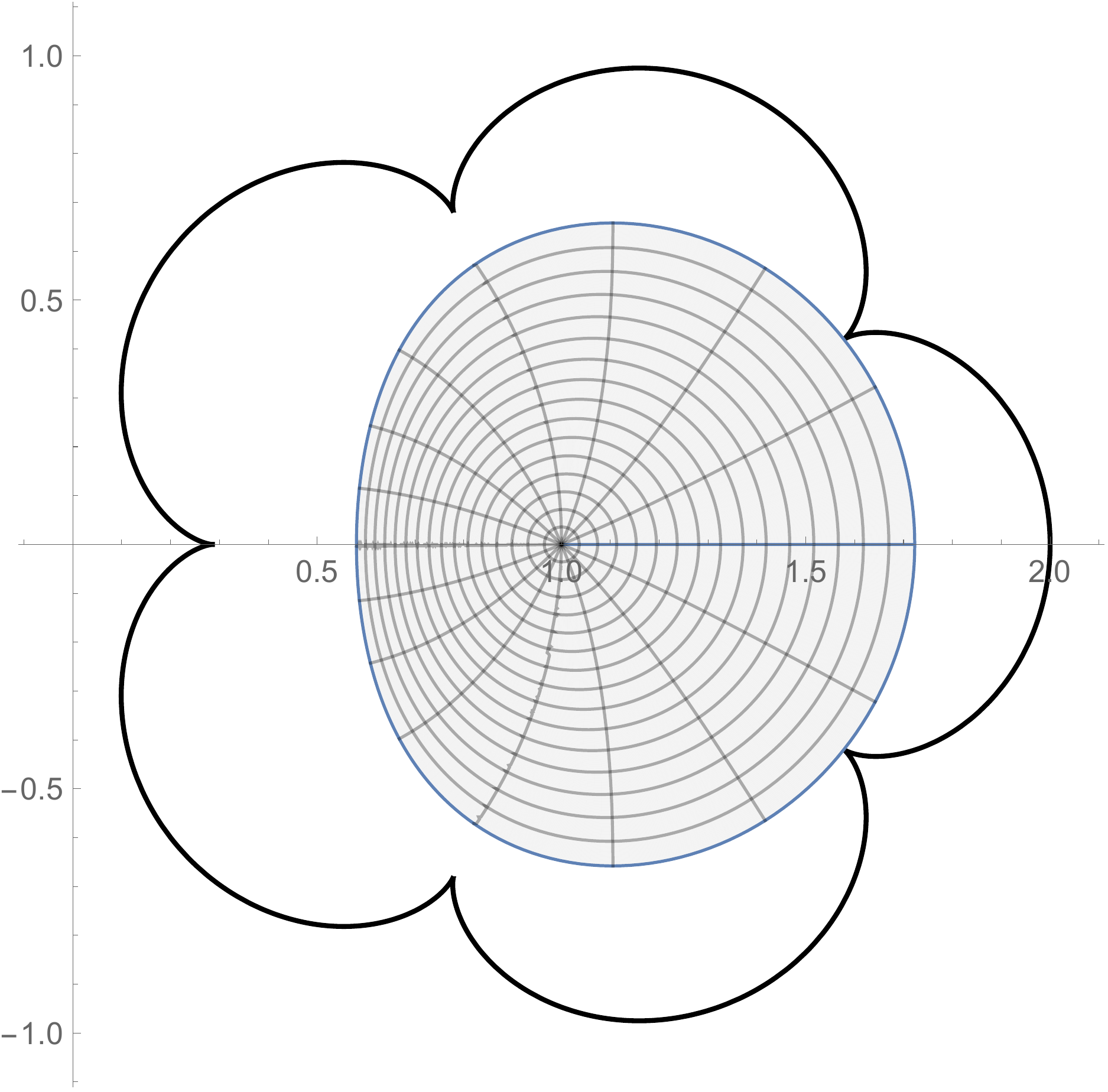}}\hspace{10pt}
		\subfigure[$\mathcal{S}^*_R$]{\includegraphics[width=1.5in]{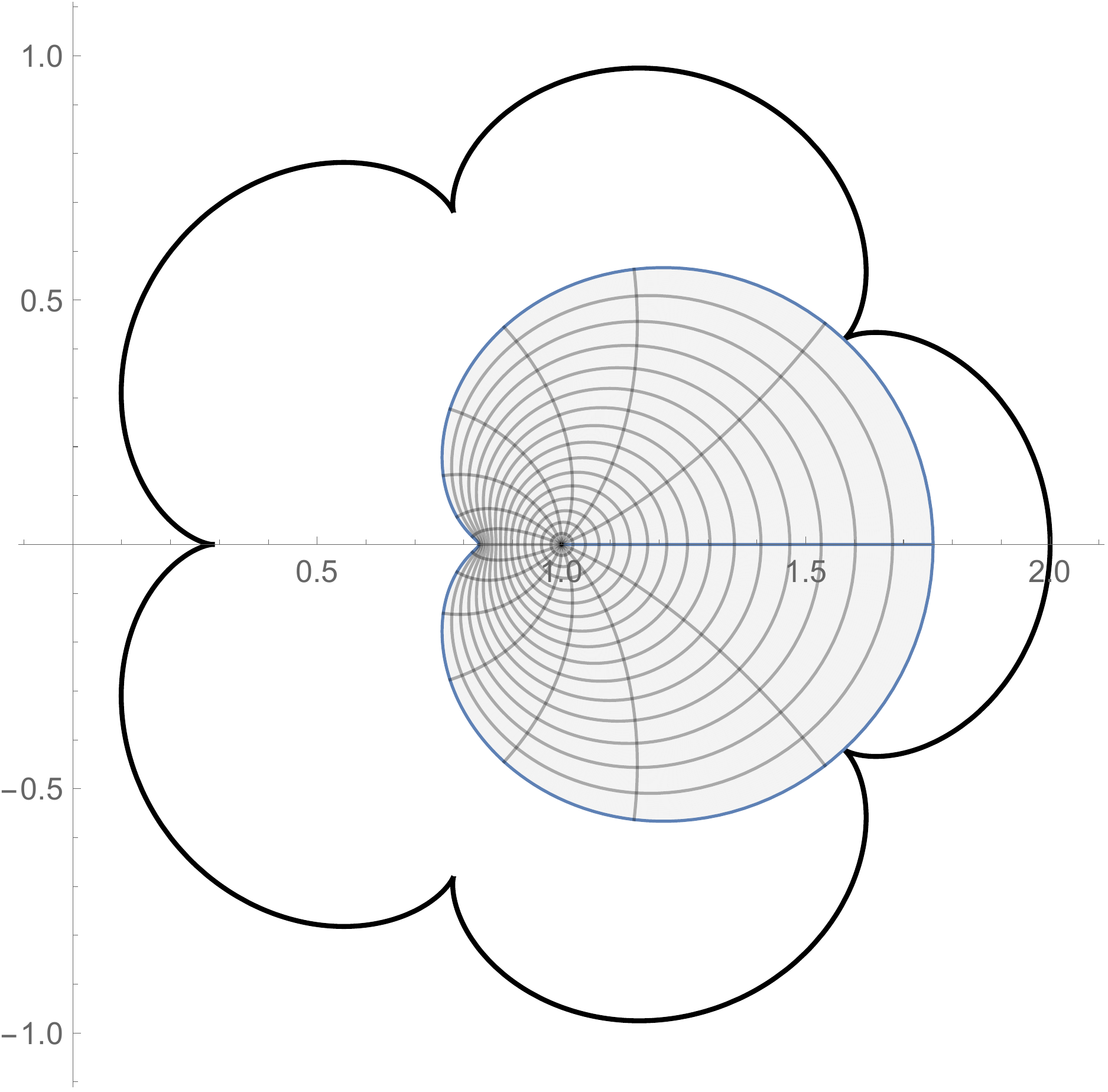}}\hspace{10pt}
			\caption{$\Snl-$radius for $\mathcal{S}^*_{car}$, $\mathcal{S}^*_{\leftmoon}$ and $\mathcal{S}^*_R$}\label{fig4}
	\end{center}
\end{figure}
\par (d) Let $f\in\mathcal{S}^*_{RL}.$ Then the image of the disk $|z|<r$ under the function $zf'(z)/f(z)$ lies in the domain $\varphi_{n\mathcal{L}}(\disc)$ for $r\leq r_4,$ where $r_4$ is the solution of the equation
\[\sqrt{2}-(\sqrt{2}-1)\sqrt{\frac{1-r}{1+2(\sqrt{2}-1)r}}=\frac{2}{n+1},\]
by geometries of the domains. The result is sharp for the function $f_4$ defined such that
\[\frac{zf_4'(z)}{f_4(z)}=\sqrt{2}-(\sqrt{2}-1)\sqrt{\frac{1-z}{1+2(\sqrt{2}-1)z}}.\]
It is clear that
\[\left.\frac{zf_4'(z)}{f_4(z)}\right|_{z=r_4}=\frac{2}{n+1}=\varphi_{n\mathcal{L}}(-1),\]
as illustrated in Figure \ref{fig5}(a).
\par (e) To compute this radius, we solve the following equation for r
\[\sqrt{2}r+\frac{r^2}{2}=\frac{n e^{i\pi/(n-1)}}{n+1}+\frac{e^{in\pi/(n-1)}}{n+1}.\]
Thus, $\Snl-$radius for the class $\mathcal{S}^*_{lim}$ is given by $\mathcal{R}_{\Snl}(\mathcal{S}^*_{lim})$ and sharpness is shown in Figure \ref{fig5}(b).
\par (f) The $\Snl-$radius for the class $\mathcal{S}^*(1+ze^z)$ by solving the equation $1+re^r=\varphi_{n\mathcal{L}}(e^{i\pi/(n-1)})$ for $r$. This gives the desired result holds for $r\leq \phi((\gamma^n+n\gamma)/(n+1))$ where the function $\phi$ is defined in Theorem 5.5(c). (See Figure \ref{fig5}(c)).
\par (g) Lastly, to compute the $\Snl-$radius for this class we will consider the cusp at the angle $\pi/(n-1)$ and obtain the equation
\[(n+1)(2r+r^2)=2(ne^{i\pi/(n-1)}+e^{in\pi/(n-1)}).\]
On solving above equation, we get the desired result holds for $r\leq \mathcal{R}_{\Snl}(\mathcal{S}^*_{car}),$ given in the statement of the theorem. Sharpness is depicted in Figure \ref{fig5}(d).
\end{proof}
\begin{figure}[h]
	\begin{center}
		\subfigure[$\mathcal{S}^*_{RL}$]{\includegraphics[width=1.2in]{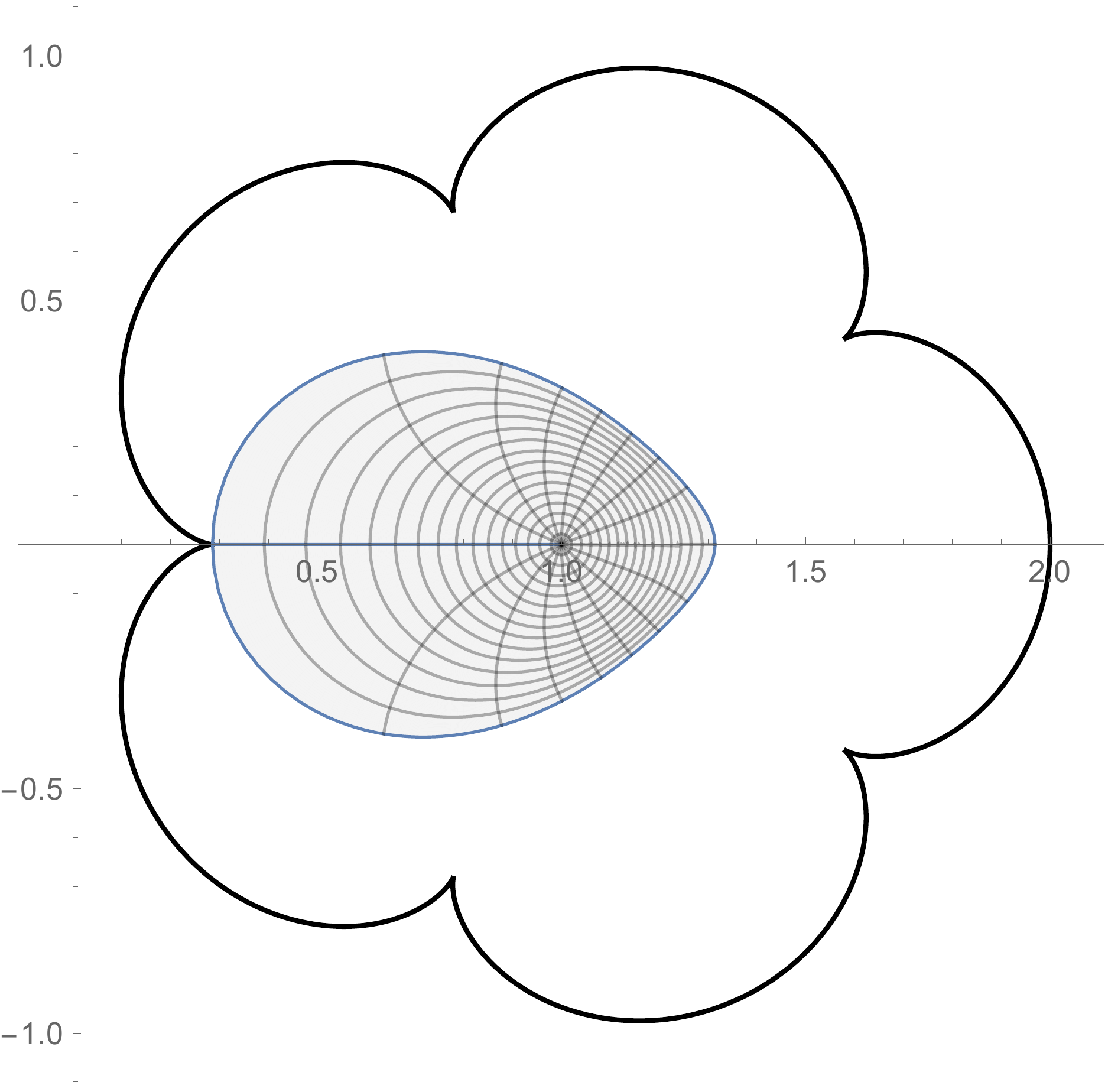}}\hspace{10pt}
		\subfigure[$\mathcal{S}^*_{lim}$]{\includegraphics[width=1.2in]{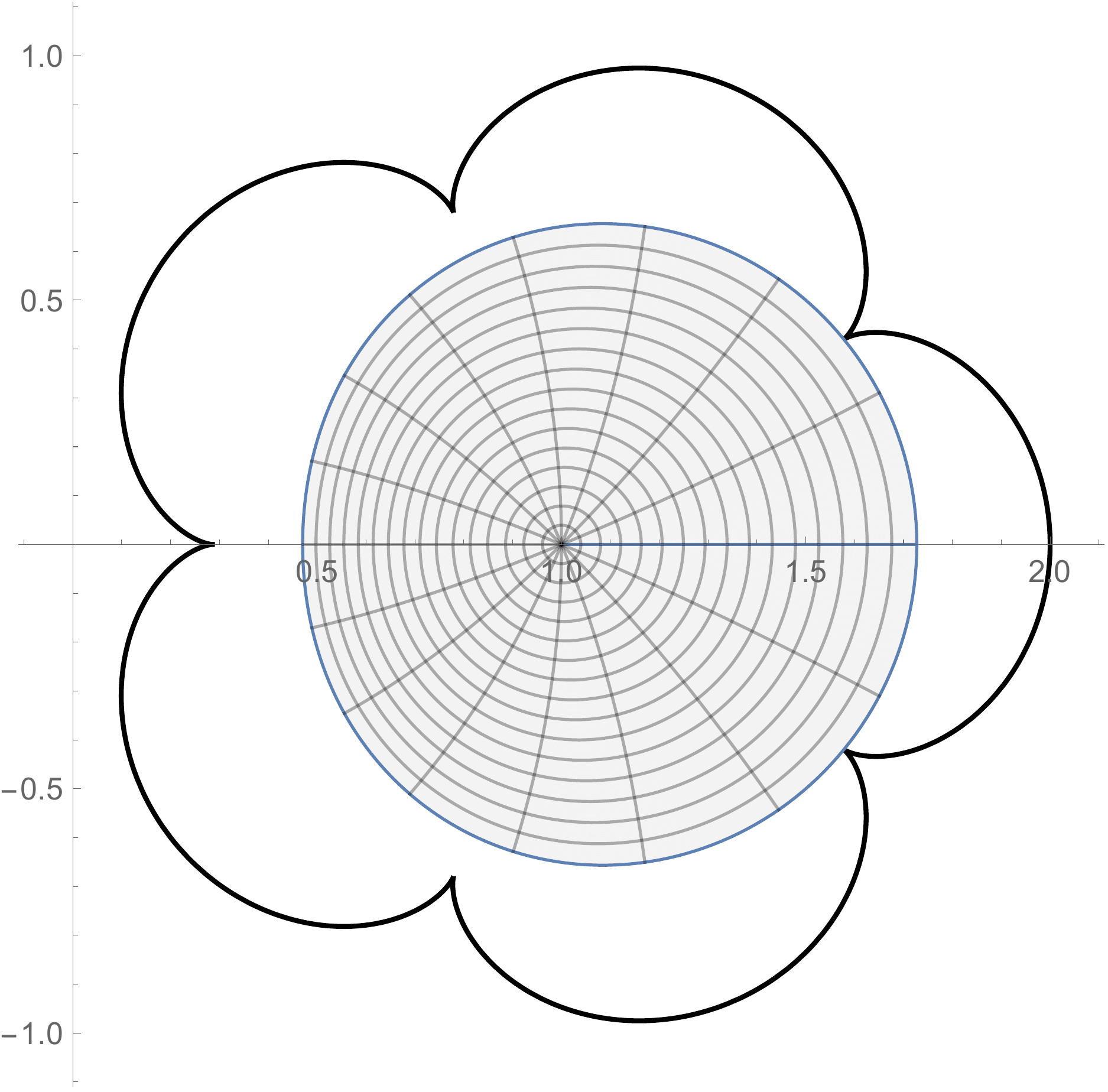}}\hspace{10pt}
		\subfigure[$\mathcal{S}^*(1+ze^z)$]{\includegraphics[width=1.2in]{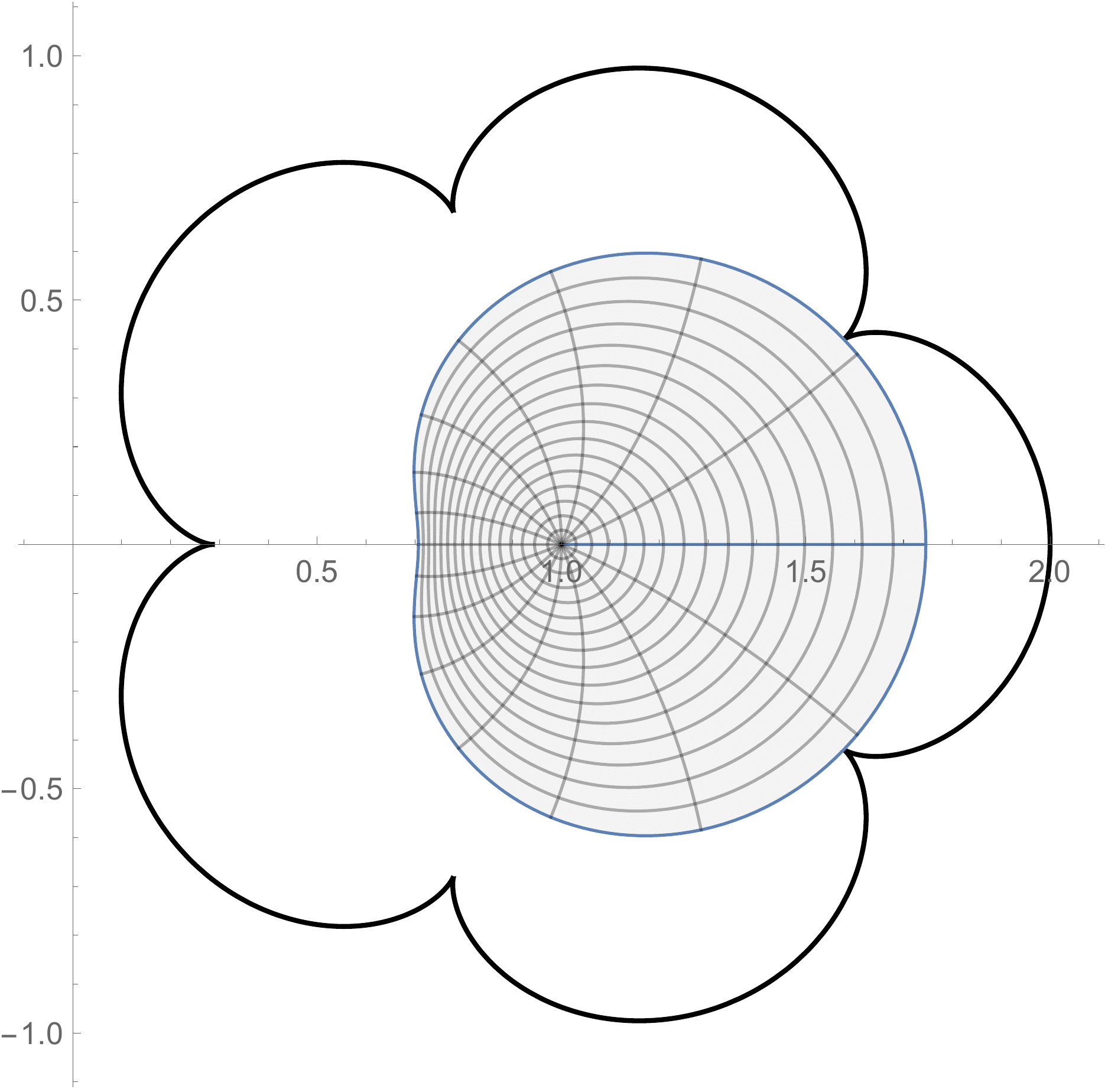}}\hspace{10pt}
			\subfigure[$\mathcal{S}^*_{car}$]{\includegraphics[width=1.2in]{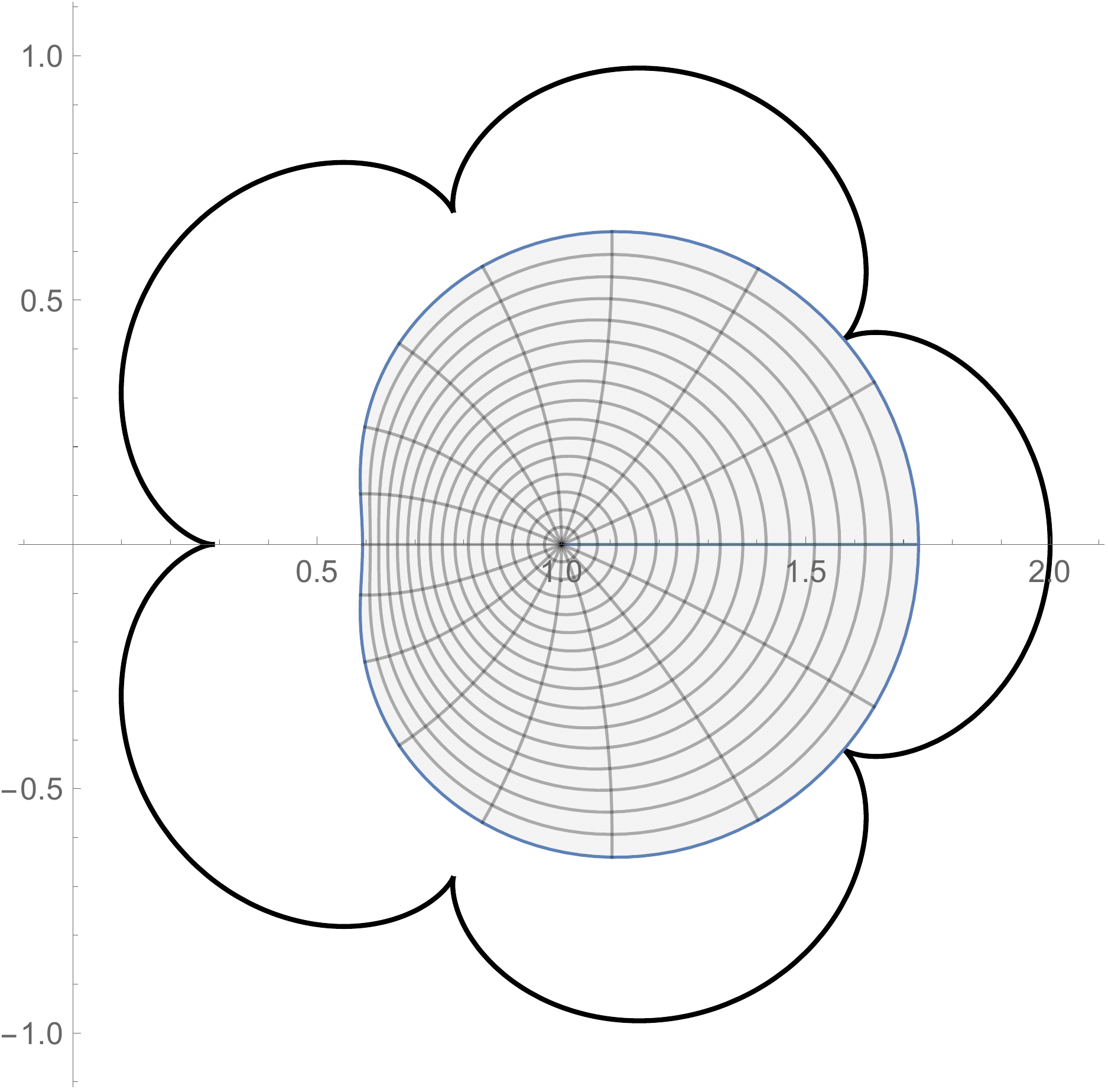}}\hspace{10pt}
			\caption{Sharpness of $\Snl$ radii for various classes }\label{fig5}
	\end{center}
\end{figure}
\begin{theorem}Let $n=2k,\,k\in\mathbb{N}.$ Then
	\begin{itemize}
		\item[(a)] The $\Snl-$radius for the class $\mathcal{S}^*_{sin}$ is given by
		$\mathcal{R}_{\Snl}(\mathcal{S}^*_{sin})=|R_2|,$ where
		\begin{align}
			R_2=\begin{cases}
				\sin^{-1}\left(\displaystyle\frac{\gamma^k(\gamma^{(n-1)}+n)}{n+1}\right),& k \text{ is odd},\\
				\\
				\sin^{-1}\left(\displaystyle\frac{\gamma^{(k-1)}(\gamma^{(n-1)}+n)}{n+1}\right),& k \text{ is even}.
			\end{cases}
		\end{align}
	\item[(b)] The $\Snl-$radius for the class $\mathcal{S}^*_{ne}$ is given by
		$\mathcal{R}_{\Snl}(\mathcal{S}^*_{ne})=|R_3|,$ where
	\begin{align}\label{eqn5.2}
		R_3=\begin{cases}
			\displaystyle\frac{\left(1+i\sqrt{3}\right)(n+1)}{2^{2/3}\delta_1}+\displaystyle\frac{\left(1-i\sqrt{3}\right)\delta_1}{2^{2/3}(n+1)},& k \text{ is odd},\\
			\\
			\displaystyle\frac{\left(1-i\sqrt{3}\right)(n+1)}{2^{2/3}\delta_2}+\displaystyle\frac{\left(1+i\sqrt{3}\right)\delta_2}{2^{2/3}(n+1)},& k \text{ is even},
		\end{cases}
	\end{align}
where
\begin{align}
	\delta_1&=\left(3(n+1)^2(\gamma^{kn}+\gamma^kn)+(n+1)^2\sqrt{9\gamma^k\left(\gamma^2n+2\gamma^{n+1}+\gamma^2n^2)-4(n+1)^2\right)}\right)^{2/3}\label{eqn4.5}\\
	\delta_2&=\left(3(n+1)^2\left(\gamma^{n(k-1)}+n\gamma^{k-1}\right)+(n+1)^2\sqrt{9\gamma^{k-1}\left(\gamma^{2n}+2\gamma^{n+1}n+\gamma^{2}n^2\right)-4(n+1)^2}\right)^{2/3}\label{eqn4.6}
\end{align}
\end{itemize}
Here $\gamma=e^{i\pi/(n-1)}.$
\end{theorem}
\begin{proof}
	(a) Let $f\in\mathcal{S}^*_{sin}$ and $n=2k,\,k\in\mathbb{N}.$ Let $k$ be odd. In this case, the cusp considered is at the angle $k\pi/(n-1).$ Thus the image of $zf'(z)/f(z)$ under $\disc_r$ lies in the domain $\varphi_{n\mathcal{L}}(\disc)$ for $r\leq |R_2|,$ where
	\[R_2=	\sin^{-1}\left(\displaystyle\frac{\gamma^k(\gamma^{(n-1)}+n)}{n+1}\right),\]
	 is the solution of the equation $\sin r(n+1)=ne^{ik\pi/(n-1)}+e^{ink\pi/(n-1)}.$ Proceeding in a similar way, we will consider the cusp at the angle $(k-1)\pi/(n-1)$ for the case when $k$ is even. The $\Snl-$radius is obtained by solving the equation $\sin r(n+1)=ne^{i(k-1)\pi/(n-1)}+e^{in(k-1)\pi/(n-1)}$ for $r$. This gives $r\leq |R_2|,$ where
	 \[\sin^{-1}\left(\displaystyle\frac{\gamma^{(k-1)}(\gamma^{(n-1)}+n)}{n+1}\right).\]
	\begin{figure}[h]
		\begin{center}
			\subfigure[$n=6$]{\includegraphics[width=2in]{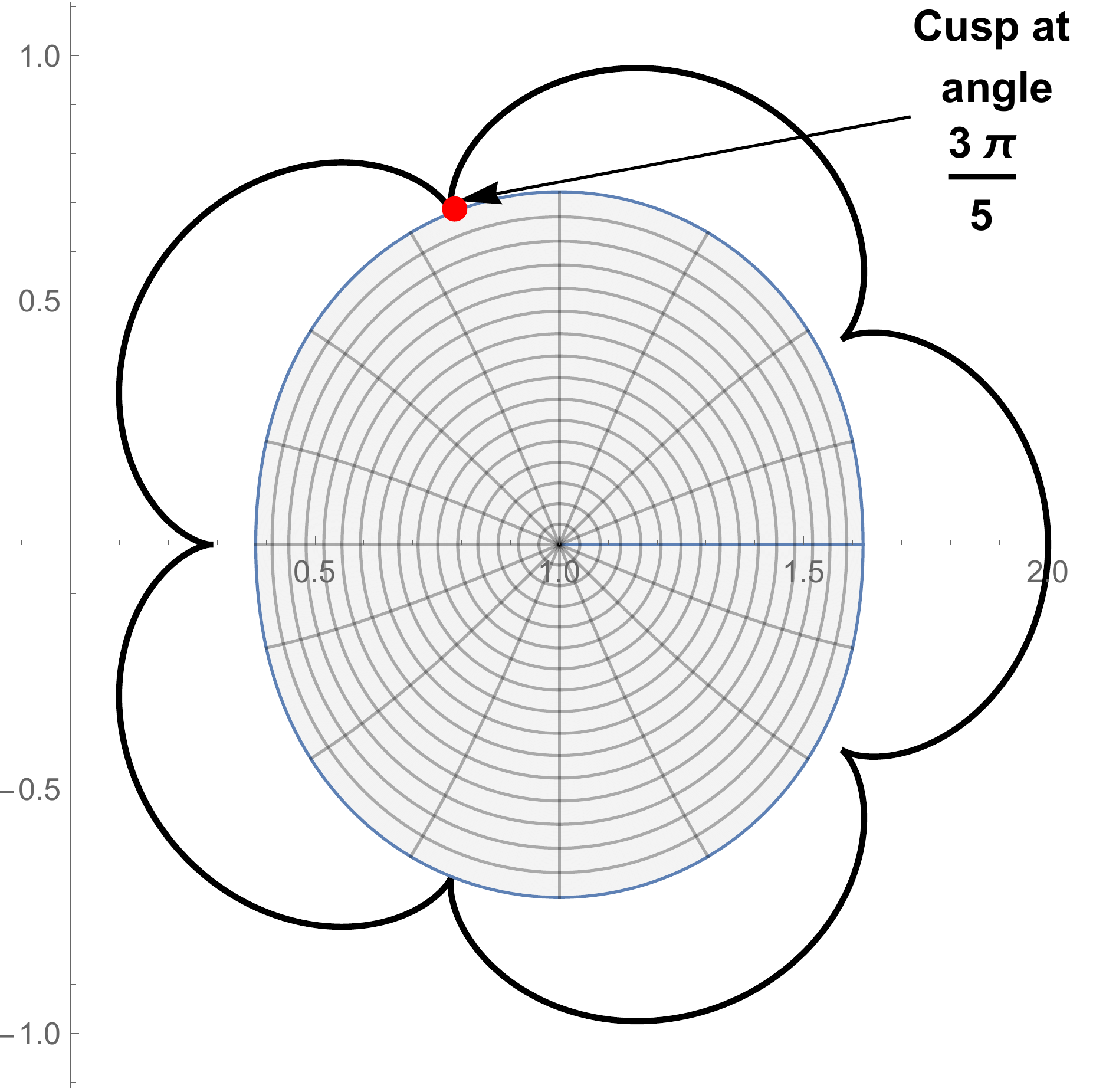}}\hspace{10pt}
			\subfigure[$n=8$]{\includegraphics[width=2in]{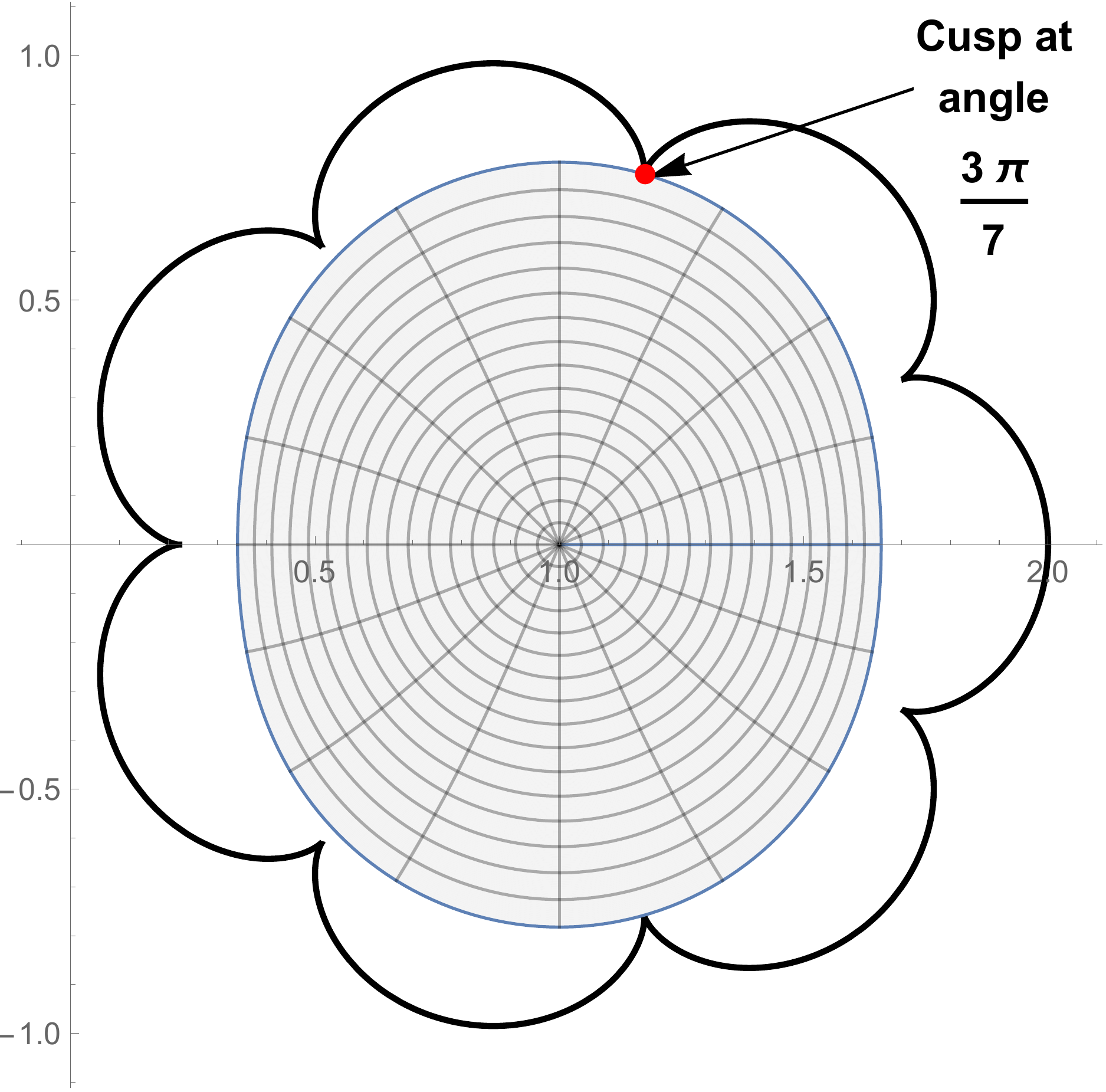}}\hspace{10pt}
			\subfigure[$n=10$]{\includegraphics[width=2in]{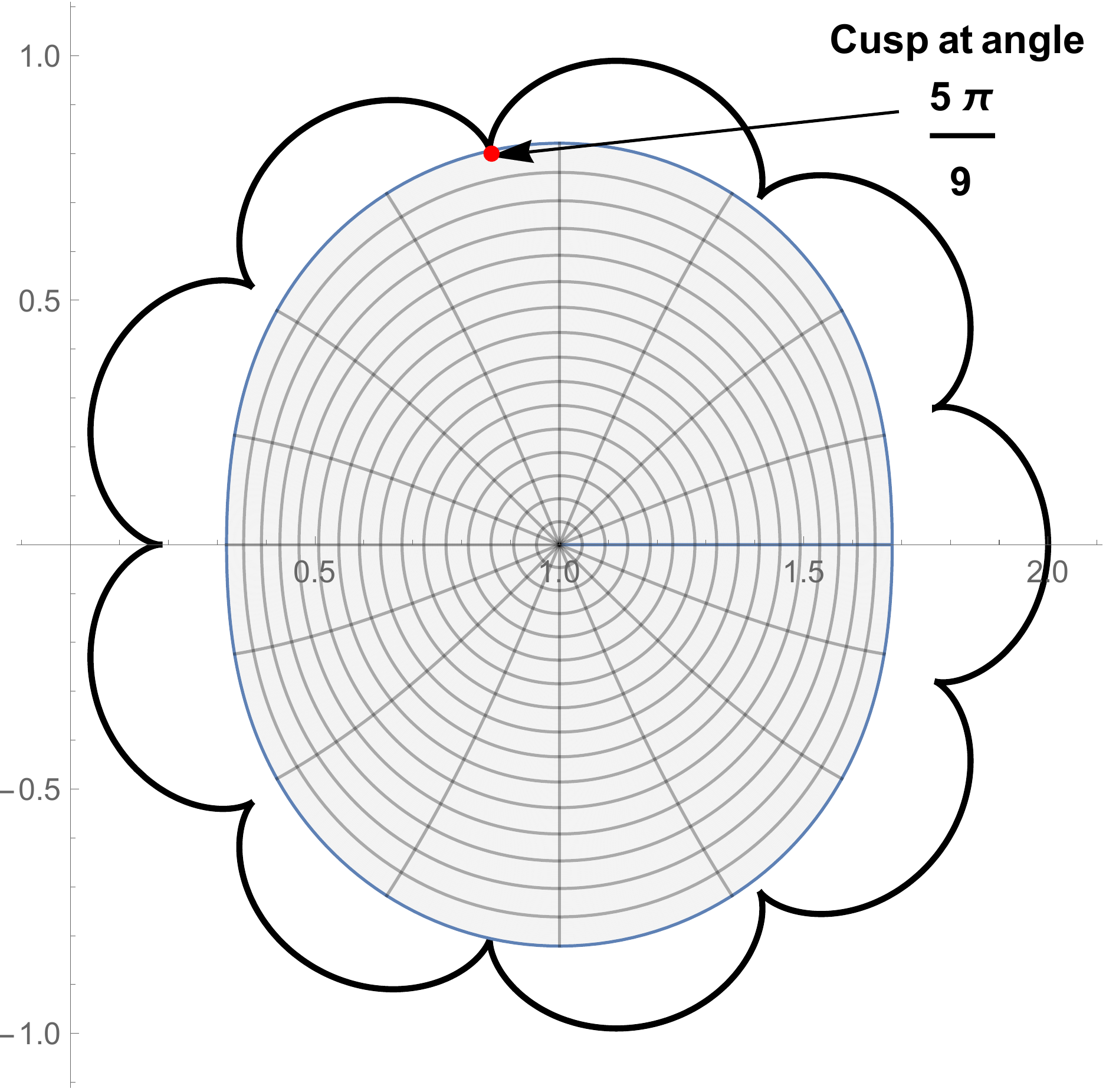}}\hspace{10pt}
			\subfigure[$n=12$]{\includegraphics[width=2in]{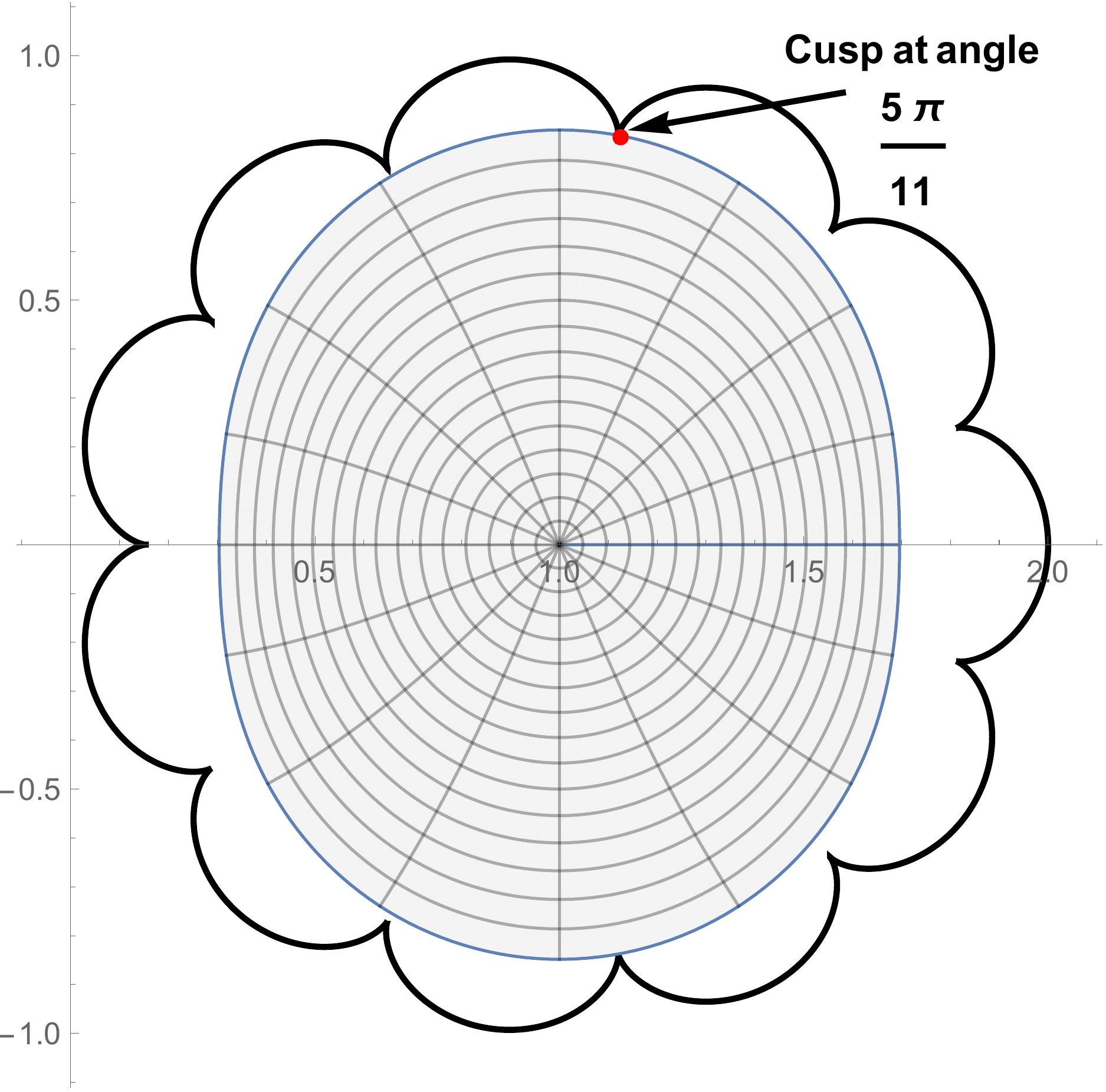}}\hspace{10pt}
			\caption{Image of $\varphi_{sin}(z)$ lying in various polyleaf domain}\label{fig6}
		\end{center}
	\end{figure}
For some choices $n$, sharpness for the above result is depicted in the Figure \ref{fig6}.
	 \par (b) Let $n=2k,\,k\in\mathbb{N}.$ Let us first consider the case when $k$ is odd. In this case, the desired radius is computed by considering the cusp at the angle $k\pi/(n-1).$ Thus, the image of the disk $|z|<r$ under the function $zf'(z)/f(z)$ lies in the domain $\varphi_{n\mathcal{L}}(\disc)$ for $r\leq |R_3|,$ where $R_3$ is the solution of the equation $(n+1)(3r-r^3)=3(ne^{ik\pi/(n-1)}+e^{ink\pi/(n-1)})$ given by
	 \[R_3=\frac{\left(1+i\sqrt{3}\right)(n+1)}{2^{2/3}\delta_1}+\displaystyle\frac{\left(1-i\sqrt{3}\right)\delta_1}{2^{2/3}(n+1)},\]
	 where $\delta_1$ is given by (\ref{eqn4.5}). Let us now assume that $k$ is even. We will consider the cusp at the angle $(k-1)\pi/(n-1).$ In this case, the $\Snl-$radius for the class $\mathcal{S}^*_{ne}$ is computed by solving the equation $(n+1)(3r-r^3)=3(ne^{i(k-1)\pi/(n-1)}+e^{in(k-1)\pi/(n-1)})$ for $r$. This gives $r\leq |R_3|,$ where
	 \[R_3=\frac{\left(1-i\sqrt{3}\right)(n+1)}{2^{2/3}\delta_2}+\displaystyle\frac{\left(1+i\sqrt{3}\right)\delta_2}{2^{2/3}(n+1)},\]
	 where $\delta_2$ is given by (\ref{eqn4.6}). The sharpness is illustrated for some choices of $n$ in the Figure \ref{fig7}.
	\begin{figure}[h]
		\begin{center}
				\subfigure[$n=6$]{\includegraphics[width=2in]{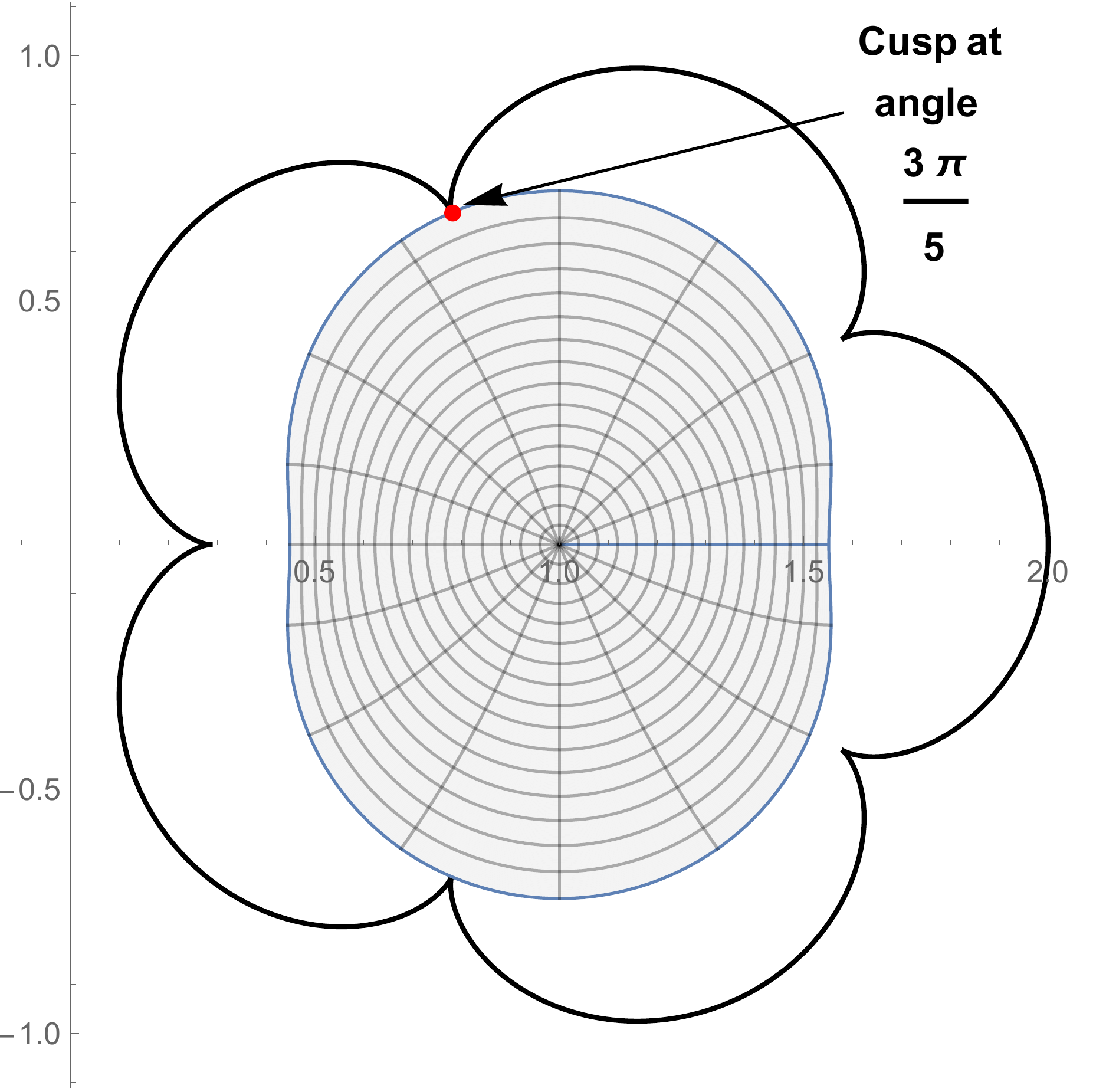}}\hspace{10pt}
					\subfigure[$n=8$]{\includegraphics[width=2in]{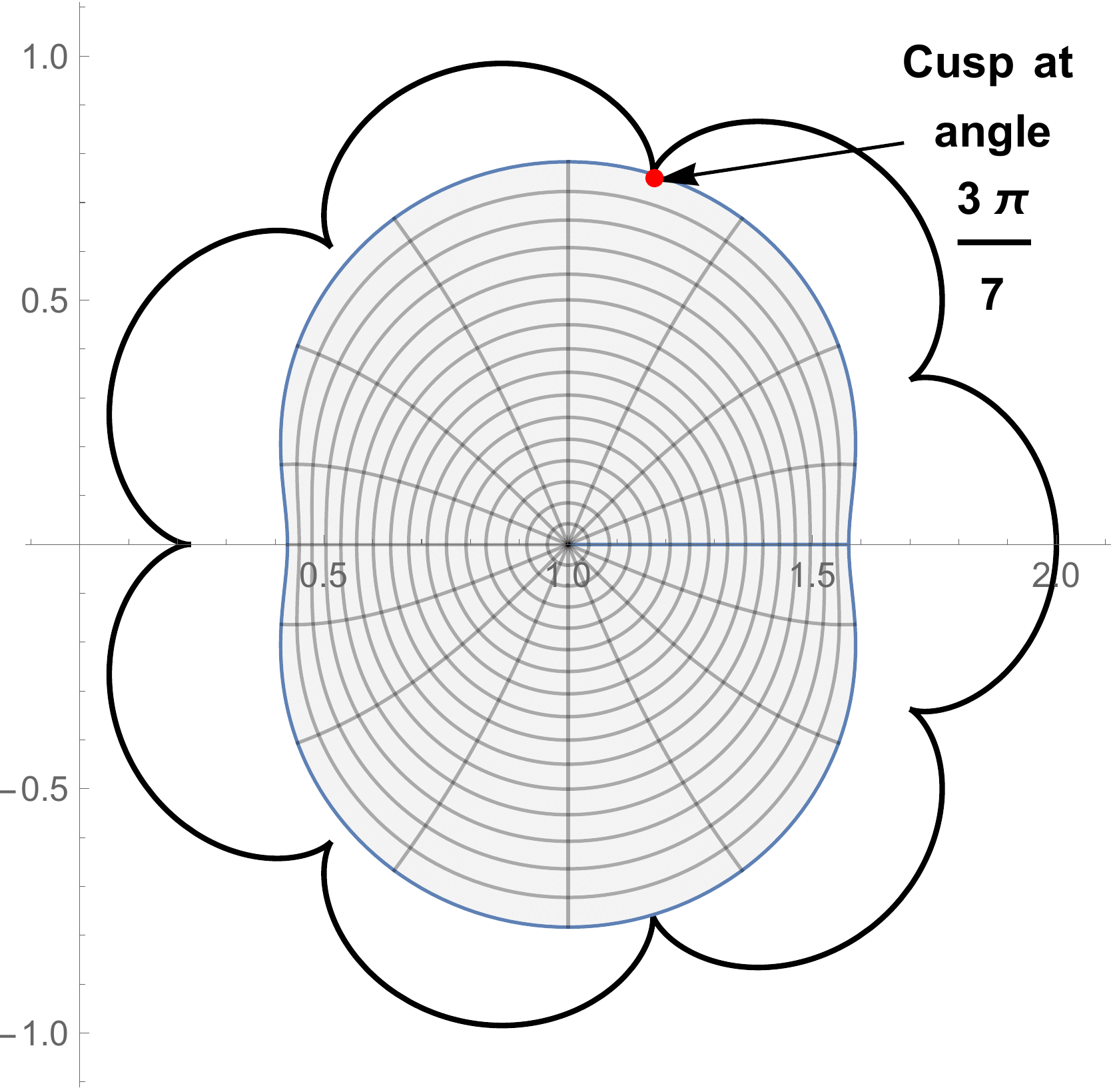}}\hspace{10pt}
					\subfigure[$n=10$]{\includegraphics[width=2in]{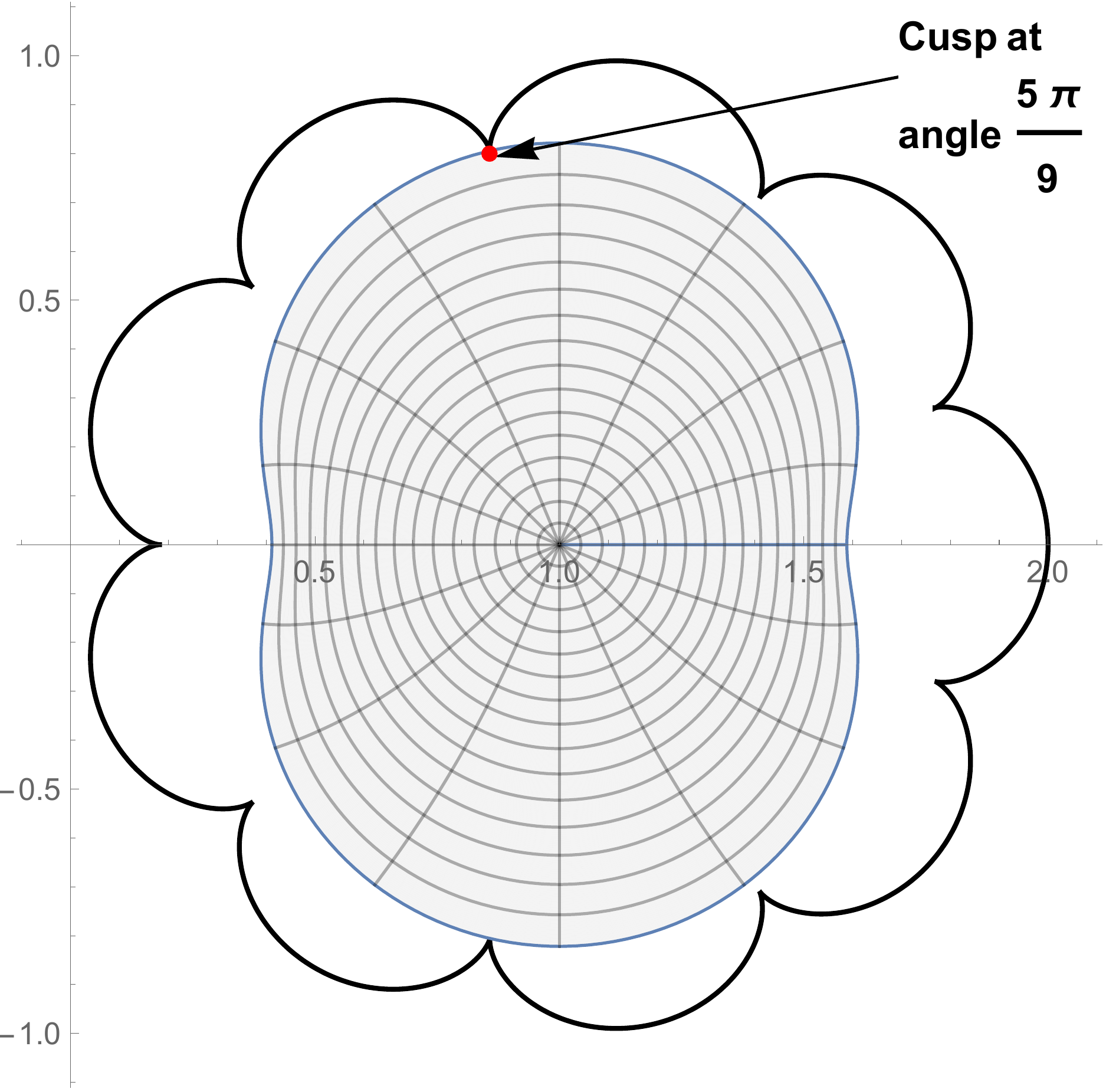}}\hspace{10pt}
					\subfigure[$n=12$]{\includegraphics[width=2in]{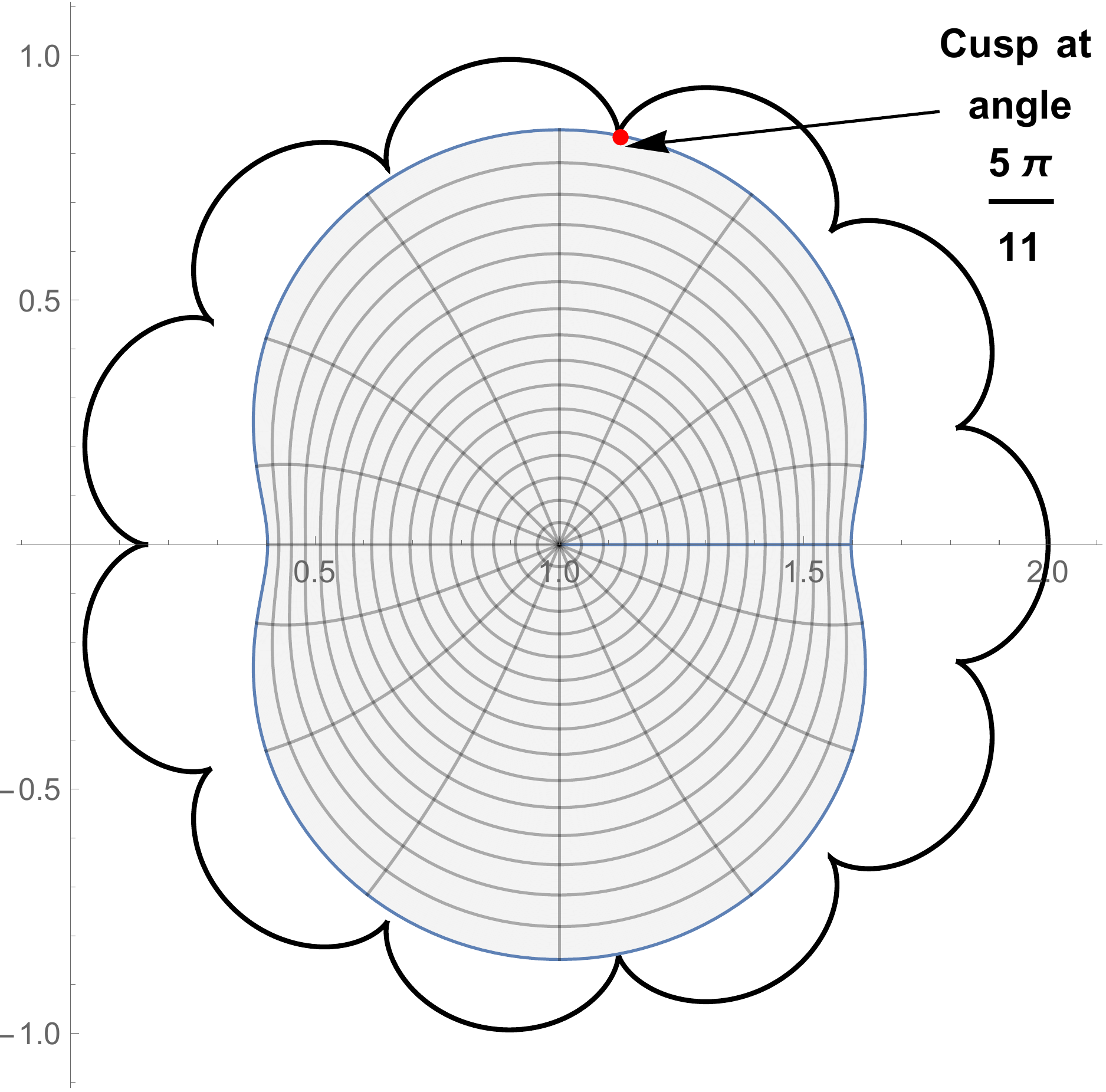}}\hspace{10pt}
				\caption{Nephroid domain lying in various polyleaf domain}\label{fig7}
		\end{center}
	\end{figure}
\end{proof}
The next theorem gives the $\Snl-$radius for some special Janowski classes. As proves earlier, this result is also obtained by considering the cusp at the angle $\pi/(n-1)$ and hence omitted here.
\begin{theorem}
	The $\Snl-$radius for some special Janowski classes is given by
	\begin{itemize}
		\item[(a)] $\mathcal{R}_{\Snl}(\mathcal{S}^*(\alpha))=\left|
		\displaystyle\frac{\gamma^n+n\gamma}{2(1-\alpha)+\gamma^n+2n(1-\alpha)+n\gamma}\right|$
		\item[(b)] $\mathcal{R}_{\Snl}(\mathcal{S}^*[\alpha,-\alpha])=\left|\displaystyle\frac{\gamma^n+n\gamma}{\alpha(2+\gamma^n+2n+n\gamma)}\right|,$ where $0<\alpha\leq1.$
		\item[(c)] $\mathcal{R}_{\Snl}(\mathcal{S}^*[1-\alpha,0])=\left|\displaystyle\frac{\gamma^n+n\gamma}{(n+1)(1-\alpha)}\right|$
		\item[(d)] $\mathcal{R}_{\Snl}(\mathcal{S}^*[1,-(M-1)/M])=\left|\displaystyle\frac{M(\gamma^n+n\gamma)}{-1+2M-\gamma^n+M\gamma^n-n+2Mn-\gamma n+Mn\gamma}\right|,$ for $M>1/2.$
	\end{itemize}
\end{theorem}
\begin{remark}
	For $\alpha=0,$ the above result gives the $\Snl-$radius for the class $\mathcal{S}^*$ of starlike function and it is given by $|(\gamma^n+n\gamma)/(2+\gamma^n+2n+n\gamma)|,$ where $\gamma=e^{i\pi/(n-1)}.$ By using Mark Strohhacker's theorem, it is known that $\mathcal{K}\subset\mathcal{S}^*(1/2).$ Thus, the $\Snl-$radius for the class $\mathcal{K}$ is atleast $|(\gamma^n+n\gamma)/(1+\gamma^n+n+n\gamma)|.$
\end{remark}
\begin{remark}
	The $\Snl-$radius for the classes $\mathcal{S}^*_{SG}$, $\mathcal{S}^*(\cos z)$ and $\mathcal{S}^*(\cosh z)$ is $1$ as these domains lie inside the domain $\varphi_{n\mathcal{L}}(\disc)$ (as depicted by Figure \ref{fig9}).
		\begin{figure}[h]
		\begin{center}
			\subfigure[$\mathcal{S}^*_{SG}$]{\includegraphics[width=1.5in]{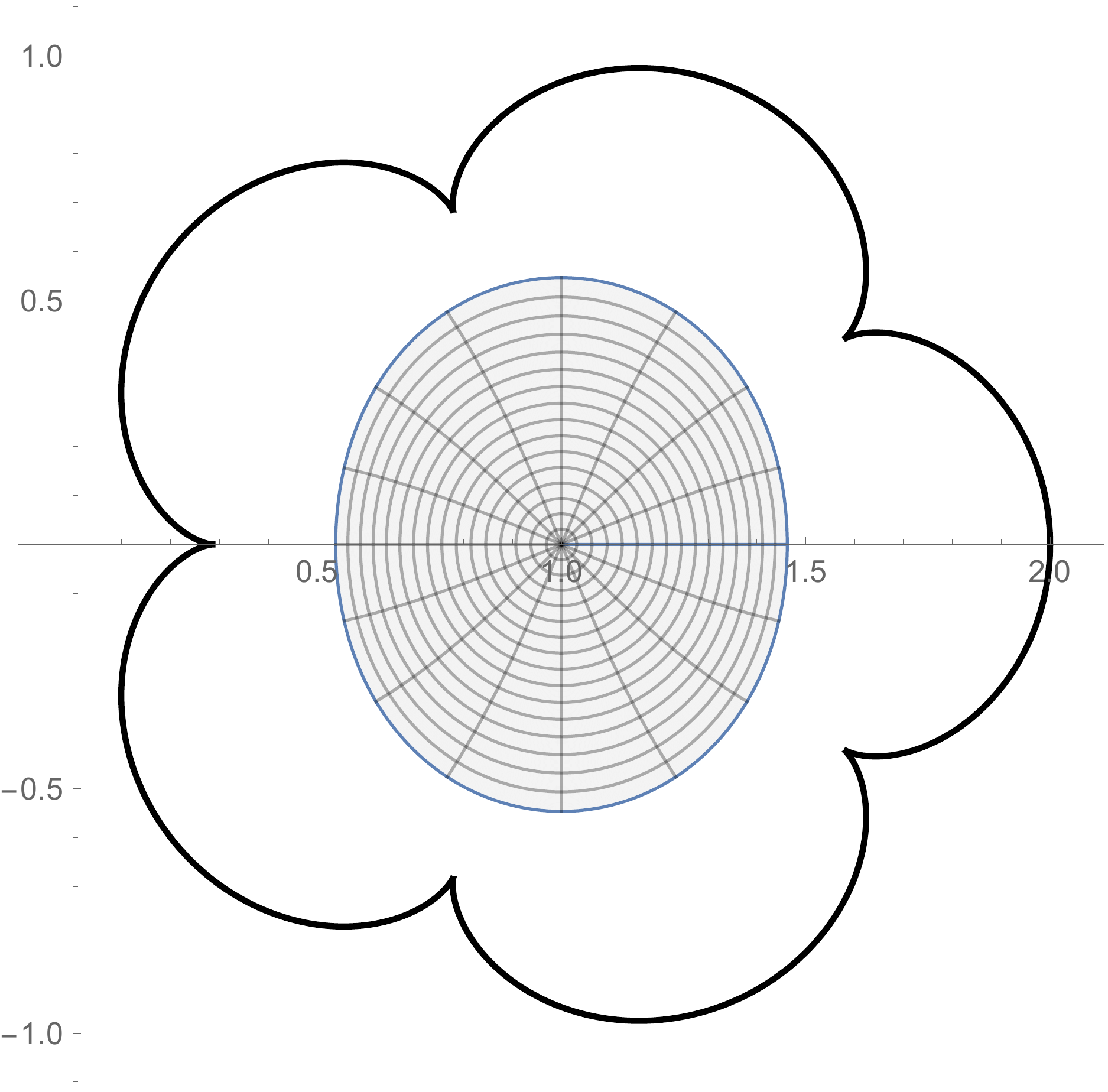}}\hspace{10pt}
			\subfigure[$\mathcal{S}^*(\cos z)$]{\includegraphics[width=1.5in]{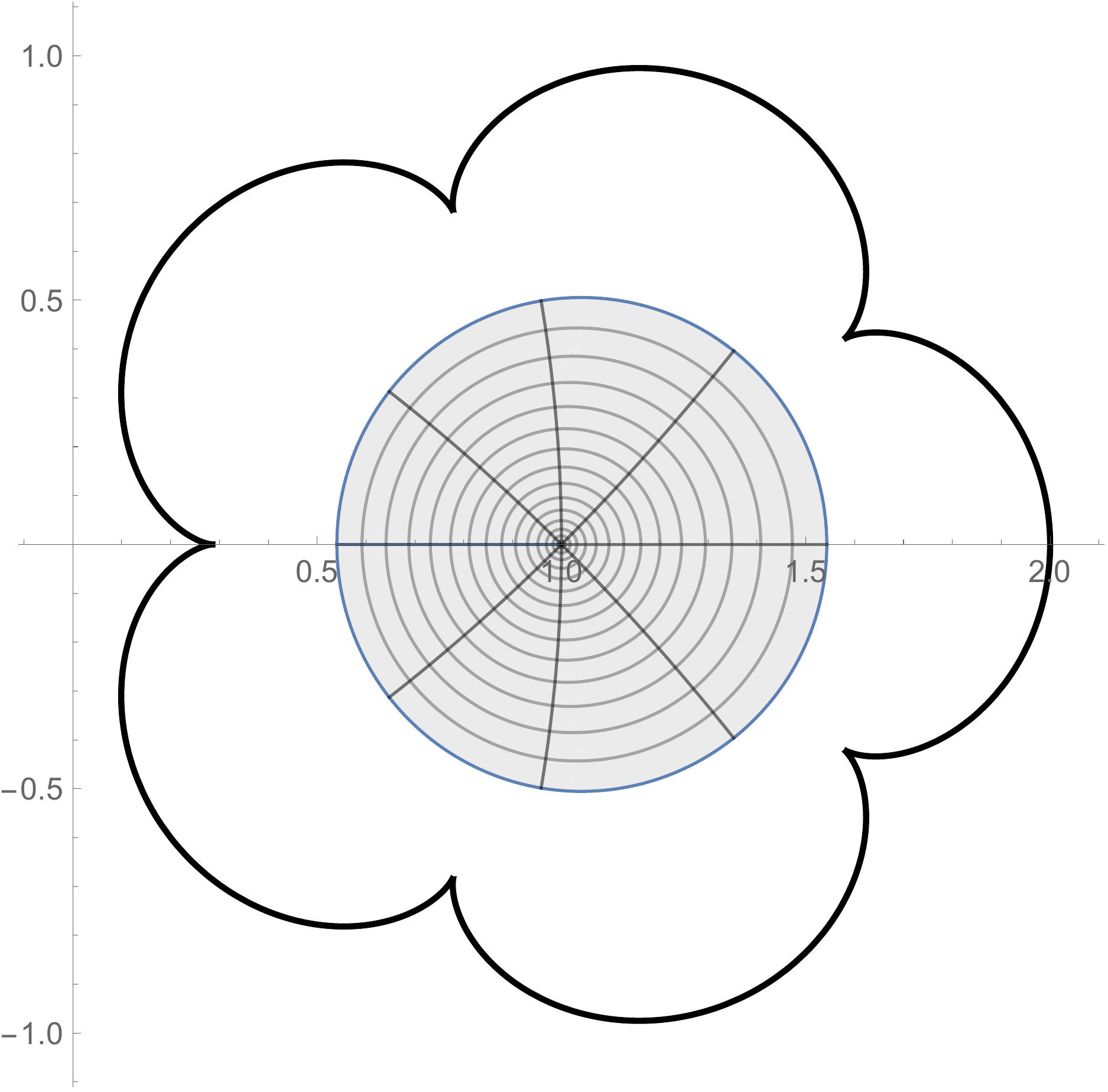}}\hspace{10pt}
			\subfigure[$\mathcal{S}^*(\cosh z)$]{\includegraphics[width=1.5in]{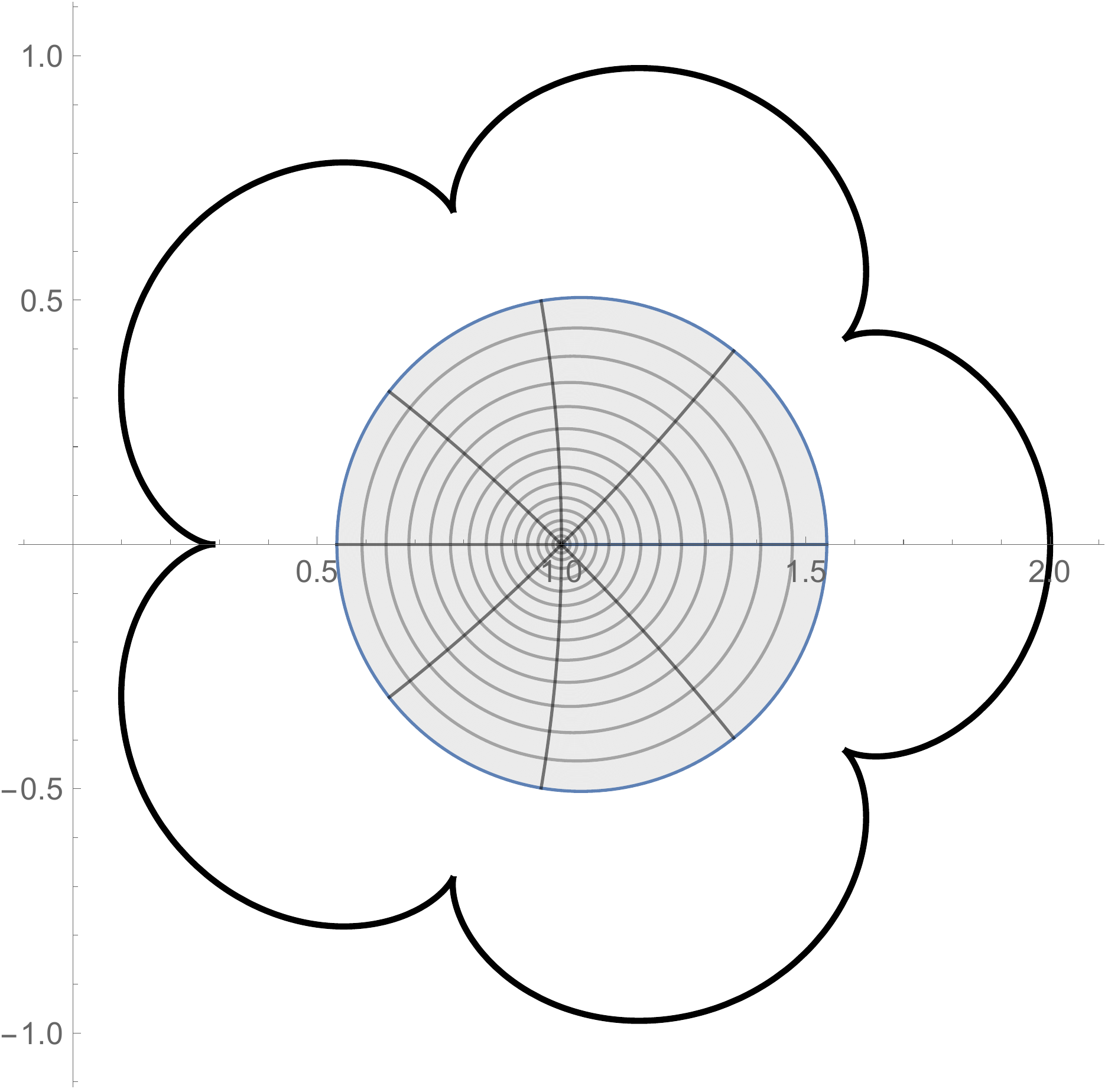}}\hspace{10pt}
			\caption{Domains lying inside $\varphi_{n\mathcal{L}}(\disc)$}\label{fig9}
		\end{center}
	\end{figure}
\end{remark}
\begin{remark}
	As mentioned earlier, the class $\Snl$ becomes the class $\mathcal{S}^*[1,0]$ for which the image domain is a disk with center $1$ and radius $1$ in the limiting case. Thus, $\mathcal{S}^*[1,0]-$radius for various classes can be obtained by taking the limit as $n\rightarrow\infty$ in the above proved results. The following table summarizes the $\mathcal{S}^*[1,0]-$radii.

\begin{table}
\begin{tabular}{ll}
	
	\begin{tabular}{QQA}
		\toprule
		\mbox{S.No.} & \mbox{Class} & \mbox{$n\rightarrow\infty$}\\\midrule
		(a)&\mathcal{W} & \sqrt{2}-1 \\\midrule
		(b)&\mathcal{F}_1 & \sqrt{5}-2 \\\midrule
		(c)&\mathcal{F}_2 & (\sqrt{17}-3)/4 \\\midrule
		(d)&\mathcal{S}^*_{RL} &  1 \\\midrule
		(e)&\mathcal{S}^*_{C} & \sqrt{5/2}-1 \\\midrule
		(f)&\mathcal{S}^*_{R} & -1-\sqrt{2}+\sqrt{6+4\sqrt{2}} \\\midrule
		(g)&\mathcal{S}^*_{\leftmoon} &3/4   \\\midrule
		(h)&\mathcal{S}^*_{lim} & 2-\sqrt{2}\\\midrule
		(i)&\mathcal{S}^*(1+ze^z) &  0.567143
		  \\\bottomrule
	\end{tabular}
&
\begin{tabular}{QQA}
	\toprule
	\mbox{S.No.} & \mbox{Class} & \mbox{$n\rightarrow\infty$}\\\midrule
	(a)&\mathcal{M}(\beta) & 1/(2\beta-1) \\\midrule
	(b)&\mathcal{B}\mathcal{S}(\alpha) & (1+\sqrt{1+4\alpha})/2\alpha\\\midrule
	(c)&\mathcal{S}\mathcal{L}^*(\alpha) & (2\alpha-1)/(\alpha-1)^2 \\\midrule
	(d)&\mathcal{S}^*_{\alpha,e} & \log((\alpha-2)/(\alpha-1)) \\\midrule
	(e)&\mathcal{S}^*(\alpha) & 1/(3-2\alpha) \\\midrule
	(f)&\mathcal{S}^*[1-\alpha,0] & 1/(\alpha-1) \\\midrule
	(g)&\mathcal{S}^*[\alpha,-\alpha] & 1/(3\alpha)  \\\midrule
	(h)&\mathcal{S}^*_M & M/(3M-2)
	\\\bottomrule
\end{tabular}
\end{tabular}
\hspace{10pt}\caption{Radii for the Limiting case}\label{tab3}	
\end{table}
\end{remark}
\section{Radius Constants for class $\Snl$}
\begin{theorem}
	The sharp radii constants for the class $\mathcal{S}^*_{n\mathcal{L}}$ as follows
	\begin{itemize}
		\item [(a)] The $\mathcal{S}\mathcal{L}^*(\alpha)-$radius is the smallest positive real root of the equation $r^n+rn-(\sqrt{2}-1)(1-\alpha)(n+1)=0,$ for $0\leq \alpha<1$.
		\item[(b)] The $\mathcal{S}^*_{RL}-$radius is the smallest positive real root of the equation $r^n+ rn-(n+1)\left(\sqrt{\gamma}-\gamma\right)^{1/2}=0$, where $\gamma=2\sqrt{2}-2.$
		\item[(c)] The $\mathcal{S}^*_R-$radius is the smallest positive real root of the equation $r^n-rn-(n+1)(2\sqrt{2}+3)=0$.
		\item[(d)] The $\mathcal{S}^*_{sin}-$radius is the smallest positive real root of the equation $r^n+rn-(n+1)\sin1=0$.
		\item[(e)] The $\mathcal{S}^*_{SG}-$radius is the smallest positive real root of the equation $r^n+rn-(n+1)(e-1)/(e+1)=0.$
		\item[(f)] The $\mathcal{S}^*_{ne}-$radius is the smallest positive real root of the equation $r^n+rn-2(n+1)/3=0.$
		\item[(g)] The $\mathcal{S}^*(1+ze^z)-$radius is the smallest positive real root of the equation $r^n-rn+(n+1)/e=0.$
		\item[(h)] The $\mathcal{S}^*(1+\sinh^{-1}(z))-$radius is the smallest positive real root of the equation $r^n+rn-(n+1)\sinh^{-1}(1)=0.$
		\item[(i)] The $\mathcal{M}(\beta)-$radius is the smallest positive real root of the equation $r^n+rn-(n+1)(\beta-1)=0,$ for $1<\beta\leq 2$ and the radius is $1,$ for $\beta\geq 2.$
		\item[(j)]  The $\mathcal{S}^*[1-\alpha,0]-$radius is the smallest positive real root of the equation $r^n+rn-(n+1)(1-\alpha)=0.$
	\end{itemize}
\end{theorem}
\begin{proof}
	Let $f\in\Snl.$ Then $zf'(z)/f(z)\prec\varphi_{n\mathcal{L}},$ where $\varphi_{n\mathcal{L}}$ is given by (\ref{eqn2}). For $|z|=re^{it},$
	\begin{align}	\label{eqn3}
	\left|\frac{zf'(z)}{f(z)}-1\right|\leq \frac{nr}{n+1}+\frac{r^n}{n+1}.
	\end{align}
	(a) By using \cite[Lemma 2.3, pp 6]{KHATTER}, it can be obtained that the disk (\ref{eqn3}) lies inside the lemniscate of Bernoulli $|((w-\alpha)/(1-\alpha))^2-1|=1$ if
	\[\frac{nr}{n+1}+\frac{r^n}{n+1}\leq (\sqrt{2}-1)(1-\alpha).\]
This gives $r\leq s_1,$ where $s_1$ is the smallest positive real root of the equation $r^n+rn-(\sqrt{2}-1)(1-\alpha)(n+1)=0,$ for $0\leq \alpha<1$. For sharpness, consider the function $f_{n\mathcal{L}}(z)$ given by (\ref{eqn4}). The value of $zf_{n\mathcal{L}}'(z)/f_{n\mathcal{L}}(z)$ is $\sqrt{2},$ for $z=s_1.$
	\par (b)  The disk (\ref{eqn3}) lies in the left-hand side of reverse lemniscate of Bernoulli $|(w-\sqrt{2})^2-1|=1$ if
\[\frac{nr}{n+1}+\frac{r^n}{n+1}\leq\sqrt{\sqrt{2\sqrt{2}-2}-2\sqrt{2}+2},\]
by  \cite[Lemma 3.2, pp 10]{MEND}. This simplifies to $r\leq s_2,$ where $s_2$ is the smallest positive real root of the equation $r^n+ rn-(n+1)\left(\sqrt{\gamma}-\gamma\right)^{1/2}=0$, where $\gamma=2\sqrt{2}-2.$ The result is sharp for the function $f_{n\mathcal{L}}$ given by (\ref{eqn4}).
\par (c) The subordination $\varphi_{n\mathcal{L}}(z)\prec\varphi_R(z)$ holds for $\disc_r$ if
\[2(\sqrt{2}-1)\leq \varphi_R(-1)\leq \varphi_{n\mathcal{L}}(-r)=1-\frac{n r}{n+1}+\frac{r^n}{n+1},\]
for $n$ is even. This gives $r\leq s_3,$ where $s_3$ is the smallest positive real root of the equation $r^n-rn-(n+1)(2\sqrt{2}+3)=0$.The bound is best possible for the function $f_{n\mathcal{L}}$ given by (\ref{eqn4}). For $z=s_3,$ the quantity $zf_{n\mathcal{L}}'(z)/f_{n\mathcal{L}}(z)=2(\sqrt{2}-1).$
\par (d) Similarly, the disk (\ref{eqn3}) lies in the image domain of $\varphi_{sin}(\disc)$ if
\[1+\frac{nr}{n+1}+\frac{r^n}{n+1}\leq\varphi_{n\mathcal{L}}(r)\leq\varphi_{sin}(1)=1+\sin1.\]
This holds for $r\leq s_4,$ where $s_4$ is the smallest positive real root of the equation $r^n+rn-(n+1)\sin1=0$. The result is best possible for the function $f_{n\mathcal{L}}(z)$ given by (\ref{eqn4}) and $zf_{n\mathcal{L}}'(z)/f_{n\mathcal{L}}(z)=1+\sin1,$ for $z=s_4.$
\par (e) By \cite[Lemma 2.2, pp 5]{GOEL}, the disk (\ref{eqn3}) lies in the modified sigmoid $|\log(w/(2-w))|=1$ if
\[\frac{nr}{n+1}+\frac{r^n}{n+1}\leq \frac{e-1}{e+1}.\]
This simplifies to $r\leq s_5,$ where $s_5$ is the smallest positive real root of the equation $r^n+rn-(n+1)(e-1)/(e+1)=0.$ The bound cannot be improved further as for $z=s_5,$ $zf_{n\mathcal{L}}'(z)/f_{n\mathcal{L}}(z)$ assumes value $2e/(e+1),$ where $f_{n\mathcal{L}}(z)$ is given by (\ref{eqn4}).
\par (f) \cite[Lemma 2.2, pp 8]{WANI2} gives the following condition for the disk (\ref{eqn3})  to lie inside the nephroid
\[\frac{nr}{n+1}+\frac{r^n}{n+1}\leq \frac{2}{3}.\]
This gives $r\leq s_6,$ where $s_6$ is the smallest positive real root of the equation $r^n+rn-2(n+1)/3=0.$ For sharpness, consider the function $f_{n\mathcal{L}}(z)$ given by (\ref{eqn4}). For $z=s_6,$ the value of $zf_{n\mathcal{L}}'(z)/f_{n\mathcal{L}}(z)$ is $5/3.$
\par (g) For $|z|<r,$ a necessary condition for the subordination $\varphi_{n\mathcal{L}}(z)\prec1+ze^z$ to hold is
\[1-\frac{1}{e}\leq \varphi_{n\mathcal{L}}(-r)=1-\frac{nr}{n+1}+\frac{r^n}{n+1}.\]
This simplifies to $r\leq s_7,$ where $s_7$ is the smallest positive real root of the equation $r^n-rn+(n+1)/e=0.$ The result is sharp for the function $f_{n\mathcal{L}}(z)$ given by (\ref{eqn4}) and $zf_{n\mathcal{L}}'(z)/f_{n\mathcal{L}}(z)=1-1/e,$ for $z=-s_7.$
\par (h) By using \cite[Lemma 2.1, pp 4]{SIVA1}, we get the disk (\ref{eqn3}) lie inside the image domain of the function $1+\sinh^{-1}(z)$ if
\[\frac{r^n}{n+1}+\frac{nr}{n+1}\leq \sinh^{-1}(1),\]
which simplifies to $r\leq s_8,$ where $s_8$ is the smallest positive real root of the equation $r^n+rn-(n+1)\sinh^{-1}(1)=0.$ The bounds are sharp for the function $f_{n\mathcal{L}}(z)$ given by (\ref{eqn4}). For $z=s_8,$ $zf_{n\mathcal{L}}'(z)/f_{n\mathcal{L}}(z)=1+\sinh^{-1}(1).$
\par (i)  As seen earlier, $\Snl\subset\mathcal{M}(\beta)$ for $\beta>2.$ Let us now assume that $1<\beta\leq 2.$ For $|z|=r<1,$
\[\RE\left(\frac{zf'(z)}{f(z)}\right)<1+\frac{nr}{n+1}+\frac{r^n}{n+1}<\beta,\]
provided $r<s_9,$ where $s_9$ is the smallest positive real root of the equation $r^n+rn-(n+1)(\beta-1)=0$. For the function $f_{n\mathcal{L}}$, the quantity $zf_{n\mathcal{L}}'(z)/f_{n\mathcal{L}}(z)=\beta$ at $z=s_9.$
\par (j) The disk (\ref{eqn3}) lies in the domain $|w-1|<1-\alpha$ if
\[\frac{r^n}{n+1}+\frac{nr}{n+1}\leq 1-\alpha,\]
which gives $r\leq s_{10},$ where $s_{10}$ is the smallest positive real root of the equation $r^n+rn-(n+1)(1-\alpha)=0.$ The result is sharp for the function $f_{n\mathcal{L}}(z)$ given by (\ref{eqn4}) and for $z=s_{10},$ $zf_{n\mathcal{L}}'(z)/f_{n\mathcal{L}}(z)=2-\alpha.$
\par For some choices of $n$, the radii are computed is tabulated in Table \ref{TAB2} and the sharpness for these results is illustrated by Figure \ref{fig8} for $n=8.$
\end{proof}
\begin{figure}
	\begin{center}
		\subfigure[$\mathcal{S}^*_{L}$]{\includegraphics[width=1.5in]{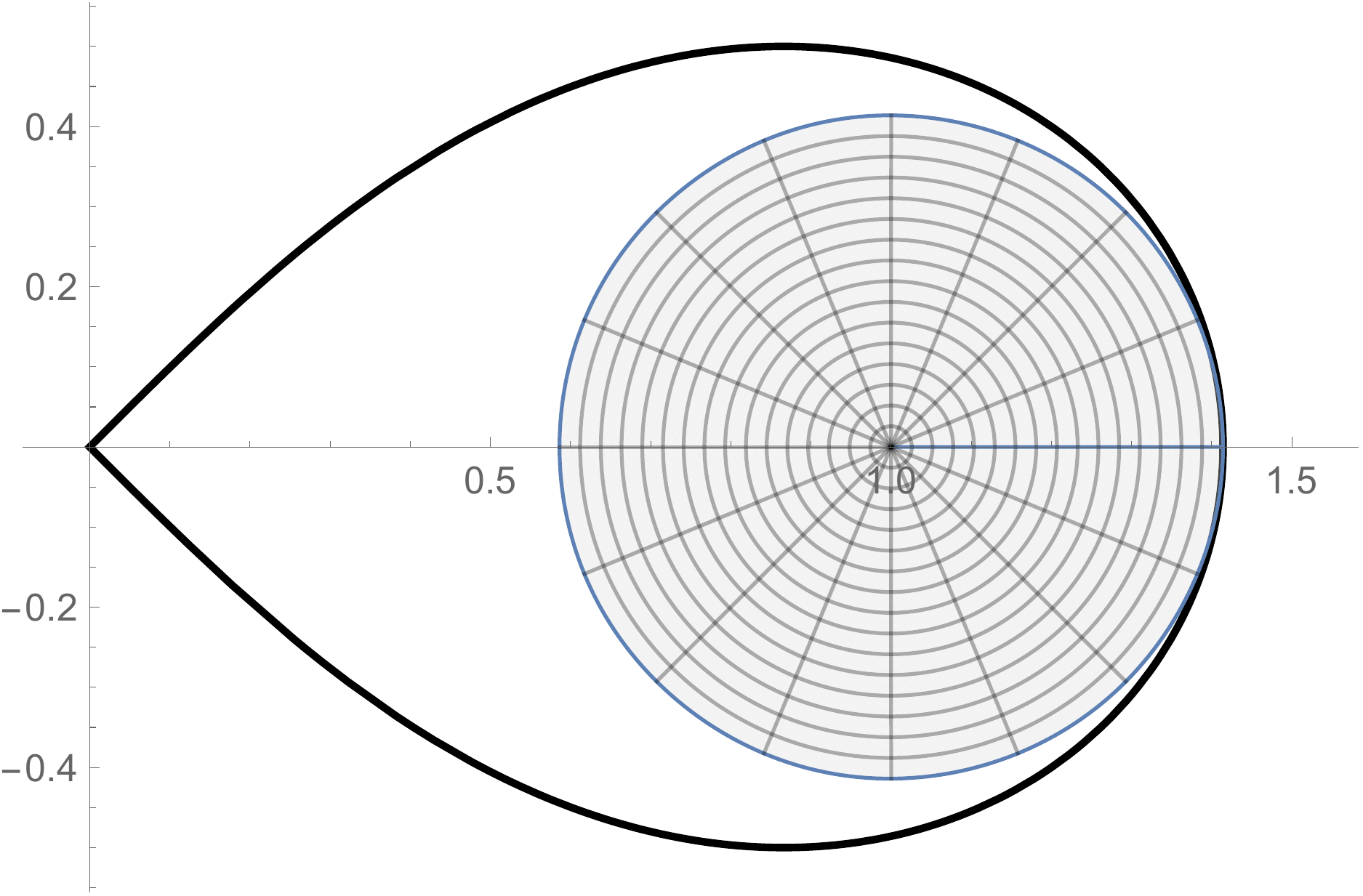}}\hspace{10pt}
		\subfigure[$\mathcal{S}^*_{RL}$]{\includegraphics[width=1.5in]{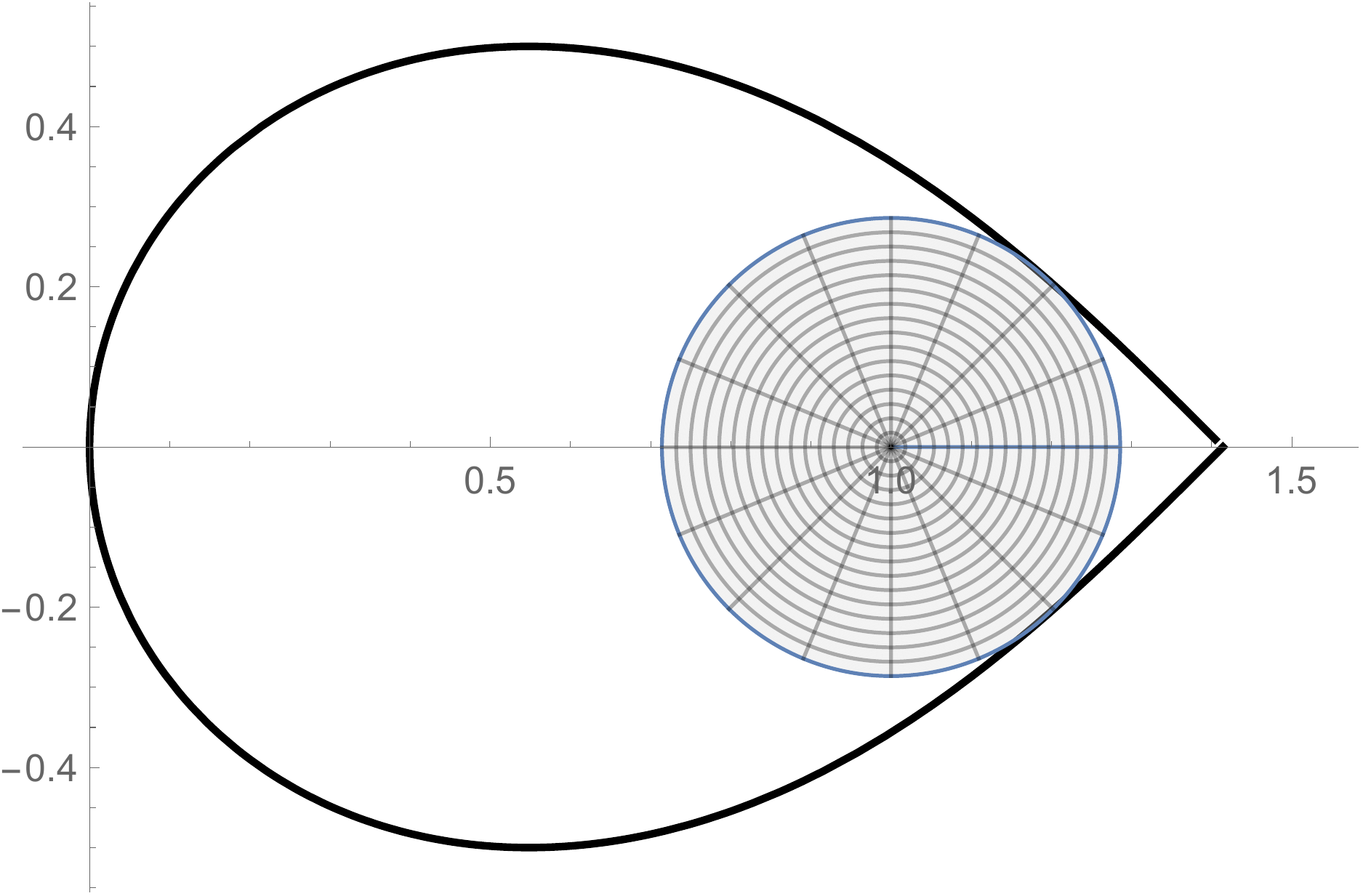}}\hspace{10pt}
		\subfigure[$\mathcal{S}^*_R$]{\includegraphics[width=1.5in]{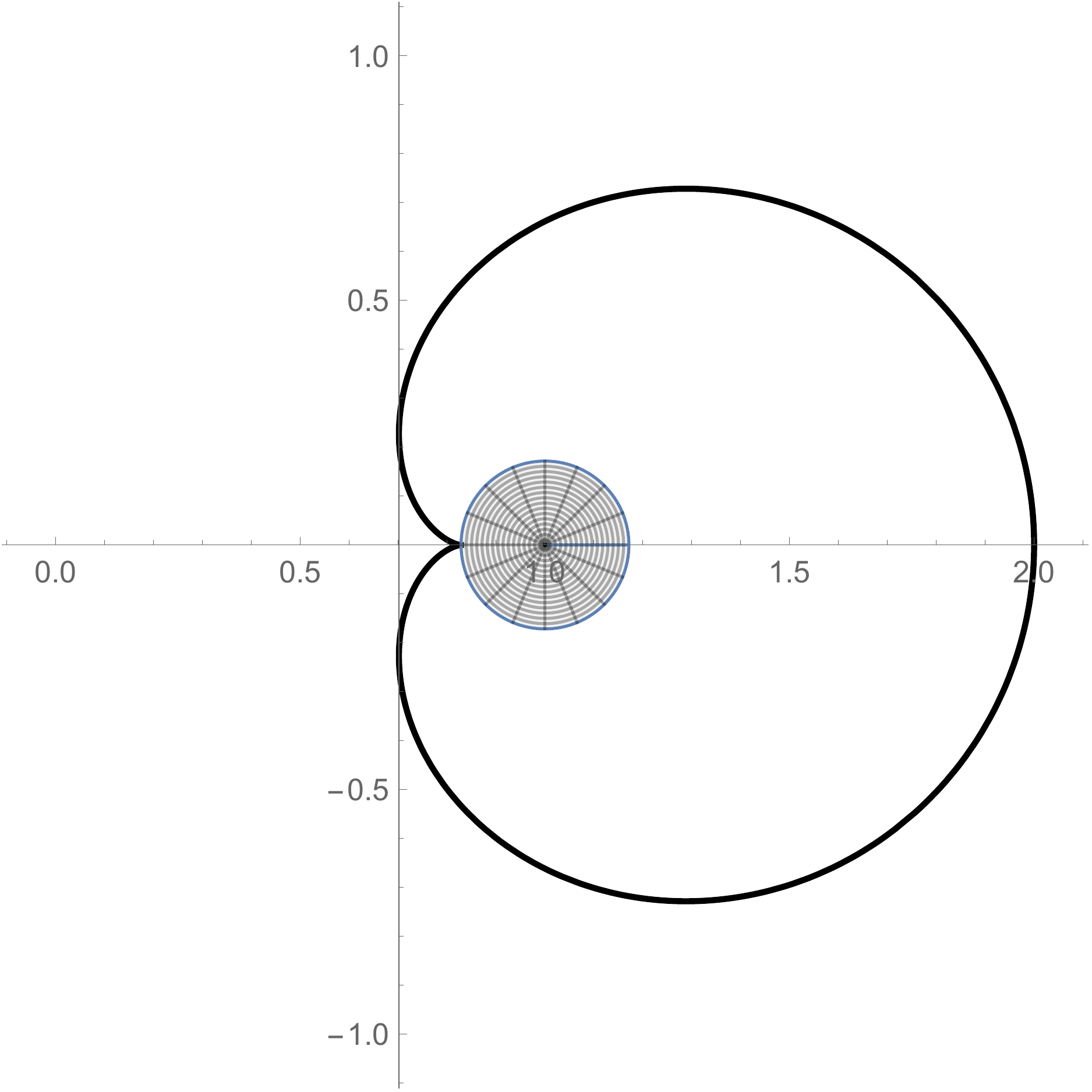}}\hspace{10pt}
		\subfigure[$\mathcal{S}^*_{sin}$]{\includegraphics[width=1.5in]{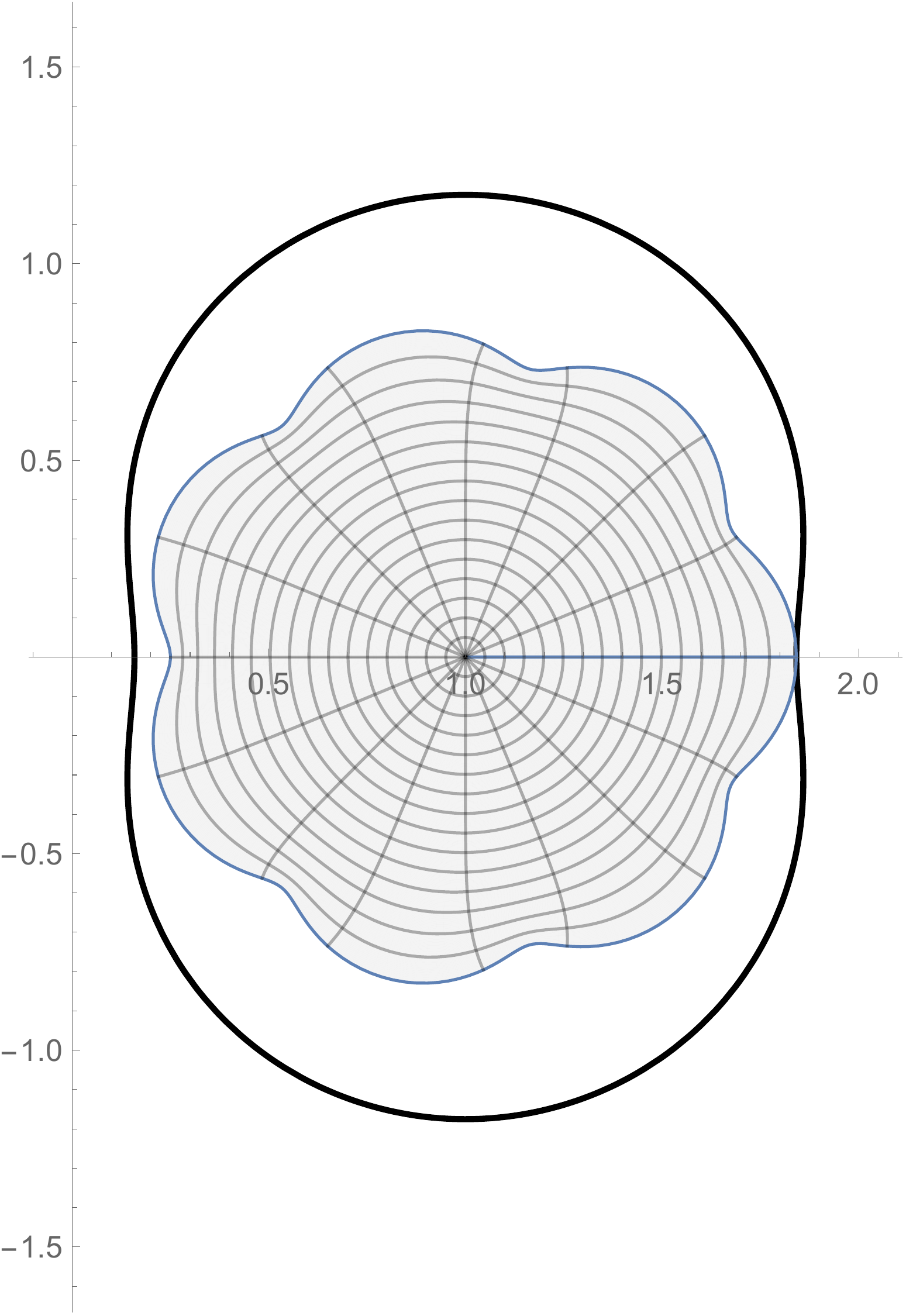}}\hspace{10pt}
		\subfigure[$\mathcal{S}^*_{ne}$]{\includegraphics[width=1.5in]{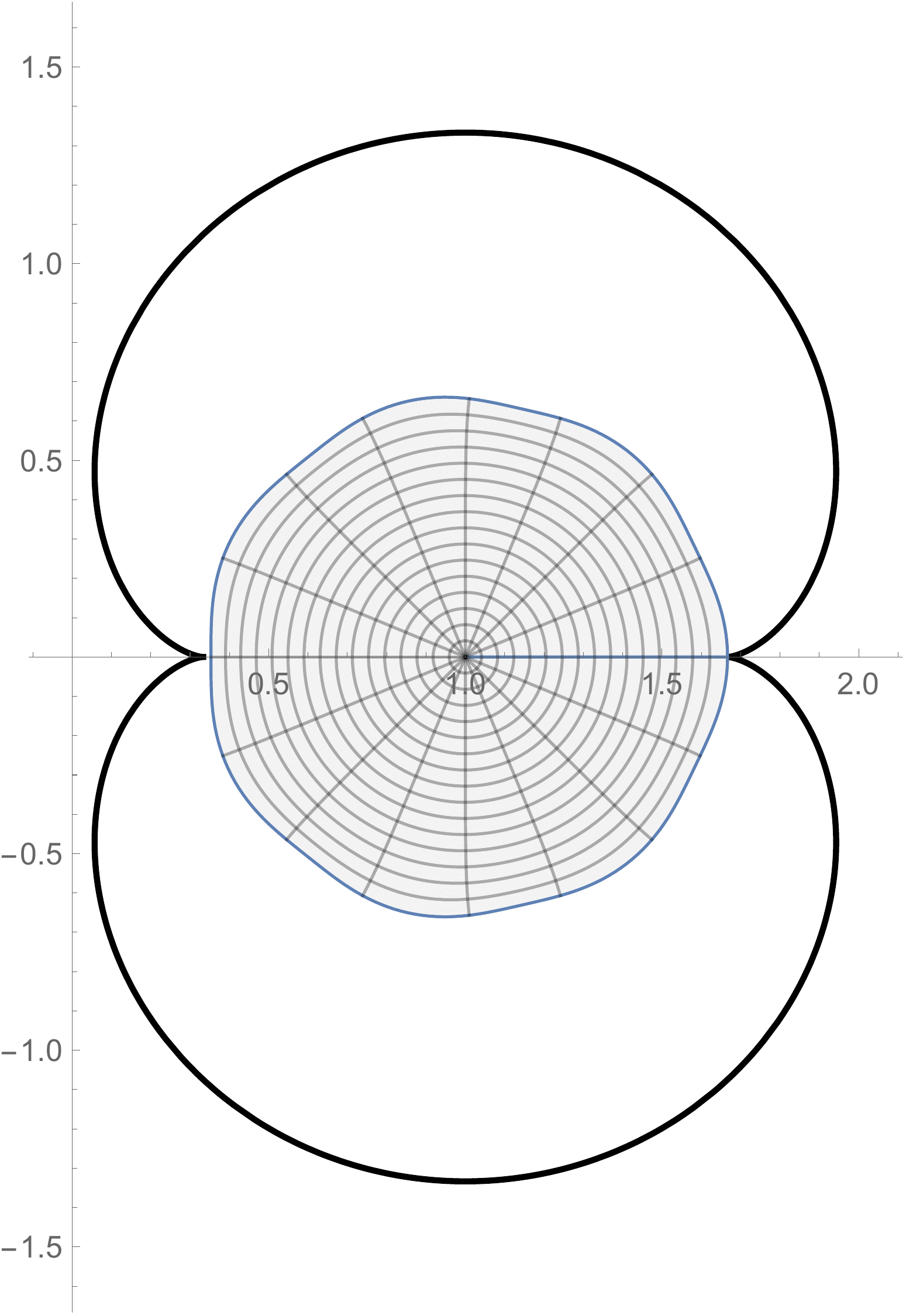}}\hspace{10pt}
	\subfigure[$\mathcal{S}^*(1+\sinh^{-1}(z))$]{\includegraphics[width=1.5in]{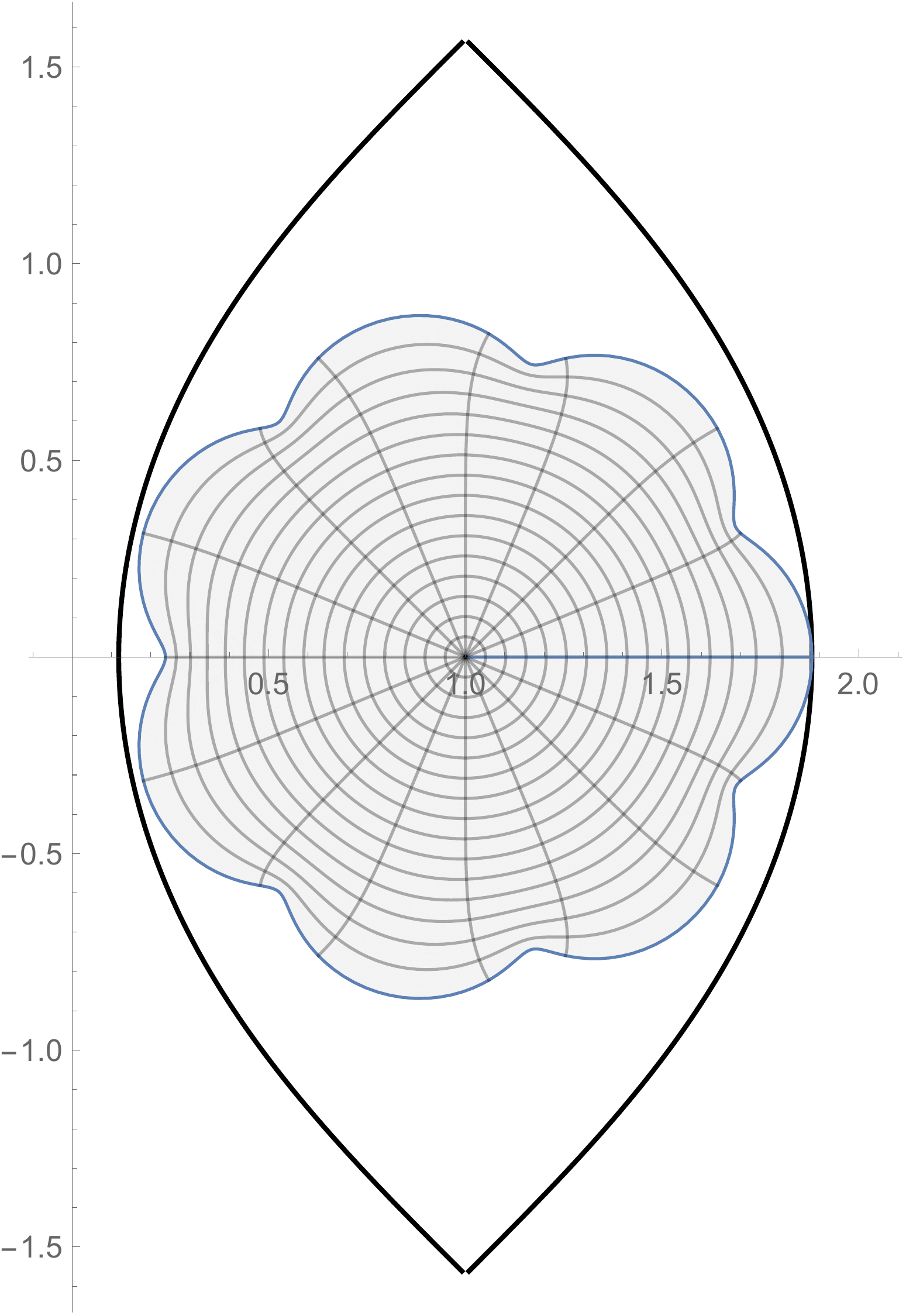}}\hspace{10pt}
		\subfigure[$\mathcal{S}^*(1+ze^z)$]{\includegraphics[width=1.5in]{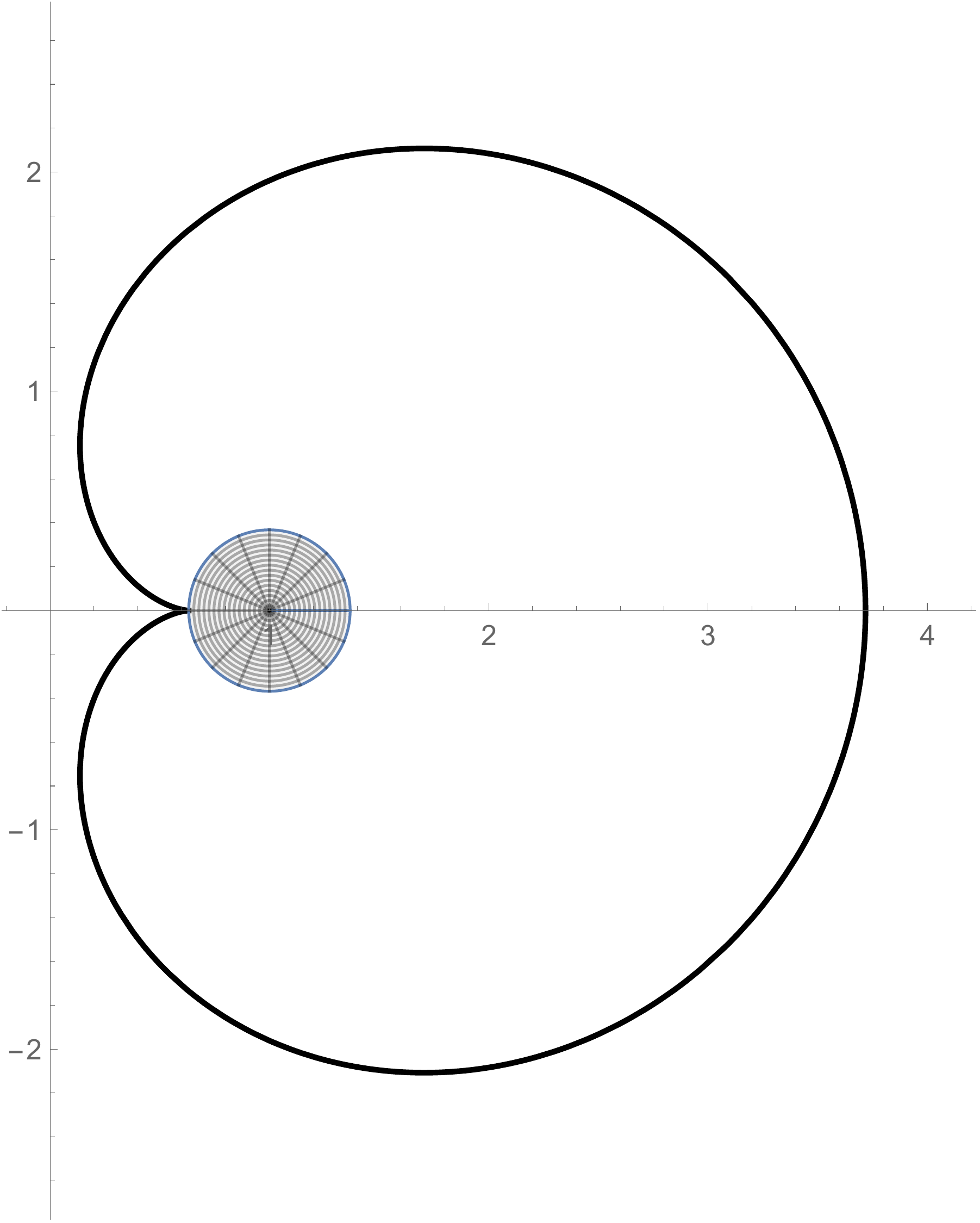}}\hspace{10pt}
		\subfigure[$\mathcal{S}^*_{SG}$]{\includegraphics[width=1.5in]{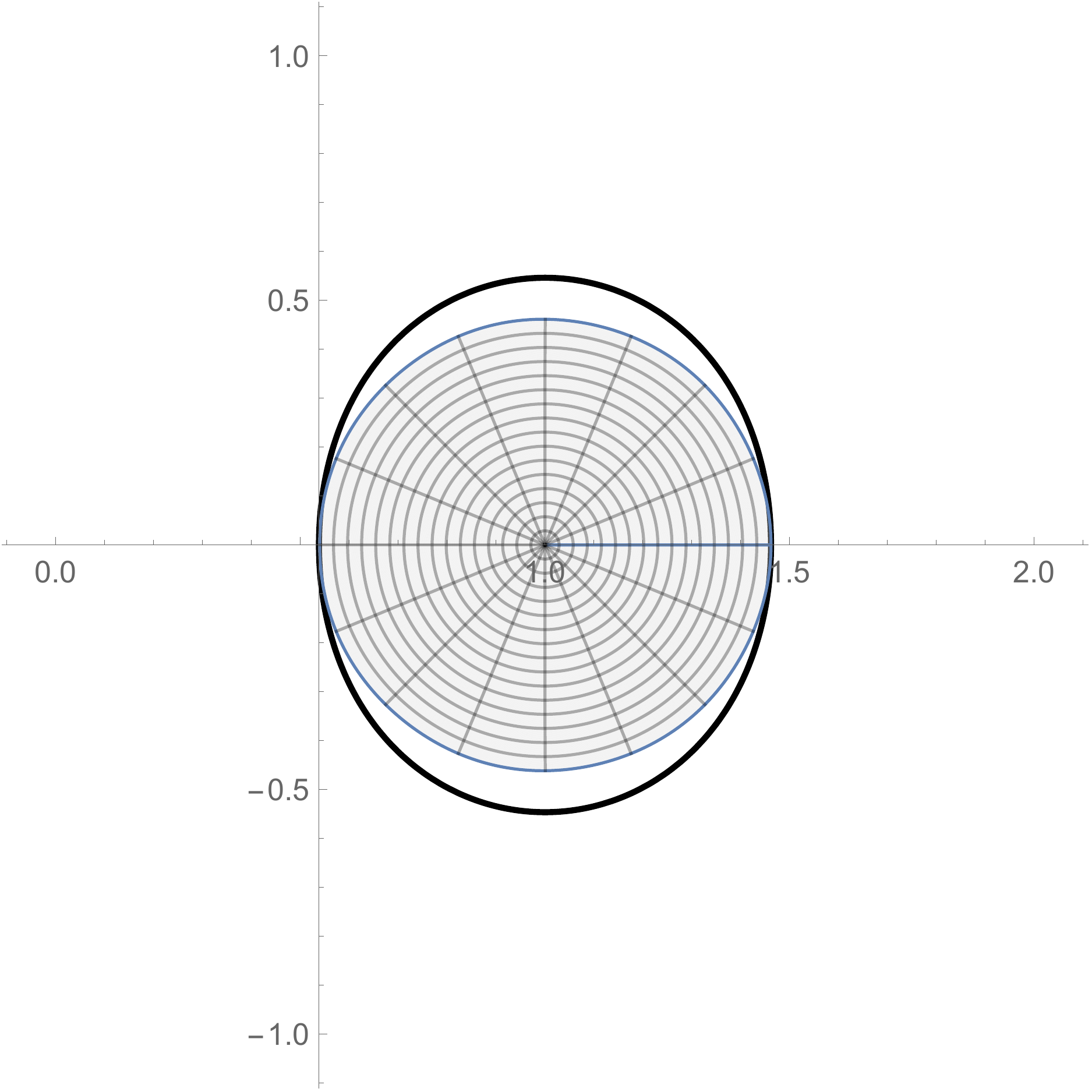}}\hspace{10pt}
			\caption{Sharpness of various radii for class $\Snl$}\label{fig8}
	\end{center}
\end{figure}
\begin{table}
\begin{center}
	\begin{tabular}{QQAAA}
		\toprule
		\mbox{S.No.} & \mbox{Class} & \mbox{n=4} & \mbox{n=6} & \mbox{n=8}\\\midrule
		(a)&\mathcal{S}\mathcal{L}^* & 0.501903 &0.48118 & 0.465714 \\\midrule
		(b)&\mathcal{S}^*_{RL} & 0.353501 & 0.333349 & 0.32165  \\\midrule
		(c)&\mathcal{S}^*_R & 0.213942 & 0.200158 & 0.193019  \\\midrule
		(d)&\mathcal{S}^*_{sin} & 0.892917 & 0.895669 & 0.895131  \\\midrule
		(e)&\mathcal{S}^*_{SG} & 0.554083 & 0.535219 & 0.519222 \\\midrule
		(f)&\mathcal{S}^*_{ne} & 0.752971 & 0.748475 & 0.738894 \\\midrule
		(g)&\mathcal{S}^*(1+ze^z) & 0.472288 & 0.43025 & 0.413972  \\\midrule
		(h)&\mathcal{S}^*(1+\sinh^{-1}(z))& 0.921471 & 0.924325 & 0.924715   \\\bottomrule
	\end{tabular}
\end{center}
\caption{Radii constants for choices of $n$}\label{TAB2}
\end{table}
\section*{Acknowledgements}
The second author is supported by a Junior Research Fellowship from Council of Scientific and Industrial Research (CSIR), New Delhi with File No. 09/045(1727)/2019-EMR-I.

\end{document}